\documentclass[11pt]{article}
\usepackage[margin=2.5cm]{geometry}
\usepackage{color}
\usepackage{amsmath,amssymb,amsthm}
\usepackage{fancyhdr}
\usepackage{multicol}
\usepackage{psfrag}
\usepackage{array}
\usepackage{graphicx}
\usepackage[normalem]{ulem}
\usepackage{cancel}
\usepackage{hyperref}
\usepackage{bbold}
\usepackage{enumitem}
\newtheorem{lemma}{Lemma}
\newtheorem{theorem}{Theorem}
\newtheorem{proposition}{Proposition}

\newtheorem{definition}{Definition}
\newtheorem{corollary}{Corollary}

\newtheorem{remark}{Remark}

\newtheorem{example}{Example}
\def\f{\frac}

\def\dst{\displaystyle}
\def\eps{\varepsilon}

\def\p{\partial}
\def\a{\alpha}

\def\R{\mathbb{R}}
\def\C{\mathbb{C}}
\def\M{\mathcal{M}}

\def\N{\mathbb{N}}

\def\supp{{\mbox{\small\rm  supp}}\,}
\date{}

\definecolor{mygreen}{rgb}{0.2235,0.6353,0.2588}
\definecolor{myorange}{rgb}{0.9568,0.4941,0.1961}
\definecolor{myred}{rgb}{0.9098,0.1294,0.2078}
\definecolor{myblue}{rgb}{0.0352,0.4981,0.6509}
\definecolor{myrose}{rgb}{0.9098,0.1294,0.7078}
\newcommand{\comMD}[1]{\textcolor{black}{#1}}
\newcommand{\Mig}[1]{{\textcolor{black}{#1}}}

\newcommand{\Mag}[1]{\textcolor{black}{#1}}
\newcommand{\MT}[1]{\textcolor{black}{#1}}

\newcommand{\inserted}[1]{{\color{black}#1}}

\graphicspath{{figs/}}

\definecolor{vert}{rgb}{0,0,0}

\title{An inverse problem: recovering the fragmentation kernel from the short-time 
behaviour of the fragmentation equation}

\date{}

\author{Marie Doumic
 \thanks{
Sorbonne Université, Inria, CNRS, Université de Paris, Laboratoire Jacques-Louis Lions, 4 place Jussieu, 75005 Paris, France. Email: marie.doumic@inria.fr} 
\and Miguel Escobedo 
\thanks{
Departamento de Matemáticas
Universidad del País Vasco (UPV/EHU)
Apartado 644, Bilbao 48080
Spain. Email: miguel.escobedo@ehu.es}
\and Magali Tournus 
\thanks{Aix Marseille Univ, CNRS, I2M, Marseille, Centrale Marseille, France.
 Email: magali.tournus@math.cnrs.fr}
}

\begin{document}

\maketitle
\begin{abstract}
\textcolor{black}{
{\color{vert}{G}}iven a phenomenon described by a {\color{vert} self-similar} fragmentation equation, how to 
{\color{vert}infer} the fragmentation kernel from experimental measurements {\color{vert}of the solution ? To answer this question at the basis of our work, a formal asymptotic \MT{expansion} suggested us that using short-time observations and initial data close to a Dirac measure should be a well-adapted strategy}. {\color{vert} As a necessary preliminary step, we study the direct problem, {\it i.e.} we prove} existence, uniqueness and stability with respect to the initial data of non negative  measure-valued solutions when 
the initial data 
{\color{vert} is a} compactly supported, bounded, non negative measure. A representation of the solution as a power series in the space of Radon measures 
is also shown. This representation is used to 
{\color{vert} propose a reconstruction formula for }
the fragmentation kernel, using short-time experimental measurements when the initial data is close to a Dirac measure.
{\color{vert} We prove error estimates}
in Total Variation and Bounded Lipshitz norms; 
{\color{vert} this gives a quantitative meaning to what a "short" time observation is}. For general initial data in the space of compactly supported measures, {\color{vert} we provide estimates on} how the short-time measurements approximate the convolution of the fragmentation kernel with a suitably-scaled version of the initial data. The series representation also yields a reconstruction formula for the Mellin transform of the fragmentation kernel $\kappa$  and an error estimate for such an approximation. 
{\color{vert} Our analysis is complemented by a numerical investigation.}
}
\end{abstract}

\section{Introduction}
\label{sec:intro}
The fragmentation equation is a size-structured PDE describing the evolution of 
a population of particles. It is ubiquitous in modelling 
physical or biological phenomena (cell division \cite{Perthame2007}, amyloid fibril breakage \cite{XR13}, microtubules dynamics \cite{HHTW18}) and technological processes (mineral processing, grinding solids \cite{Kolmogorov40}\comMD{,}
 polymer degradation \cite{Montroll_Simha40} and break-up of liquid droplets or air bubbles). 
 \textcolor{black}{As presented in \cite{Melzak}, the equation may be written as follows
 \begin{equation}\label{eq:frag_intro_gen}
\dfrac{\p}{\p t}u(t,x) =- \frac {u(t, x)} {x}\int _0^x y F(x, y)dy +  \dst\int_{x}^{\infty} u(t,y)F(y, x) dy ,
 \end{equation}
{\color{vert} where $u(t,x)$ represents the concentration of particles at time $t$ of size $x,$ and the fragmentation measure $F(y,x)$ the creation of particles of size $x$ out of fragmenting particles of size $y$.}  The mathematical properties of the fragmentation equation have been extensively studied  using a great variety of methods (statistical physics; formal asymptotics; real, complex and  functional analysis; linear semigroup theory; probability methods). Only a few references are given here  among the vast existing mathematical  literature as:   on particular solutions \cite{Montroll_Simha40, ZiffMcGrady85},  on the existence and uniqueness of  solutions for the Cauchy problem \cite{Stewart90, Melzak, BallCarr, DubovskiStewart, MLM, Banasiak1}, on detailed properties of the solutions \cite{PerthameRyzhik, BCG13, CCM, Bertoin1,   Haas2008}. For a rather complete list of references the interested reader may consult \cite{Banasiak1, Banasiak2,  zora77208}.}

\textcolor{black}{Due to  its importance in the modelling (see for example \cite{Val, Ram74}), and its  very rich mathematical properties,  a  \MT{class of} fragmentation \MT{measures} $F$ ha\MT{ve} proved to be particularly fruitful: \MT{the measures composed with} a fragmentation rate $B(x)$, that depends on the particle size, and a fragmentation kernel $\kappa (y/x)$, that describes the probability that a particle of size $y$ is created by fragmentation of a particle of size $x$:
 \begin{equation}\label{eq:frag_intro_simF}
 F(x, y)= \frac {B(x)} {x}\kappa\left(\frac {y} {x} \right).
 \end{equation}
 \MT{The fact that the probability to obtain a particle of size $y$ out of a particle of size $x$ only depends on the ratio $y/x$ is a classical assumption often referred to as a 'self-similarity property'~\cite{Bertoin1,BW,CCM}.}
 In order to be coherent with the modelling, the fragmentation kernel must  be  a finite measure compactly supported on $(0, 1)$ \MT{OR [0,1]?} and such that $zd\kappa(z)$ is a probability measure.
 With such a fragmentation measure,  equation (\ref{eq:frag_intro_gen}) reads then
}
\begin{equation}\label{eq:frag_intro}
\dfrac{\p}{\p t}u(t,x) =- B(x)u(t,x) +  \dst\int_{x}^{\infty}\kappa\left(\dfrac{x}{y} \right)B(y) u(t,y)\f{dy}{y}.
 \end{equation}

The  two key physical parameters $B$ and $\kappa$
encode fundamental information on the mechanical stability 
of each particle, 
and can take  different
forms depending on the particular process considered.
\Mig{To estimate} the parameters $B$ and $\kappa$ using 
population data (\comMD{when }only the particles density $u(t,x)$ can be accessed, not the trajectory of each individual particle) is a challenging \textcolor{black}{mathematical } problem, \textcolor{black}{important for the applications.}
The \Mig{specific} application that led us to its study originates from 
the \MT{works} \cite{XHR09,XR13}, where the authors provide experimental 
size distribution profiles of different types of amyloid fibrils, 
in order 
to estimate their intrinsic division properties ($B$ and $\kappa$) and then to relate them to their respective pathogenic properties \cite{BTMPSTDX19}.
It is not possible to follow experimentally each fibril one by one,
hence the necessity to draw the characteristic features 
of each particle from the evolution of the whole population.

{\color{vert}\subsection{Review on existing results to estimate the fragmentation kernel} \label{sec:review}}

Identifying the fragmentation kernel $\kappa$ from observed population data has been a challenging problem \textcolor{black}{for some time. {\color{vert} As detailed below, i}n most of the cases up to now, the analysis of this problem has been based on the idea of  self-similar long-time asymptotic behaviour of the solutions to (\ref{eq:frag_intro}){\color{vert}, see~\cite{BCG13,Bertoin1,BW,CCGU,DubovskiStewart}}.}

In the seminal paper \cite{Kolmogorov40} of A. N. Kolmogorov (1940)\comMD{} \textcolor{black}{on general random processes of particle grinding, the self-similar}  
 large time behaviour of the size distribution  is 
identified in a slightly different but closely related  equation, \textcolor{black}{discretised in  time and  with  a constant fragmentation rate $B$}.
The self-similar asymptotic behaviour of the fragmentation equation \textcolor{black}{written for the cumulative distribution function} \textcolor{black}{was} established in \cite{Filippov61}
by Filippov (1961)  for the case $B(x)=x^{\gamma},\;\gamma>0$  and the result
is now well-known by the scientific community
under fairly general balance assumptions on the parameters \textcolor{black}{(see for instance  \cite{PerthameRyzhik, EMR05,BCG13, CCM}).}

From the seventies, scientists from physics and chemical departments
have been using this similarity concept for the kernel inverse problem. 
In 1974, a scientist of a department of chemical engineering 
\cite{Ram74} developed a method to extract information 
on probabilities of droplet-breakup, and in particular on the daughter-drop-distribution (in 
modern terms: the fragmentation kernel), 
as a function of drop sizes data, obtained from an experiment
of pure fragmentation in a batch vessel. 
To do so, the self-similar  behaviour of the \textcolor{black}{solutions of the } fragmentation equation\textcolor{black}{, written here too for the cumulative distribution function, is assumed}, thereby restricted to {\color{vert} power law} fragmentation rates {\color{vert} ({\it i.e.} $B(x)=\alpha x^\gamma$ with $\alpha,\,\gamma>0$)},
and
the moments of the kernel \MT{are estimated} from the moments of the large time size distribution.
To recover the kernel from its moments, a method based on the expansion of the kernel 
on a specific polynomial basis is suggested.
These results are generalised later in 1980 \cite{NRamG80} to non-{\color{vert}power law} fragmentation rates associated with 
an adapted definition for the self-similarity of the kernel so as to keep the  self-similar asymptotic behaviour of the model.

From the late nineties, the large improvements in computer hardware
opened the field of numerical investigations
of mathematical models.
In  \cite{KK05} the authors provide insights on how the 
stationary shape of the particle size distribution is impacted by the kernel. 
Their conclusion is that the inverse problem of assigning
a breakage kernel to a known self-similar particular size distribution is
ill-posed not only in a mathematical but also in a physical
sense since quite different kernels correspond to almost
the same particles size distribution. {\color{vert} This conclusion has been confirmed by the theoretical results of~\cite{BCG13,DET18}: in these articles, key properties of the fragmentation kernel have been proved to be linked to unobservable quantities of the asymptotic profile, namely its behaviour for very small or very large sizes.}
{\color{vert} In~\cite{DET18}, we proposed} a reconstruction formula for $\kappa$ based on the mere
knowledge of the long-time asymptotic profile $g$   \textcolor{black}{of the solutions of  (\ref{eq:frag_intro})}
in suitable functional spaces \cite{DET18}.
This formula involves the moments of order $s$ of the asymptotic profile $g$, $s$ being taken along a vertical complex line, {\it i.e. }
$s= u+i v, \; v\in (-\infty,\infty)$. {\color{vert}However, due to its severely ill-posed\MT{ness} on the one hand, and to the impossibility of observing the asymptotic profile for very small or very large sizes on the other hand, this reconstruction formula revealed of little practical use. Of note, a similar estimate in the case of the growth-fragmentation equation with constant growth and division rates has been carried out in a statistical setting in~\cite{HNRT19}, together with a consistency result and a numerical study.}

\textcolor{black}{We thus explored further the influence of the kernel on the time evolution of the length distribution~\cite{DETX20}.
We showed that {\color{vert} despite the previously seen limitations}, the asymptotic profile 
remains helpful to distinguish 
whether the fragmentation kernel is an
erosion-type kernel
(one of the fragments has a size close to
that of the parent particle) or  produces particles of similar sizes.
{\color{vert} By statistical testing, w}e also showed that  departing from the same initial condition, 
there
exists a time-window right after
the initial time where two different kernels give rise to a maximal difference of their corresponding size distribution solutions, and 
that the
initial condition that maximizes this difference is a very sharp Gaussian.} {\color{vert} This last remark led us to explore further the short-time behaviour of the solution, which is the basis of our present study.}

\textcolor{black}{Inverse problems for fragmentation equations related with our {\color{vert}"short-time"} approach appeared in 2002 \cite{AgoshkovDubovski} and in 2013 \cite{Alomari}. In the first article the authors consider the reconstruction of a source term in a coagulation-fragmentation equation. The equation is linearized assuming that for short times the solution $c(t, x)$ of the equation may be approximated by the initial data $c_0$, and keeping only linear terms in the perturbation. The inverse problem for the linear equation is then solved
using optimal control methods, the solvability theory of operator equations, and  iteration algorithms.  In the second article the authors solve the linearization of the  inverse problem for (\ref{eq:frag_intro_gen}), obtained assuming   $F=1+f$ with $|f|$ small, $u=c_0+g$ where $c_0$ is the solution of (\ref{eq:frag_intro_gen}) with $F=1$, assuming that $|g|$ is also small, and keeping in the equation only principal terms.} 

\

{\color{vert}\subsection{Outline of our main results}}

\textcolor{black}{In the present article we revisit the question of estimating $\kappa$  from measurements on the population density $u(t,x)$, and  we introduce two main novelties.} \textcolor{black}{First, a}  new method, that  only uses  short-time measurements of the solutions. {\color{vert} As pointed out in the above review,} \textcolor{black}{this is a very different idea  from those generally used up to now since these  are  based on the long time self-similar behaviour of the solutions.  Second, a reconstruction formula for the Mellin transform of $\kappa$ and an estimate of the error of the approximation.
More precisely,}
we assume the fragmentation rate $B$ to be known, and  provide a reconstruction formula for the \MT{sole} fragmentation kernel.

\comMD{Unless specific assumptions are stated, w}e restrict the study to {\color{vert} power law} fragmentation 
rates 

 \Mig{
\begin{equation}
\label{B}
B(x)= \a x^{\gamma}, \qquad \gamma>0,\quad \alpha > 0.
\end{equation}
}
The {\color{vert}} guiding idea {\color{vert}of our study} is based on the following remark: for $\Delta t$ small enough, 
the solution \textcolor{black}{$\mu $} to the fragmentation equation \eqref{eq:frag_intro} with $B$ defined by~\eqref{B} formally satisfies

\begin{equation}\label{eq:approx1}
\textcolor{black}{\mu }(t+\Delta t,x)\approx \textcolor{black}{\mu }(t,x) -\a  \Delta t x^{\gamma}\textcolor{black}{\mu }(t,x)+ \alpha \Delta t \dst\int_{x}^{\infty}\kappa\left(\dfrac{x}{y} \right)y^{\gamma-1} \textcolor{black}{\mu }(t,y)dy+ o(\Delta t).
\end{equation}

If we assume that at time $t$, the size distribution $\textcolor{black}{\mu }(t,x)$ is a Dirac delta function at $x= 1$, \Mig{that
is denoted $\delta_1$ or $\delta(x-1)$,} then 

\begin{equation*}
\textcolor{black}{\mu }(t+\Delta t,x)\approx \delta(x-1) -\a  \Delta t \delta(x-1) +  \alpha  \Delta t \kappa(x) +o(\Delta t), 
\end{equation*}
and thus the kernel $\kappa$ can be  directly  estimated from the measurement of the profile $\textcolor{black}{\mu }$ at time $t+\Delta t$ as
\begin{equation*}
 \kappa(x) \approx \dfrac{1}{\a \Delta t}\left(\textcolor{black}{\mu }(t+\Delta t,x) - (1-\a \Delta t) \delta(x-1)\right) +o(1), \qquad  \Delta t \ll 1.
\end{equation*}

To make the above estimate of $\kappa$ rigorous, 
we  \textcolor{black}{first prove the  uniqueness of a \MT{non negative} solution $\mu $ to the Cauchy problem for the  equation (\ref{eq:frag_intro}) when $\kappa$ and the initial data \MT{$\mu_0$}  are non negative measures satisfying some suitable conditions (see Theorem~\ref{thm:WP_BV} below).} \textcolor{black}{Then we} expand the solution $\textcolor{black}{\mu }(t,x)$ 
as a power series about $t$ in the Banach space of Radon measures.
Up to our knowledge, such representation of \textcolor{black}{the measure-valued  solution} of the fragmentation equation \textcolor{black}{with a fragmentation kernel measure $\kappa$ is new{\color{vert}, though }\textcolor{black}{some explicit solutions of the fragmentation equation in form of series are given in  \cite{ZiffMcGrady85, ZiffMcGrady86}  for  particular continuous fragmentation functions  and particular initial  data $\mu _0$.}}

\textcolor{black}{To estimate}  $\kappa$ 
 from the measurement of the distribution profile $\textcolor{black}{\mu }(\Delta t ,.)$ \textcolor{black}{for small values of $\Delta t$}, the cunning \textcolor{black}{observation}
is to impose that the initial distribution $\textcolor{black}{\mu_0 }$ is a Dirac mass. 
In other words, at time $t=0$, all particles should have the same size. 
Heuristically, 
if all particles have the exact same size at $t=0$, after a time $t$
long enough so that {\color{vert} a non-negligible quantity of} particles have broken once, but short enough so that 
{\color{vert} a negligible quantity of} particles has broken twice, 
it is clear that the kernel $\kappa$, sometimes referred to as the ``daughter particle distribution" can directly be read on the distribution {\color{vert} of particles strictly smaller than initially}.

\textcolor{black}{Of course} no experiment may produce a suspension where all the particles have the same size
since it would mean being able to follow each particle one by one. 
However, we can hope to obtain a suspension where all particles have approximately the same size, described {\color{vert} for instance} by a gaussian
distribution \textcolor{black}{that would be not too far from a Dirac delta function}.  For that reason the stability 
 of  \textcolor{black}{ our estimates of $\kappa$ with respect  to measurement noise and to the error on the initial data $\mu _0$ must be estimated
. It is then necessary to consider 
measure-valued solutions. The existence and uniqueness of such solutions to coagulation or  fragmentation equations has been already studied in the literature of mathematics, for example in \cite{Norris}, \cite{Bertoin} for the coagulation equation, \cite{BW}, \cite{BCGM} for a growth-fragmentation equation,  \cite{CCGU} for a fragmentation equation but where only the case $\gamma =0$ would satisfy the hypothesis.}

Quantifying the stability result first \Mig{requires to understand what are the types of experimental
uncertainties on the initial data 
coming from the experiments. These} are twofold: first, 
instead of a delta function at $x=x_0$, the initial data is a spread {\color{vert}distribution} 
with variance $\sigma>0$ (\Mig{due to the} impossibility to obtain a perfectly homogeneous suspension). \Mig{Second, this {\color{vert}distribution} is centered at  $x=x_0+\eps$ for some $\eps>0$, instead
of $x=x_0$} 
(possible bias on the measurement of the particles' sizes). \Mig{In order to deal with} these uncertainties, the Bounded-Lipshitz (BL) norm is \Mig{better suited} than the total variation norm (TV). For instance, $\forall a\in \R,\,b\in \R,$  such that $|b-a| <2$
\Mig{
\begin{equation*}
 \|\delta_{a}-\delta_b \|_{TV} =2,\qquad \|\delta_{a}-f_{a,\sigma} \|_{TV} =2,
\end{equation*}
whereas
\begin{equation*}
\|\delta_{a}-\delta_b \|_{BL} =|b-a|,\qquad \|\delta_{1}-f_{1,\sigma} \|_{BL} \leq\f{2\sqrt{\sigma}}{\sqrt{2\pi}},
\end{equation*}
}
where $f_{a,\sigma}$ is the density of the gaussian function centered at $x=a$ with variance $\sigma$.

\textcolor{black}{However, {\color{vert} in the case of a generic initial data not necessarily close to a delta function, a reconstruction formula may still be obtained through the use of the moments of the solution. F}rom the  very beginning {\color{vert} of} the study of inverse problems for  the fragmentation equation,   the moments of the solution $\mu $ (\cite{Ram74, NRamG80}), and then its Mellin transform, have been extensively used.  Of note, the Mellin transform of $\kappa $ (denoted by $K$ from now on), is of interest by itself since it provides a  range of moments of the fragmentation kernel, in particular variance and skewness.  An exact expression for $K$ was obtained in \cite{DET18} from the  long-time self similar asymptotic profile of the solution  $\mu $ in terms of an (in general)  oscillatory  integral, but no way to approximate this integral and estimate the error was given.
The exact series representation of the solution $\mu $ to (\ref{eq:frag_intro})  obtained \MT{in the present paper} may be used in order to deduce an approximation of  $K$ and estimate the error of such an approximation.}

Our last contribution is  thus a robust reconstruction formula of $K$. \textcolor{black}{To this end, we use short-time measurements of  the solution $\mu $ to equation (\ref{eq:frag_intro}) for generic initial data $\mu _0$, not necessary close to a Dirac measure, and the initial data itself. This dependence on the initial data $\mu _0$  contrasts with the result in \cite{DET18} where the reconstruction formula (see Theorem 2 in~\cite{DET18}) only involves the long-time asymptotic profile of the solution.}
Since the equation is autonomous, this means to be able to access two close consecutive measurements of the particles size distribution {\color{vert} - an experimental setting much more realistic than to depart from a mono-disperse suspension}.

{\color{vert} To sum-up,} the main novelties brought by this paper are

\

\begin{itemize}
 \item \Mag{a proof of the uniqueness \MT{and stability} of the solution in the space of \MT{non negative} measures endowed with the total variation norm (Theorem \ref{thm:WP_BV})},
 \item a representation of \MT{a} solution to the fragmentation equation (\MT{endowed with any non-negative measure as initial data)}
 as a power series in the Banach space of measures endowed with the total variation norm
 (Theorem \ref{thm:power_serie}), \MT{implying in particular existence of measure-valued solutions to \eqref{eq:frag_intro},}
 \item \MT{a proof of the non-negativity of the power series solution (Theorem \ref{thm:positivity}),}
 \item \MT{as a consequence of the three previous items and summarized in Corollary \ref{cor:summary}, a statement of existence of a unique non-negative measure-valued solution to \eqref{eq:frag_intro}, accompanied with a power series representation of this unique solution,}
 \item a stability result for the solution to the fragmentation equation 
 for the BL norm, which is a norm adapted to weak convergence of measures (Theorem  \ref{thm:stabBL}),
 \item a robust reconstruction formula for the fragmentation kernel 
 involving the short-time solution of the fragmentation equation endowed with 
 a delta function as initial condition. 
 Robustness is to be understood in the sense that if the initial condition is close to 
 a delta function at $x=x_0$ in the BL norm (for instance
 a rectangular function centered in $x_0$ or a delta function at $x=x_0+\epsilon$ with $\epsilon$ small), then the estimated kernel obtained with the reconstruction formula
 is close to the real kernel in the BL norm
 (Theorem \ref{thm:short} and Theorem \ref{thm:stab}),
 \item a reconstruction formula for the Mellin transform $K$ of the fragmentation kernel $\kappa$  involving the short-time solution of the fragmentation equation endowed with 
any initial condition (Theorem \ref{thm:reco}).
\end{itemize}

The outline of the paper is as follows. 
\textcolor{black}{In the remaining of Section \ref{sec:intro},
some properties of
measures and classical results on measure theory are recalled, as well as the definition of Mellin transform and Mellin  convolution. Section \ref{Section20} is devoted to the proof of the existence, uniqueness, \MT{non negativity} and  series representation of solutions to  the problem \eqref{eq:frag_intro} (with $B$ defined by~\eqref{B}) in the space of Radon measures, and their stability with respect to the initial data in the TV norm.}
In Section \ref{sec:short},
estimates of the fragmentation kernel and bounds for the  error of such estimates are obtained using, for small values of the time variable,  the expression as a series of the solution $\mu$ provided by Theorem~\ref{thm:power_serie}. The stability of these estimates with respect to the initial data and noise measurements is also considered in BL norm.
In Section \ref{sec:reco}, we study the Mellin transform $K$ of $\kappa $. \Mig{Under some regularity assumption on $\kappa $ and on the initial data $\mu _0$, a reconstruction formula $K^{est}$ of
$K$ is obtained, only based on short time-intervals measurements
of the solution to the fragmentation equation and  an  estimate of the error $K-K^{est}$ is obtained. An estimate of the variance of $\kappa $ is then deduced, and under a stronger regularity assumption on $\kappa $, a pointwise estimate of the difference of $\kappa$ and the inverse Mellin transform of $K^{est}$ is proved.} We end the paper with a numerical investigation of the short-time behaviour of the fragmentation equation, we illustrate the estimation results of Theorems~\ref{thm:short} and~\ref{thm:stab}, and 
 we explore how Theorem \ref{thm:reco} can be applied 
 to recover the variance of the kernel from the data.
 \MT{For every theorem, the constants arising in estimates and depending continuously on parameters $p_1,p_2, \dots$ are denoted by $C(p_1,p_2, \dots)$.}

\subsection{\Mig{Short reminder} on measure theory}
We define $\M(\R^+)$ as the set of Radon measures $\mu$ (not necessarily probability measures)
such that $\supp(\mu) \subset \R^+$. 
Let us recall that $\M(\R^+)$ is the dual space 
of the space $(\mathcal{C}(\R^+),\|.\|_{\infty})$ of continuous functions.
We denote by
$(\mu^+,\mu^-)$ the Jordan decomposition of $\mu$.
\Mag{We endow $\M(\R^+)$  with two different norms: the total variation norm and the Bounded-Lipschitz norm.
As mentio\comMD{n}ed in the introduction, the final purpose is to obtain stability with respect {\color{vert} to} the BL norm, the TV norm being a technical intermediate tool
to reach this pur\comMD{p}ose.
}
\Mag{
The TV norm of
the (signed) measure  $\mu \in \M(\R^+)$ is
defined as}
\Mag{
\begin{equation}
\label{TVdef}
    \|\mu\|_{TV} = \sup \{\dst\int_{R^+}  \varphi(x)d\mu(x), \quad \varphi \in \mathcal{C}(\R^+)\cap L^1(d|\mu|), \; \|\varphi\|_{\infty} \leq 1\}.
\end{equation}
}
We recall that 
$\left(\M(\R^+),\|.\|_{TV}\right)$ is a Banach space.
We now define the BL norm as 
\begin{equation}\label{def:BL}
    \|\mu\|_{BL} = \sup \{\dst\int_{R^+}  \varphi(x)d\mu(x), \quad \varphi \in \mathcal{C}(\R^+)\cap L^1(d|\mu|), \; \|\varphi\|_{\infty} \leq 1, \; 
    \|\varphi '\|_{\infty}\leq 1\}.
\end{equation}
\Mag{
Comparing \eqref{TVdef} and \eqref{def:BL}, it is clear that 
\begin{equation}
\label{BL_TV}
 \comMD{\forall \; \mu\in \M(\R^+),\qquad} \|\mu\|_{BL} \leq \|\mu\|_{TV}.
\end{equation}
}
An optimal transportation point of view is provided in \cite[Proposition 23]{PRT17} 
for the BL norm.
It is proven that for any signed Radon measure with finite mass $\mu$ we have

\Mig{
\begin{equation}
\begin{aligned}
&\|\mu\|_{BL} =\inf \Big\{ (\|\mu^+ - \nu\|_{TV} +\|\mu^-- \eta\|_{TV})+  W_1(\nu,\eta),\quad  (\nu, \eta) \in \mathcal{M _{ \mu  }^+}(\R^+),\Big\}
 \label{BLdef}
 \end{aligned}
\end{equation}
\begin{equation}
\begin{aligned}
& \mathcal{M _{ \mu  }^+}(\R^+)=\Big\{ (\nu, \eta)\in \mathcal{M^+}(\R^+)\times \mathcal{M^+}(\R^+);\,  \nu\leq \mu^+ , \eta \leq \mu^-, \|\nu\|_{TV}=\|\eta\|_{TV} 
\Big\}\label{BLdef2}
 \end{aligned}
\end{equation}
}

where $\mathcal{M^+}(\R^+)$ is the space of positive Radon measures with support in $\R^+$, and $W_1$ stands for the classical Wasserstein distance \cite{Villani_topics} between two positive measures of same mass, namely
\begin{equation}\label{defWasserstein}
\begin{array}{c}W_1(\nu,\eta):=\inf\limits_{\pi\in \Pi(\nu,\eta)} \int_{\R^+} \vert x-y\vert d\pi(x,y),\\ \\
 \Pi(\nu,\eta):=\big\{ 
\pi\text{ positive measure on }\R^+\text{ s.t.} \int_{\R^+}\pi(x,y)dx=\eta(y),\quad \int_{\R^+}\pi(x,y)dy=\nu(x)\big\}.
\end{array}
\end{equation}
Let us recall that for $\mu,\nu$ two probability measures and for $a>0$, we have 
$W_1(a\mu, a\nu) =  a W_1(\mu,\nu)$.
Formula \eqref{BLdef} \eqref{BLdef2} can be interpreted as follows: 
the BL norm of the signed measure $\mu$ is the BL distance between the two positive measures 
$\mu^+$ and $\mu^-$.
Now take $\mu^+$ and $\mu^-$ two positive measures. Consider $\nu$ and $\eta$ two positive measures such that $\nu \leq \mu^+$, $\eta\leq \mu^-$ and $\|\nu\|_{TV}=\|\eta\|_{TV}$.
The subpart $\nu$ of the measure $\mu^+$ is transported onto the subpart $\eta$ of the measure $\mu^-$, with a cost $W_1(\nu,\eta)$. The remaining positive measures $(\mu^+ -\nu)$ and 
 $(\mu^- -\eta)$
are both cancelled with a cost 
$\|\mu^+ -\nu\|_{TV}+\|\mu^- -\eta\|_{TV} $.
Among all couples $(\nu,\eta)$ that satisfy $\nu \leq \mu^+$, $\eta\leq \mu^-$ and $\|\nu\|_{TV}=\|\eta\|_{TV}$, we choose one such that 
the sum $(\|\mu^+ - \nu\|_{TV} +\|\mu^-- \eta\|_{TV})+  W_1(\nu,\eta)$ is minimal (such a couple exists, it is proved in \cite{PR14} that the infimum is actually a minimum).
Let us give three examples.

\begin{itemize}
\item
Take $\mu=\delta(x-1)$ and $\mu_{\eps}=\delta(x-(1+\eps))$.
Consider $\nu_a = a \mu$ and $\eta_a= a \mu_{\eps}$ with $0\leq a \leq 1$.
Then  $0 \leq \nu_a \leq \mu$, $0 \leq \eta_a \leq \mu_{\eps}$, and 
$\|\nu_a\|_{TV} =\|\eta_a\|_{TV}   =a$.
Using formula \eqref{BLdef}\eqref{BLdef2} we have 
\begin{equation*}
\begin{aligned}
 \|\mu-\mu_{\eps}\|_{BL}& = \inf\limits_{ 0 \leq a \leq 1}
 \Big\{(\|\mu-\nu_a\|_{TV} +\|\mu_{\eps}-\eta_a\|_{TV})
+  W_1(\nu_a,\eta_a)
 \Big\} \\
 & = \inf\limits_{ 0 \leq a \leq 1}
 \Big\{2(1-a)
+  a \eps
 \Big\} \\
 &= 
 \begin{cases}
 & \eps\qquad\text{for}\;\eps\leq 2, \\
&2 \qquad \text{ for}\;\eps>2.
 \end{cases}
 \end{aligned}
\end{equation*}

\item
Take $\mu=\delta(x-1)$ and $\mu_{\sigma}$ is the measure with the rectangular density 
$\f{1}{2\sigma\sqrt{3}}\mathbb{1}_{[1-\sigma\sqrt{3},1+\sigma\sqrt{3}]}$ with variance $\sigma^2$
for $0<\sigma<1$. 
 We take $\nu_a=\mu=\delta_1$ and $\eta_a=\mu_\sigma=f_\sigma dx,$ in~\eqref{BLdef}, and obtain 
  \begin{equation*}
 \|\mu-\mu_{\sigma}\|_{BL} \leq  W_1(\mu,\mu_\sigma)\leq\int_{1-\sigma \sqrt{3}}^{1+\sigma \sqrt{3}} \f{\vert y-1 \vert dy}{2\sqrt{3}}=\f{\sqrt{3}}{2}\sigma.
 \end{equation*}
\item Take $\mu=\delta(x-1)$ and $\mu_\sigma$ the Gaussian with mean $1$ and variance $\sigma^2.$ We have
$$\Vert \mu - \mu_\sigma\Vert_{BL} \leq W_1(\mu,\mu_\sigma) \leq \int \vert x \vert \f{e^{-\f{x^2}{2\sigma^2}}}{\sqrt{2\pi}\sigma}=\f{2\sigma}{\sqrt{2\pi}}.$$ 
\end{itemize}

 We recall that for $\mu\in\M(\R^+)$ and $T\in \mathcal{C}(\R^+)$, the pushforward $\eta$ of the measure $\mu$ by the function $T$ is defined as the unique measure 
 \begin{equation*}
   \MT{\eta=T\#\mu}
 \end{equation*} such that
for all $\varphi \in \mathcal{C}(\R^+)$,
\begin{equation*}
 \dst\int \varphi(x) d\eta(x) =   \dst\int (\varphi\circ T)(x) d\mu(x).
\end{equation*}
  For $\ell>0$, we define the application
\begin{equation}
\label{Tell}
 T_{\ell}(x) = \ell x, \qquad x\in \R^+.
\end{equation}

\subsection{Mellin transform}

\begin{definition}
 For a measure $\mu \in \M(\R^+)$, 
 its Mellin transform $M[\mu]$ is defined as 
 \begin{equation}
 \label{MT}
   M[\mu](s)= \dst\int_{\R^+} x^{s-1}d\mu(x),
 \end{equation}
 for $s\in \mathbb{C}$ such  that \eqref{MT} is well-defined.
\end{definition}

\begin{definition}[Mellin convolution (cf. \cite{ML}]
\label{def:mult_conv}
Take $\mu$ and $\nu$ two compactly supported finite measures on $\R^+$.
Their \MT{Mellin convolution (sometimes referred to as multiplicative convolution)} is defined as 
\begin{equation*}
 \forall \varphi \in \mathcal{C}(\R^+), \qquad  \langle\mu \ast \nu ,\varphi\rangle =  \langle \mu^x \otimes \nu^y,\varphi\circ p \rangle,
\end{equation*}
where $p : (x,y)\to xy$.
\end{definition}

If $d\mu(x)=f(x)dx$ and $d\nu(x)=g(x)dx$ for $f$ and $g$ in $L^1(\R^+)$, then $\mu\ast\nu$ is the measure with density 
\begin{equation*}
(f\ast g) (x) =\dst\int_{\R^+} f(y)\; g\left( \dfrac{x}{y}\right)\dfrac{dy}{y}.
 \end{equation*}
If $d\mu(x)=f(x)dx$ with $f\in \mathcal{C}(\R^+)$ and $\nu=\delta(y-\ell)$, then $\mu\ast\nu=T_\ell \# \mu$ is the measure with density
 \begin{equation*}
  (f\ast \nu)(z)=   \dfrac{1}{\ell}f\left(\dfrac{z}{\ell}\right).
 \end{equation*}

 \begin{proposition}[Mellin transform and Mellin convolution]
\label{prop:MellinMC}
Take $\mu$ and $\nu$ two compactly supported finite measures on $\R^+$.
For the $s$ for which the expression below is defined, we have
\begin{equation*}
 M[\mu\ast\nu](s) = M[\mu](s) M[\nu](s).
\end{equation*}
\end{proposition}

\section{
\textcolor{black}{Measure-valued solutions to the fragmentation equation: existence, uniqueness, stability  and series representation.}}

\label{Section20}
\Mig{The basis of our analysis in all the remaining of this work are the weak solutions to the Cauchy problem for equation \ref{eq:frag} with the initial condition
\begin{equation}
\label{eq:data}
\mu _t(\MT{t=0})=\mu _0,
\end{equation}
whose precise definition is given  below.} 
Throughout the present paper, the following assumptions are used.

\begin{enumerate}[label={\bf Hyp-\arabic*}]
\item \label{hyp1} The fragmentation kernel  $\kappa \in \mathcal{M^+}(\R^+)$  contains no atom at $x=0$ and at $x=1$, and satisfies 
\begin{equation}
 \supp(\kappa)\subset [0,1], \qquad \int\limits_0^1d\kappa(z) =N<+\infty, \qquad  \int\limits_0^1 z d\kappa(z)=1.
\end{equation}
\item \label{hyp2} The initial condition  $\mu_0 \in \mathcal{M^+}(\R^+)$ is compactly supported 
\begin{equation}
 \supp(\mu_0)\subset [0,\MT{L}].
\end{equation}
\end{enumerate}

Even though $\kappa$ and $\mu_t$ are measures, we sometimes write the fragmentation equation as

 \begin{equation}\label{eq:frag}
 \begin{aligned}
\dfrac{\p}{\p t}\mu_t (x)& =-  \a x^{\gamma}\mu_t(x) + \alpha \dst\int_{x}^{\infty}\kappa\left(\dfrac{x}{y} \right)y^{\gamma-1} d\mu_t(y) ,\qquad 
  \mu_{t=0}(x)  = \mu_0(x),
 \end{aligned}
 \end{equation}
 
or as

 \begin{equation*}\label{eq:frag2}
 \begin{aligned}
\dfrac{\p}{\p t}\mu_t (x)& =-  \a x^{\gamma}\mu_t(x) + \alpha \dst\int_{x}^{\infty}\kappa\left(\dfrac{x}{y} \right)y^{\gamma-1} \mu_t(y)dy ,\qquad 
  \mu_{t=0}(x)  = \mu_0(x).
 \end{aligned}
 \end{equation*}

\begin{definition}[Weak solution for \eqref{eq:frag}]
 A family $(\mu_t)_{t\geq 0} \subset \M(\R^+)$ is called a measure\MT{-valued} solution to \Mig{problem}  \eqref{eq:frag_intro} \eqref{B}  \eqref{eq:data} with initial data
 $\mu_0 \in \M(\R^+)$ satisfying \eqref{hyp2} if the mapping $t\to \mu_t$ is narrowly continuous and for all $\varphi \in \mathcal{C}(\R^+)$ {\color{vert} such that $x\mapsto \varphi(x)/(1+x)$ is bounded on $[0,\infty)$,} and all $t \geq 0$,
\Mag{\begin{equation}
\label{ppcm1}
 \int_{\R^+} \varphi(x) d\mu_t(x)=  \int_{\R^+} \varphi(x) d\mu_0(x) + \dst\int_0^t ds \dst\int_{\R^+}
   d\mu_s(x)  \alpha x^{\gamma}\left( - \varphi(x)
 + \dst\int_{0}^{1} d\kappa(z)\varphi(xz)\right).
\end{equation}
}
\label{def:measure_valued_sol}
\end{definition}

We recall that $\mu_n$ converges narrowly toward $\mu$ if for all $\varphi \in C_b(\R^+)$, 
$\int \varphi d\mu_n \to \int \varphi d\mu$, where $C_b(\R^+)$ denotes the set of continuous and bounded functions defined on $\R^+$.

\textcolor{black}{
Although several results may be found  in the references given in the introduction  about the existence and uniqueness of solutions to fragmentation equations, none of them covers exactly the hypotheses that we have in mind for $\kappa$ and the initial data $\mu _0$. For the sake of completeness, our first result is then an existence and uniqueness of compactly supported \MT{and non negative} measure-valued solutions to \eqref{eq:frag}  under assumptions \eqref{hyp1}, \eqref{hyp2}. We begin with a uniqueness and stability result.}

\begin{theorem}[Uniqueness and \MT{TV}-stability for the fragmentation equation in $(\mathcal{M}_+(\R^+),\|.\|_{TV})$]
\label{thm:WP_BV}
Assume\comMD{~\eqref{hyp1}}, \eqref{hyp2} and $\comMD{\gamma\geq 0}$. Suppose that $\mu _t\in \mathcal{C}(\R^+, \mathcal{M}(\R^+))$, 
is a non negative measure-valued solution to \eqref{eq:frag}, in the sense of Definition \ref{def:measure_valued_sol}.
\textcolor{black}{Then, for all $t>0$,}
\begin{align}
&\supp(\mu_t) \subset [0,\MT{L}]  \label{pgcd11} \\
&\|\mu_t\|_{TV} \leq  \|\mu_0\|_{TV} e^{\alpha (2\MT{L})^{\gamma} (1+N) t}  \label{pgcd2}\\
&\dst\int_{\R^+} xd\mu_t(x) =\dst\int_{\R^+} xd\mu_0(x) \label{pgcd3} 
\end{align}
 where \MT{$N$ is defined in \eqref{hyp1} and $L$ is} defined in \eqref{hyp2}. In particular {\color{vert}such a solution is unique.}
\end{theorem}

\begin{proof}
\MT{Consider $\mu_t \in \mathcal{C}(\R^+,\M(\R^+))$ a non negative measure-valued solution to \eqref{eq:frag}
in the sense of Definition \eqref{def:measure_valued_sol}}. We start  proving property \eqref{pgcd11}. To this end we first notice that
\begin{align*}
\alpha x^\gamma \mu _t(x)=\MT{\alpha}\mu _t(x)\int _0^x\frac {y} {x}x^{\gamma -1}\kappa \left( \frac {y} {x}\right)dy.
\end{align*}
Then in the right-hand side of (\ref{ppcm1}) we write
\begin{align*}
\int_{\R^+}
   d\mu_s(x)  \alpha x^{\gamma}&\left( - \varphi(x)
 + \dst\int_{0}^{1} d\kappa(z)\varphi(xz)\right)=\\
 &=\int _0^\infty\int _0^x b(x, y)\, y\left( \frac {\varphi (y)} {y}- \frac {\varphi (x)} {x}\right)dyd\mu _s(x).
\end{align*}
where
\begin{align}
\label{bxy}
b(x,y)=\MT{\alpha x^{\gamma-1}\kappa \left(
\frac{y}{x}\right)}
.
\end{align}
Consider then the test function
\begin{align*}
\varphi (x)=
\begin{cases}
0 &\forall x\in [0, L]\\
x(x-L)\quad &\forall x\in [L, L+1]\\
x &\forall x\ge L+1.
\end{cases}
\end{align*}
Since $x\to \varphi  (x)/(1+x)$ is bounded and non decreasing  on $[0, \infty)$, by (\ref{ppcm1})
\begin{align*}
\int_{\R^+} \varphi(x) d\mu_t(x)&=  \int_{\R^+}
\varphi(x) d\mu_0(x) + \int_0^t \int _0^\infty\int _0^x b(x, y)\, y\left( \frac {\varphi (y)} {y}- \frac {\varphi (x)} {x}\right)dyd\mu _s(x) ds\\
& \le  \int_{\R^+} \varphi(x) d\mu_0(x) =0,
\end{align*}
\MT{where the last inequality is justified since $\mu_t\geq 0$ for $t\geq 0$ and since $x\to \varphi(x)/x$ is non decreasing as well.}
Since $\varphi\ge 0$ and $ \mu _t\ge 0$ for all $t\ge 0$ it follows that for $\varphi (x)d\mu _t(x)=0$ for all $t>0$ and almost every $x>0$. Since by construction $\varphi (x)>0$ for all $x>\MT{L}$ we deduce that for every $t>0$, $\supp (\mu _t)\subset [0, \MT{L}]$.

To prove the BV estimate (\ref{pgcd2}), we use definition \eqref{TVdef}, and take
$\varphi \in \mathcal{C}(\R^+)$ such that
$\|\varphi\|_{\infty} \leq 1$. By (\ref{ppcm1}),
\begin{equation*}
\int_{\R^+} \varphi(x) d\mu_t(x)=  
\MT{\int_{\R^+} \varphi(x) d\mu_0(x)+}\dst\int_0^t ds \dst\int_{\R^+} \alpha x^{\gamma}
   d\mu_s(x) \left( - \varphi(x)
 + \dst\int_{z=0}^{1} \varphi(xz)d\kappa(z)\right).
\end{equation*}
Let $\chi\in C([0, \infty))$ such that $||\chi||_\infty=1$,  $\chi(x)=1$ for all $x\in [0, \MT{L}]$ and $\chi(x)=0$ for $x>2\MT{L}$ and consider the function defined as
\begin{equation*}
\psi(x) :=  \alpha x^{\gamma}\chi(x)   \left( - \varphi(x)
+ \dst\int_{z=0}^{1} \varphi(xz)d\kappa(z)\right),
\end{equation*}
It satisfies $ \psi \in \mathcal{C}(\R^+)$ and, since
$\|\varphi\|_{\infty} \leq 1$,  and $\supp(\chi)\subset [0, 2 \MT{L}]$,

\begin{equation*}
\sup_{0\le x\le L} |\psi\MT{(x)}|\leq \alpha (2 \MT{L})^{\gamma} (1+N).
 \end{equation*}
Therefore,
\begin{equation*}
\int_{\R^+} \varphi(x) d\mu_t(x)\leq \|\mu_0\|_{TV}+ \alpha (2\MT{L})^{\gamma} (1+N) \dst\int_0^t \|\mu_s\|_{TV} ds,   \end{equation*}
which implies 
\begin{equation*}
\|\mu_t\|_{TV}\leq \|\mu_0\|_{TV}+ \alpha (2\MT{L})^{\gamma} (1+N) \dst\int_0^t \|\mu_s\|_{TV} ds,      
\end{equation*}
and Gronwall Lemma yields (\ref{pgcd2}).
Finally, mass conservation property (\ref{pgcd3}) is obtained by choosing $\varphi(x)=x$ in definition \ref{def:measure_valued_sol}
\MT{and using the last statement of \eqref{hyp1}}.
\end{proof}

\MT{The following proposition provides us with a solution to the fragmentation equation when initial condition is a delta mass localized at $x=\ell$ 
in terms of a fundamental solution. 
}

\begin{proposition}[Fundamental solution rescaled]
\label{prop:rescaling}
Fix $\ell>0$.
\MT{Assume that $\mu_t^{F}$ is a fundamental solution to \eqref{eq:frag}, i.e. a solution to \eqref{eq:frag} 
when the initial data is $\mu_0=\delta(x-1)$. 
Then, $\mu_t^{\ell} = T_{\ell} \# \mu_{\ell^{\gamma}t}^{F}$, with $T_{\ell}(x)= \ell x$,
is a solution to \eqref{eq:frag} with 
$\mu_0^{\ell}=\delta(x-\ell)$.}
\end{proposition}
\begin{proof}
We set $\mu_t^{\ell}:= T_{\ell} \# \mu_{\ell^{\gamma}t}^{F}$. 
Let us prove 
that $\mu_t^{\ell}$ is a solution to \eqref{eq:frag}  with
initial condition 
$\mu_0^{\ell}=\delta(x-\ell)$ and conclude by uniqueness of the solution.
First, $\mu_0^{\ell} = T_{\ell}\# \mu_0 = T_{\ell}\# \delta(x-1) =\delta(x-\ell)$.
Then, we obtain that for all $\varphi\in \mathcal{C}_c(\R^+)$
\begin{equation*}
 \dst\int_{\R^+} \varphi(x) d\mu_t^{\ell}(x)=  \dst\int_{\R^+} \varphi(x) d\left(T_{\ell}\# \mu_{\ell^{\gamma}t}^{F}\right)(x)
= \dst\int_{\R^+} (\varphi\circ T_{\ell})(x) d \mu_{\ell^{\gamma}t}^{F}(x)
= \dst\int_{\R^+} \varphi(\ell x) d \mu_{\ell^{\gamma}t}^{F}(x).
\end{equation*}
Since $\mu_t$ is a weak solution to \eqref{eq:frag}, we have 
\begin{equation*}\begin{array}{ll}
 \dst\int_{\R^+} \varphi(\ell x) d \mu_{\ell^{\gamma}t}^{F}(x)=&
  \int_{\R^+} \varphi(\ell x) d\mu_0(x)
   \\ \\
  &+\a \dst\int_0^{\ell^{\gamma t}} \dst\int_{\R^+} 
 \left( -x^{\gamma} \varphi(\ell x)d\mu_s^{F}(x)+ \varphi(\ell x)\dst\int_x^{\infty}y^{\gamma-1}  d\kappa\left(\dfrac{x}{y}\right)d\mu_s^{F}(y)\right)ds.
 \end{array}
\end{equation*}
Let us treat each of the three terms of the sum above separately.
The first term is 
\begin{equation*}
  \int_{\R^+} \varphi(\ell x) d\mu_0(x) =  
  \int_{\R^+} \varphi(x) d\mu_0^{\ell}(x).
\end{equation*}
The second term is treated using the change of variables $s=\ell^{\gamma}u$
\begin{equation*}
 -\a \dst\int_0^{\ell^{\gamma} t} \dst\int_{\R^+} 
x^{\gamma} \varphi(\ell x)d\mu_s^{F}(x)ds=
 -\a \dst\int_0^{t} \dst\int_{\R^+} 
\left(x\ell\right)^{\gamma} \varphi(\ell x)d\mu_{\ell^{\gamma}u}^{F}(x)du,
\end{equation*}
and then the change of variables $z=T_{\ell}(x)$ {\it i.e.} 
$d\mu_{\ell^{\gamma}u}^{F}(x)= d\left( T_{\ell} \#\mu_{\ell^{\gamma}u}^{F} \right)(z)$ 
\begin{equation*}
\begin{aligned}
  -\a \dst\int_0^{t} \dst\int_{\R^+} 
\left(x\ell\right)^{\gamma} \varphi(\ell x)d\mu_{\ell^{\gamma}u}^{F}(x)du
&=  -\a \dst\int_0^{t} \dst\int_{\R^+} 
z^{\gamma} \varphi(z)d\left(T_{\ell} \# \mu_{\ell^{\gamma}u}^{F}\right)(z)du\\
&=  -\a \dst\int_0^{t} \dst\int_{\R^+} 
z^{\gamma} \varphi(z)d\mu_u^{\ell}(z)du.\\
\end{aligned}
\end{equation*}
For the third term we also use the change of variables $s=\ell^{\gamma}u$
followed by the change of variables $z=T_{\ell}(x)$ 
and to finish the change of variable $w=T_\ell (y)$ {\it i.e.} $d\mu_{\ell^{\gamma}u}^{F}(y)= d\left( T_{\ell} \#\mu_{\ell^{\gamma}u}^{F} \right)(w)=d\mu_u^\ell(w)$ 
to get 
\begin{equation*}
 \a \dst\int_0^{\ell^{\gamma t}} \dst\int_{\R^+} 
\varphi(\ell x)\dst\int_x^{\infty}y^{\gamma-1}  d\kappa\left(\dfrac{x}{y}\right)d\mu_s^{F}(y)ds
=
\a \dst\int_0^{t} \dst\int_{\R^+} 
\varphi(z)\dst\int_z^{\infty}w^{\gamma-1}  d\kappa\left(\dfrac{z}{w}\right)d\mu_u^{\ell}(w)du.
\end{equation*}
To summarize, 
\begin{equation*}
  \dst\int_{\R^+} \varphi(x) d\mu_t^{\ell}(x) =
   \dst\int_{\R^+} \varphi(x) d\mu_0^{\ell}(x)-\a \dst\int_0^{t} \dst\int_{\R^+} 
 \left(z^{\gamma} \varphi(z)d\mu_u^{\ell}(z)+  
\varphi(z)\dst\int_z^{\infty}w^{\gamma-1}  d\kappa\left(\dfrac{z}{w}\right)d\mu_u^{\ell}(w) \right)du.
\end{equation*}
Finally, since 
$t\to \mu_t^{F}$ is narrowly continuous, then $t\to \mu_t^{\ell}$ is narrowly continuous as well.
This ends the proof of Proposition \ref{prop:rescaling}.
\end{proof}
\MT{
\begin{theorem}[Existence of a solution to \eqref{eq:frag} represented as a power series]
\label{thm:power_serie}
For any  fragmentation kernel $\kappa $
satisfying \eqref{hyp1}, $\gamma\geq 0$ and  $\mu_0$ satisfying \eqref{hyp2},
 there exists a 
 weak solution $\mu_t\in\mathcal{C}(\R^+,\M(\R^+))$ to \eqref{eq:frag}
in the sense of Definition  \ref{def:measure_valued_sol}. This solution is 
given by {\color{vert} the following everywhere convergent series}
\begin{equation}
\label{representation_solution}
  \mu_t = e^{-\a x^{\gamma}t} \mu_0
  +  \sum_{n=0}^{\infty} (\a t)^n
  \dst\int_0^{\infty} \ell^{n\gamma}a_n\left(\dfrac{x}{\ell}\right)\mu_0(\ell) \dfrac{d\ell}\ell{},
 \end{equation}
 where the sequence $a_n$ is defined as follows for $x\in [0,1]$, 
 \begin{equation}
 \label{induction}
  a_0(x)=0, \qquad a_{n+1}(x) = \dfrac{1}{n+1}\left(-x^{\gamma}a_n(x)+ \dst\int_x^{\infty} y^{\gamma-1} \kappa \left(\dfrac{x}{y} \right)a_n(y)dy
  + \kappa(x) \dfrac{(-1)^n}{n!}\right).
 \end{equation}
{\color{black} In particular,
 \begin{equation}
  \label{weakform}
    \dst\int_0^{\infty} \varphi(x) d\mu_t(x)
    =   \dst\int_0^{L}  \dst\int_0^{1} \varphi(\ell x) d \mu_{t\ell^{\gamma}}^{F} (x)d\mu_0(\ell),\qquad \forall \varphi\in \mathcal{C}(\R^+),
\end{equation} 
and 
\begin{equation}
\supp(\mu_t)\subset \supp(\mu_0), \qquad \text{for all } t>0.
\end{equation}
}
\end{theorem}
}

\begin{remark}
\label{remark1}
We emphasize that in \eqref{induction} and  \eqref{representation_solution}, 
 $\kappa$, $\mu_0$ and $a_n$ may be measures.
 In this case, the Mellin convolution product 
 has to be understood in the sense of measures. 
 For instance, the formula \eqref{representation_solution} means 
 \begin{align}
 &\mu_t =  e^{-\a x^{\gamma}t} \mu_0 
  + \sum_{n=0}^{\infty} (\a t)^n a_n  \ast  b_n \label{remark1E1}\\
 & db_n(y)=y^{n\gamma } d\mu _0(y),  \label{remark1E2}
 \end{align}
 where the Mellin convolution product $\MT{\ast}$ is defined in Definition \ref{def:mult_conv}. 
\end{remark}

\MT{
\begin{proof}[Proof of Theorem \ref{thm:power_serie}]
{\bf Step 1: A power series representation for the fundamental solution.}
In this step, we prove that
\begin{align}
&\mu_t^{F}(x)=e^{-\a t}\delta(x-1) +v_t, \label{shape_u}
\end{align}
where
\begin{align}
v_t= \sum_{n=1}^{\infty} (\a t)^n a_n \in  \mathcal{C}((0,T),\M(\R^+)),  \qquad v_t(0)=0, \qquad v_t\perp \delta (x-1),
\label{serie} 
\end{align}
is a fundamental solution to \eqref{eq:frag}, 
and we prove that for all $t>0$, this solution satisfies
\begin{align}
& \supp(\mu^F_t)\subset[0, 1].
\end{align}
Fix $T>0$. Assume that $\mu_t^F$ is a fundamental solution to \eqref{eq:frag} on $[0,T]\times \R^+$.
We recall that
 $\mu \perp \nu$ if there exists $E \in \mathcal{B}(\R)$ such that 
 $\mu(\R)= \mu( E)$ and $ \nu( E)=0$. 
The Radon-Nikodym decomposition 
guarantees that for all $t>0$, $\mu_t^F$ can be decomposed as 
\begin{equation}
\label{shape_u1}
 \mu_t^{F} = A(t) \delta (x-1) +v_t,
\end{equation}
where $v_t \perp \delta(x-1)$, $A(0)=1$ and $v_0(x)=0$.
 We plug  \eqref{shape_u1} into~\eqref{eq:frag} and get
 \begin{equation*}
  \begin{aligned}
   A'(t) \delta(x-1)+ \p_tv_t(x) &=- A(t) \a x^{\gamma}\delta(x-1)
   - v_t(x) \alpha x^{\gamma} +\\ 
&   + 
   \alpha \dst\int_x^{\infty} y^{\gamma-1} \kappa\left(\dfrac{x}{y} \right)
   \left(A(t) \delta(y-1) +dv_t(y)\right) 
  \end{aligned}
 \end{equation*}
 which is 
 \begin{equation*}
  \begin{aligned}
   A'(t) \delta(x-1)+ \p_tv_t(x) &=- \a A(t) \delta(x-1) 
   - v_t(x) \alpha x^{\gamma} + 
  \alpha \kappa(x) A(t)+  \alpha \dst\int_x^{\infty} y^{\gamma-1} \kappa\left(\dfrac{x}{y} \right)
   dv_t(y) .
  \end{aligned}
 \end{equation*}
By identification, we get that necessarily
\begin{equation}
\label{eq:Aandv}
 \left\{
 \begin{aligned}
  A'(t)&= - \alpha  A(t), \quad A(0)=1, \\
  \p_t v_t(x)&= 
  - \alpha x^{\gamma}v_t(x)+\alpha \dst\int_x^{\infty} y^{\gamma-1} \kappa\left(\dfrac{x}{y}\right)dv_t(y)
  + \alpha \kappa(x) A(t), \quad v_0(x)=0. 
 \end{aligned}
 \right.
\end{equation}
The first line gives $ A(t)= e^{-\a t}$.
This proves that $\mu_t^F$ is necessarily equal to $e^{-\alpha t} +v_t$, where $v_t$ satisfies the second line of \eqref{eq:Aandv}. Now let us verify that the series \eqref{serie} converges in $\mathcal{C}((0,T),\M(\R^+))$.
Since $(\M(\R^+),\|.\|_{TV})$ is a Banach space, it is enough to prove
the normal convergence of the series \eqref{serie}.
We first claim that
\begin{equation}
\label{ineq_serie}
\| a_{n+1} \|_{TV} \leq  \dfrac{1}{n+1} \left( (N+1)\| a_{n} \|_{TV}+ \f{N}{n!} \right),
\end{equation}
This comes directly from the induction formula \eqref{induction} since
$x\in [0,1]$ implies
\begin{equation*}
 \| x^{\gamma} a_{n} \|_{TV} \leq  \|  a_{n} \|_{TV},
\end{equation*}
and
\begin{equation*}
\begin{aligned}
  \| \dst\int_x^{\infty} y^{\gamma-1} \kappa \left(\dfrac{x}{y} \right)a_n(y)dy\|_{TV}
  &= \dst\int_0^{\infty} \left| \dst\int_x^{\infty} y^{\gamma-1} \kappa \left(\dfrac{x}{y} \right)a_n(y)dy \right| dX\\
&\leq \dst\int_0^{\infty}\dst\int_x^{\infty}  \left| y^{\gamma-1} \kappa \left(\dfrac{x}{y} \right)a_n(y)\right|dy  dx\\
&=\dst\int_0^{1}\dst\int_0^{\infty}  \left| y^{\gamma} \kappa \left(z \right)a_n(y)\right|dy  dz\\ 
 &\leq \dst\int_0^{1} \left|\kappa \left(dz \right)\right| \dst\int_0^{\infty} y^\gamma\left|a_n(y)\right|dy  
\leq N \|  a_{n} \|_{TV},
 \end{aligned}
\end{equation*}
and finally
\begin{equation*}
\| \kappa(x) \dfrac{(-1)^n}{n!} \|_{TV} = \f{N}{n!}.
\end{equation*}
We deduce from~\eqref{ineq_serie} that for all $n\in\N,$
\begin{equation}\label{u_n}
\| a_n \|_{TV} \leq \f{(N+2)^n}{n!},
\end{equation}
hence the \Mag{normal} convergence of the series~\eqref{serie} in $(\M(\R^+),\|.\|_{TV})$ for all $t>0$.
We prove~\eqref{u_n} by induction: \eqref{u_n} is true for $n=0$ and $n=1,$ and if it is satisfied for $n\geq 1$ we have
$$\Vert a_{n+1}\Vert_{TV} \leq \f{1}{n+1} \left((N+1)\Vert a_n\Vert_{TV} + \f{N}{n!}\right) \leq \f{1}{(n+1)!} \f{(N+1)(N+2)^{n} + N}{n!} \leq \f{(N+2)^{n+1}}{(n+1)!}.$$
Then the series $v_t$ defined in \eqref{serie} converges in the Banach space $\mathcal{C}((0,T),\M(\R^+)).$
Since from the induction rule  (\ref{induction}) $\supp(a_n)\subset [0, 1]$ for all $n\ge 0$,  it follows that $\supp(v_t)\subset [0, 1]$.
We prove then, using the differentiation rule of power series in a Banach space,
that the power series \eqref{serie} is a solution 
to the second line of \eqref{eq:Aandv}. We have
\begin{equation*}
 \dfrac{d}{dt} v(t,x) =  \dfrac{d}{dt} \sum_{n=1}^\infty (\a t)^n a_n= 
\a \sum_{n=1}^\infty  n  (\a t)^{n-1} a_n 
= \a \sum_{n=0}^\infty (n+1)  (\a t)^n a_{n+1}.
\end{equation*}
Using the induction hypothesis \eqref{induction} we get
\begin{equation*}
\begin{aligned}
 \dfrac{d}{dt} v(t,x) = 
&  \a \sum_{n=0}^\infty  (\a t)^{n}\left(  -x^{\gamma}a_n(x)
+ \dst\int_x^{\infty} y^{\gamma-1} \kappa \left(\dfrac{x}{y} \right)a_n(y)dy
  + \kappa(x) \dfrac{(-1)^n}{n!}\right)\\
  & = - \a  x^{\gamma}  \sum_{n=0}^\infty  (\a t)^{n} a_n 
  +\a \dst\int_x^{\infty} y^{\gamma-1} \kappa \left(\dfrac{x}{y} \right)\left(  \sum_{n=0}^\infty  (\a t)^{n} a_n(y)\right) dy
  +\a   \sum_{n=0}^\infty  (-\a t)^{n}\kappa(x)\\
   & = - \a  x^{\gamma} v(t,x)
  +\a \dst\int_x^{\infty} y^{\gamma-1} \kappa \left(\dfrac{x}{y} \right)v(t,y)dy
  +\a   e^{-\a t} \kappa(x),
\end{aligned}
\end{equation*}
which {\color{vert} is~\eqref{eq:Aandv}.}
The property of the support $\supp(\mu _t^F)\subset [0, 1]$ follows from the hypothesis on the support of $\kappa$ and by inspection of formulas  (\ref{serie}) and (\ref{induction}).
\vspace{0.5cm}
\\
{\bf Step 2: A power series representation of  a solution with a generic initial condition.}
By the classical superposition principle, if  
\begin{equation}
\label{integralmu1}
\mu_t (x)=  \int _0^\infty \mu _0(\ell)\mu _t^\ell\left({\color{vert}{x}} \right)d\ell
\end{equation} 
converges  in $\mathcal{C}((0,T),\M(\R^+))$ 
(where $\mu_t^{\ell} $ is the scaled fundamental solution obtained in Proposition \ref{prop:rescaling} from $\mu_t^F$ the fundamental solution obtained in Step 1),  $\mu _t$ will be  a solution  to the fragmentation equation with 
 initial condition $\mu_0 \in \M(\R^+)$.
Notice that {\color{vert} we have the following equality for} the integral (\ref{integralmu1}): 
\begin{align}
\int_0^{\infty} {\color{vert}\mu_0(\ell) \mu_t^{\ell} \left( {x} \right)} d\ell=e^{-\alpha x^\gamma t}\mu _0+\dst\int_0^{\infty}\sum_{n=0}^{+\infty} (\a t)^n
 \ell^{n\gamma}a_n\left(\dfrac{x}{\ell}\right)\mu_0(\ell) \dfrac{d\ell}\ell{}.\label{seriemu1}
\end{align}
Since for every $n$, by (\ref{u_n})
\begin{align*}
\sum_{n=0}^{m} \left|\left| (\a t)^n
\dst\int_0^{\infty} \ell^{n\gamma}a_n\left(\dfrac{x}{\ell}\right)\mu_0(\ell) \dfrac{d\ell}\ell{}\right|\right| _{ TV }\le  ||\mu _0|| _{ TV }\sum_{n=0}^{m} (\a t)^n||a_n|| _{ TV } \textcolor{black}{L}^{\gamma n} \\
\le  ||\mu _0|| _{ TV } \sum_{n=0}^{m} (\a t)^n \f{(N+2)^n}{n!}\textcolor{black}{L}^{\gamma n} ,
\end{align*}
the series in the right-hand side of (\ref{seriemu1}) converges absolutely in the Banach space $\mathcal{C}((0,T),\M(\R^+))$ for all $T>0$. The integral (\ref{integralmu1}) is then absolutely convergent  and defines a solution to the fragmentation equation with 
 initial condition $\mu_0 \in \M(\R^+)$ and for $t\in[0,T]$.  Property  (\ref{weakform}) follows then from the definition of $\mu^F  _{ t\ell^\gamma}$. Since $\supp(a_n)\subset [0, 1]$ for every $n\ge 0$, it follows from  (\ref{representation_solution}) that $\supp(\mu _t)\subset [0, \textcolor{black}{L}]$ for all $t>0$.
 Using a classical diagonal argument, and since the property on the support of the solution does not depend on $T$, the power series defines a solution in $\mathcal{C}({\R}^+,\M(\R^+))$
 This ends the proof of Theorem \ref{thm:power_serie}].
\end{proof}
}

\MT{
\begin{theorem}[Non negativity of the power series solution]
\label{thm:positivity}
Assume the fragmentation kernel $\kappa$ satisfies \eqref{hyp1} with $\gamma \geq 0$ and take $\mu_0$ that satisfies \eqref{hyp2}. Then, the power series solution \eqref{representation_solution} to the fragmentation equation \eqref{eq:frag} is non-negative.
\end{theorem}
\begin{remark}
Up to our knowledge, no proof of positivity for the fragmentation equation is available in the literature in the case where either the initial condition $\mu_0$ is a measure or the fragmentation kernel $\kappa$ is a measure. In our case, both are measures.
\end{remark}
}
\MT{
\begin{proof}
 We first prove the non negativity of the fundamental solution $\mu_t^F$ defined by \eqref{shape_u} \eqref{serie} using an approximation argument. Consider to this end  the function $\chi$ such that:
\begin{align*}
&\chi\in \mathcal C(0, \infty),\,\,\chi\ge 0,\,\,||\chi||_\infty=1, \quad \chi(x)=1\,\,\,\forall x\in [0, 2L];\quad 
\chi(x)=0\,\,\forall x>3L,
\end{align*}
and a  mollifier
$\theta\in \mathcal C(0, \infty)$, such that $\theta \ge 0$, $\supp \theta\subset [0, 1]$
and the sequence $\theta _m(x)=m\theta (m(x-1))$. Let us then denote by $\kappa_m$ a regularization of the fragmentation kernel
\begin{equation*}
\kappa _m=\kappa \ast \theta_m,
\end{equation*}
by $a$ a regularization of the fragmentation rate
\begin{equation*}
    a(x)=x^{\gamma}\chi(x),
\end{equation*}
and by $\theta_k$ a regularization of the initial condition $\delta(x-1)$.
By construction $||a||_\infty \le (3L)^\gamma$ and for each $m\ge 1$, $\kappa _m$ is a regular function 
satisfying  
\begin{align}
\supp[\kappa_m]\subset [0, 2], \quad \lim_{m\to \infty}||\kappa _m-\kappa||_{BL}=0, \quad 
||\kappa _m||_{TV}\le ||\MT{\kappa}||_{TV}||\theta_m||_{TV}\le N. 
\label{kmBL2}
\end{align}
Consider for every  $m\ge 1$ the sequence of functions $\{a_{n, m}\}_{n\in \N, m\in \N}$
defined as follows, 
\begin{align}
\label{inductionB}
& a_{0, m}(x)=0,\nonumber\\
&a_{n+1, m}(x) = \dfrac{1}{n+1}\left(-a(x)a_{n, m}(x)+ \dst\int_{\MT{0}}^{\infty} a(y) \kappa_m \left(\dfrac{x}{y} \right)a_{n, m}(y)\MT{\dfrac{dy}{y}}
+ \kappa_m(x) \dfrac{(-1)^n}{n!}\right).
\end{align}
 It immediately follows, for all $n\ge 1, m\ge 1$,
 \begin{align*}
 a_{n, m}\in \mathcal C([0, \infty)),\,\,\supp[a_{n, m}]\subset[0, 2]
 \end{align*} 
 and then, the sequence  $\{a_{n, m}\}_{n\in \N, m\in \N}$ satisfies also,
 \begin{align}
\label{inductionC}
a_{n+1, m}(x) = \dfrac{1}{n+1}\left(-x^{\gamma}a_{n, m}(x)+ \dst\int_{\MT{0}}^{\infty} y^{\gamma} \kappa_m \left(\dfrac{x}{y} \right)a_{n, m}(y)\frac{dy}{y}
+ \kappa_m(x) \dfrac{(-1)^n}{n!}\right).
\end{align}
Since $\supp[a_{n, m}]\subset [0, 2]$ it follows that 
\begin{equation*}
    ||x\to x^\gamma a_{n, m}(x)||_{TV}\le 2^{\gamma}||a_{n, m}||_{TV}.
\end{equation*}. 
Then, for every $m\ge 1$ fixed, as for the proof of (\ref{u_n})  it follows now that
 \begin{equation}
 \label{u_nm}
\| a_{n, m} \|_{TV} \leq \f{2^{\gamma n}(N+2)^n}{n!},\,\,\forall n\ge 1, \, \forall m\ge 1.
\end{equation}
\\
{\bf Step 1: The solution $u_{km}$ to the regularized problem is non negative.}
Consider the regularized problem ($m<\infty$ and $k<\infty$)
\begin{equation}
\label{pgcd1}
\left\{
\begin{aligned}
&\dfrac{\p}{\p t} u_{k, m}(t, x) =-\a a(x) u_{k, m}(t, x) + \alpha \dst\int_{x}^{\infty}\kappa_m\left(\dfrac{x}{y} \right)a(y)  u_{k, m}(t, y)\frac{dy}{y},\\
&\MT{u_{k,m}(0,x)=\theta_k(x)}. 
\end{aligned}
\right.
\end{equation}
We define for $k\geq 1$ and $m\geq 1$ the sequence of functions
\begin{align}
\label{pgcd8C}
 u_{k, m}(t, x) = e^{-\a x^{\gamma}t} \theta_k +  
\sum_{n=0}^{\infty} (\a t)^n
\dst\int_0^{\infty} \ell^{n\gamma}a_{n, m}\left(\dfrac{x}{\ell}\right) \theta_k(\ell) \dfrac{d\ell}\ell{},
\end{align} 
For $k\geq 1$ and $m\geq 1$, the series is absolutely convergent since
Moreover, by construction $\supp[u_{k, m}]\subset [0, 2]$. Then, $u_{k, m}$ satisfies \eqref{pgcd1}.
By Lemma 3 of \cite{Melzak}, of which  the equation (\ref{pgcd1}) and the initial data $\theta_k$ satisfy the hypothesis, the Cauchy problem for (\ref{pgcd1}) with initial data $\theta_k$ possesses a global solution bounded, continuous, non negative, analytic in $t$ for each $x>0$ and integrable in $x$ for every $t>0$. Moreover, by construction, for all $T>0$ and $t\in (0, T)$,
\begin{equation*}
    \begin{aligned}
        \dst\int|u_{km}(t,x)| dx& =
      \dst\int_0^{_infty} e^{-\alpha x^{\gamma}t}\theta_k(x)dx+ \dst\int_0^{\infty}\sum_{n=0}^{\infty} (\a t)^n
\dst\int_0^{\infty} \ell^{n\gamma}\Big|a_{n, m}\left(\dfrac{x}{\ell}\right)\Big| \theta_k(\ell) \dfrac{d\ell}\ell dx\\  
&\leq 1+  \sum_{n=0}^{\infty} (\a t)^n 2^{n\gamma} \dfrac{(N+2)^n}{n!}<\infty.
    \end{aligned}
\end{equation*}
 Therefore, by Lemma 4 in \cite{Melzak}, the function $u_{k,m} $ is the unique
 solution of \MT{\eqref{pgcd1}}
 that satisfies   (\ref{mltzk}). 
 \MT{Moreover, this solution is non negative.}
 \vspace{0.5cm}
 \\
 {\bf Step 2: limit $k\to \infty$.}
 Consider the problem ($m<\infty$)
 \begin{equation}\label{eq:fragReg}
 \begin{aligned}
&\dfrac{\p}{\p t} u_{m}(t, x) =-\a a(x) u_{m}(t, x) + \alpha \dst\int_{x}^{\infty}\kappa_m\left(\dfrac{x}{y} \right)a(y)  u_{m}(t, y)\frac{dy}{y},\\ 
&u_{m}(0, x) = \delta (x-1).\
\end{aligned}
\end{equation}
We define the sequence of measures $u_m$ as
\begin{align}
\label{pgcd8}
 u_{m}(t, x) = e^{-\a t} \delta (x-1) +  
\sum_{n=0}^{\infty} (\a t)^n a_{n, m}
\end{align}
The series in (\ref{pgcd8}) is absolutely convergent in $TV$ norm  for every $m\ge 1$ and it defines a measure $u_m\in \mathcal M(\R_+)$ 
such that
\begin{align}
||u_m\MT{(t)}||_{TV}\le \exp( \alpha t 2^{\gamma }(N+2))\MT{+e^{-\alpha t}}.
\end{align}
\\
The measure $u_{m}\MT{(t,.)}$ satisfies \eqref{eq:fragReg}
but since $\supp[u_m(t)]\subset [0, 2]$ it also satisfies,
\begin{align*}
&\dfrac{\p}{\p t} u_{m}(t, x) =-\a  x^{\gamma} u_{m}(t, x) + \alpha \dst\int_{x}^{\infty}\kappa_m\left(\dfrac{x}{y} \right) y^{\gamma-1}  u_{m}(t, y)dy.
\end{align*}
\\
We claim now that $(u_{km})_{k\geq 0}$ converges weakly towards $u_m$.
Indeed, on one hand, for all $\varphi \in \mathcal{C}_C(\R^+)$, 
\begin{equation}
\label{fleur1}
\lim\limits_{k\to \infty} \dst\int_0^{\infty} e^{-\alpha x^{\gamma} t} \theta_k(x) \varphi(x)dx = e^{-\alpha t} \varphi(1),
\end{equation}
and on the other hand, let us
notice that  for each $n\ge 1$ and $m\ge 1$ fixed,
\begin{equation*}
\lim_{k\to \infty} \int_0^{\infty} \ell^{n\gamma}a_{n, m}\left(\dfrac{x}{\ell}\right) \theta_k(\ell) \dfrac{d\ell}\ell{}=a_{n, m}(x), \quad x\in [0,2].
\end{equation*}
Moreover, for all $k\ge 1$, $n\ge 1$ and  $m\ge 1$
\begin{align*}
\int_0^{\infty} \ell^{n\gamma}a_{n, m}\left(\dfrac{x}{\ell}\right) \theta_k(\ell) \dfrac{d\ell}\ell{}=
\int _0^\infty \left(\frac{x}{y}\right)^{n\gamma-1}\!\!\!a_{n, m}(y)k\, \theta\left(k\frac{x}{y}\right) dy.
\end{align*}
Therefore, if
\begin{align*}
\psi _k (z)= z^{n\gamma-1}k \theta (k z)
\end{align*}
then $\psi _k\in \mathcal C([0, \infty)$ and  $\psi _k\ge 0$. Since $\supp \theta \subset [0, 1]$ and $||\theta||_\infty\le 1$, for all $n$ such that $n\gamma >2$,
\begin{align*}
\sup_{z>0}\psi _k(z)=\sup_{z>0 }z^{n\gamma-1}k\theta(k z)=
\sup_{z\in [0, k^{-1}]}z^{n\gamma-2}(kz)\theta(k z)\le 1
\end{align*}
It follows that for all $m\ge 1$, $n\ge 1$, $k\ge 1$ and $x>0$,
\begin{align*}
\left|\int_0^{\infty} \ell^{n\gamma}a_{n, m}\left(\dfrac{x}{\ell}\right) \theta_k(\ell) \dfrac{d\ell}\ell{}\right|
\le ||a_{n, m}||_{TV}
\end{align*}
It follows by the Lebesgue's convergence that 
\begin{equation}
    \label{fleur2}
    \lim\limits_{k\to \infty} \sum_{n=0}^{\infty} (\a t)^n
\dst\int_0^{\infty} \ell^{n\gamma}a_{n, m}\left(\dfrac{x}{\ell}\right) \theta_k(\ell) \dfrac{d\ell}\ell{} = \sum_{n=0}^{\infty} (\a t)^na_{n,m}(x), \quad x\in[0,2],
\end{equation}
and then, combining \eqref{fleur1} with \eqref{fleur2} gives us
\begin{align*}
\lim_{k\to \infty } \int _0^\infty \varphi (x) u_{k, m}(t, x)dx
=\int _0^\infty \varphi (x) u_m(t, x)dx,\quad \varphi\in \mathcal C_C(\R^+),
\end{align*}
and then 
\begin{equation}
    u_m\ge 0.
\end{equation}
\\
 {\bf Step 3: limit $m\to \infty$.}
 We prove here that $u_m$ converges weakly towards $\mu_t^F$.
 To do so, we prove by induction that
 \begin{equation}
     \|a_{n,m}-a_n\|_{BL}\underset{m\to \infty}{\to} 0.
 \end{equation}
For $n=1$, $a_{1, m}=\kappa _m$, $a_1=\kappa$, and by construction   
$||\kappa _m-\kappa ||_{BL} \underset{m\to \infty}{ \to} $.  Assume then  $||a_{n, m}-a_n||_{BL}\to 0$. In order to prove that the same property holds for the sequence $\{a_{n+1, m}\}_{m\in \N}$ it is sufficient to prove
\begin{align}
\label{convmeme}
\lim_{m\to \infty}\left|\left|\int_x^{\infty} y^{\gamma-1} \kappa_m \left(\dfrac{x}{y} \right)a_{n, m}(y)dy
 - \int_x^{\infty} y^{\gamma-1} \kappa \left(\dfrac{x}{y} \right)a_{n}(y)dy\right|\right|_{BL}= 0.
\end{align} 
If, for the sake of notation we define the functions  $\tilde a_{n, m}$ and $\tilde a_{n}$ as
\begin{align*}
&\tilde  a_{n, m}(\ell)=\ell^{\gamma}a_{n, m}(\ell),\qquad
\tilde  a_{n}(\ell)=\ell^{\gamma}a_{n}(\ell),
\end{align*}
then property (\ref{convmeme}) reads,
\begin{align}
\label{convmemeB}
\lim_{m\to \infty}\left|\left|\tilde a_{n, m}\ast \kappa _m-  \tilde a_{n}\ast \kappa \right|\right|_{BL}= 0.
\end{align}
Notice indeed that, for all test function $\varphi $ such that $||\varphi ||_\infty\le 1$ and $||\varphi '||_\infty\le 1$,
\begin{align}
\label{comvmeme2}
\int_0^\infty \varphi (x)
\left (\tilde a_{n, m}\ast \kappa _m(x)- \tilde a_{n}\ast \kappa(x)\right)dx=
\int_0^\infty \varphi (x)
\left (\tilde a_{n, m}- \tilde a_{n}\right)\ast \kappa _m(x)dx+\nonumber\\
+\int_0^\infty \varphi (x)
\left (\kappa _m- \kappa \right)\ast\tilde a_{n}(x)dx
\end{align}
\\
The two terms in the right-hand side of (\ref{comvmeme2}) may be bounded with the same arguments. Consider for example the first.
\begin{align*}
\left|\int_0^\infty \varphi (x)
\left (\tilde a_{n, m}- \tilde a_{n}\right)\ast \kappa _m(x)dx\right|=\left|\int _0^\infty \varphi (x)\int _0^\infty
\left (\tilde a_{n, m}- \tilde a_{n}\right)\left( \frac{x}{y}\right)\kappa _m(y)\frac{dy}{y}dx\right|\\
=\left|\int _0^2 \kappa _m(y) \int _0^\infty
\left (\tilde a_{n, m}- \tilde a_{n}\right)\left( \frac{x}{y}\right)\varphi (x)dx\frac{dy}{y}\right|\\
\le ||\kappa _m||_{ TV }
\sup_{y \in[0, 2]}
\left|\int _0^\infty
\left (\tilde a_{n, m}- \tilde a_{n}\right)\left( z\right)\varphi (zy)dz \right|
\end{align*}
For each $y\in[0, 2]$,
\begin{align*}
\left|\int _0^\infty
\left (\tilde a_{n, m}- \tilde a_{n}\right)\left( z\right)\varphi (zy)dz \right|\le ||\tilde a_{n, m}- \tilde a_{n}||_{BL}
\left(||\varphi ||_\infty+2||\varphi'||_\infty \right)\end{align*}
from where, by (\ref{kmBL2})
\begin{align*}
 \left|\left|
\left (\tilde a_{n, m}- \tilde a_{n}\right)\ast \kappa _m\right|\right|_{BL}
\le 3 N
||\tilde a_{n, m}- \tilde a_{n}||_{BL}.
\end{align*}
A similar arguments shows, using (\ref{u_n}),
\begin{align*}
\left|\left|
\left (\kappa _m- \kappa \right)\ast \tilde a_n \right|\right|_{BL}\le 
3 ||k_m-k||_{BL }||\tilde a_n||_{TV}\le
3 \f{(N+2)^n}{n!}||k_m-k||_{BL } 
\end{align*}
and then, (\ref{convmeme}) holds true.
\\
Now, for  any $\varphi\in \mathcal C^1([0, \infty))$, such that $||\varphi ||_\infty+||\varphi '||_\infty <\infty$, by definition of the measure $u_m$
\begin{align*}
\int _0^\infty u_m(t, x)\varphi (x)dx&=
e^{-\alpha t}\varphi (1)+\int _0^\infty 
\left (  \sum_{n=0}^{\infty} (\a t)^n a_{n, m}(x)\right) \varphi (x)dx\\
&=e^{-\alpha t}\varphi (1)+ \sum_{n=0}^{\infty} (\a t)^n \int _0^\infty a_{n, m}(x) \varphi (x)dx.
\end{align*}
Since for every $n\ge 1$, 
\begin{equation*}
\lim _{ m\to \infty } \int _0^\infty 
a_{n, m}(x) \varphi (x)dx=\int _0^\infty a_{n}(x) \varphi (x)dx
\end{equation*}
and 
\begin{align*}
(\a t)^n \int _0^\infty |a_{n, m}(x) \varphi (x)|dx& \le (\alpha t)^n||\varphi ||_\infty ||a_{n, m}||_{TV}\\
&\le (\alpha t)^n||\varphi ||_\infty\f{2^{\gamma n}(N+2)^n}{n!}
\end{align*}
one has,
\begin{align*}
\lim _{ m\to \infty } \int _0^\infty\varphi (x)u_{m}(t, x)dx&=
e^{-\alpha t}\varphi (1)+ \sum_{n=0}^{\infty} (\a t)^n \int _0^\infty a_{n}(x) \varphi (x)dx=\int _0^\infty\varphi (x)d\mu _t^F
\end{align*}
and it follows 
\begin{equation*}
    \mu _t^F\ge 0.
\end{equation*}
\\
This ends the proof of Theorem \ref{thm:positivity}.
\end{proof}
}

\MT{
The following corollary  follows now easily  from Theorem \ref{thm:WP_BV}, Theorem (\ref{thm:power_serie}), Theorem \ref{thm:positivity} and Proposition (\ref{prop:rescaling}).
\begin{corollary}[Well-posedness of the fragmentation equation]
\label{cor:summary}
Assume the fragmentation kernel $\kappa$ satisfies \eqref{hyp1} with $\gamma \geq 0$ and take $\mu_0$ that satisfies \eqref{hyp2}.
Then, there exists a unique non-negative solution $\mu_t$ to the fragmentation equation in $\mathcal{C}(\R^+,\M(\R^+))$ that has a finite TV norm.
Moreover, this solution satisfies
\begin{equation*}
\supp(\mu_t) \subset \supp(\mu_0)
\end{equation*}
and can be represented as the power series \eqref{representation_solution}.
\end{corollary}
}
Let us provide two cases where we have explicit formulations for the fundamental solution to \eqref{eq:frag} for $\mu_0=\delta(x-1)$.
\begin{example}
For \MT{$\alpha=\gamma=1$} and $\kappa= 2\mathbb{1}_{[0,1]}$, 
we have \cite[\MT{formula 11}]{ZiffMcGrady85}
\begin{equation*}
\mu_t^F(x) = e^{-t} \delta(x-1) + (\MT{2}t+(1-x)t^2)e^{-xt},
\end{equation*}
\end{example}
\begin{example}
For \MT{$\alpha=1,\;\gamma=0,$ and} $\kappa(z)=2 \delta(z-1/2)$ we have \cite[\MT{Proposition 1}]{DvB18}
\begin{equation*}
\mu_t^F(x)=e^{-t} \delta(x-1) + \MT{e^{-t}} \sum_{k=1}^{\infty} \dfrac{(4t)^k}{k!}\delta\left(x-\dfrac{1}{2^k}\right).
\end{equation*}
\end{example}
In both examples, the mass initially located at $x=1$ decreases exponentially with respect to time and 
is teleported on $(0,1)$.

\textcolor{black}{The stability of the solution with respect to the TV norm has been proved in Theorem \ref{thm:WP_BV}. The stability in the BL norm 
is deduced now from the explicit expression provided by Theorem~\ref{thm:power_serie}.}
\begin{theorem}[Stability of the fragmentation equation in $(\mathcal{M}(\R^+),\|.\|_{BL})$]
\label{thm:stabBL}
  Assume $\kappa$ satisfies \eqref{hyp1},  
 $\mu_0 \in \M^+(\R^+)$ satisfies \eqref{hyp2}\comMD{, and moreover either $\gamma\geq 1$ or $\supp (\mu_0)\subset [m,M]$ with $m>0$}.
Then the unique solution $\mu_t$ to the fragmentation equation \eqref{eq:frag} satisfies
\begin{equation*}
 \|\mu_t\|_{BL}
 \leq C(L,N,T,\alpha,\MT{\gamma,m^{\gamma-1}}) \; \|\mu_0\|_{BL}, \qquad 0 \leq t \leq T,
 \end{equation*}
\end{theorem}
\begin{proof}
We use the definition of the BL norm given by~\eqref{def:BL} and the representation of the solution provided in Theorem \ref{thm:power_serie}. Take $\varphi \in \mathcal{C}(\R^+)$ such that 
 $\|\varphi\|_{\infty}\leq 1$ and $\|\varphi'\|_{\infty}\leq 1$.
 Then by  (\ref{weakform}) in Theorem \ref{thm:power_serie},
 \begin{equation*}
 \dst\int_0^{+\infty} \varphi(x)d\mu_t(x) =  \dst\int_0^{\textcolor{black}{}} \dst\int_0^{1} \varphi(\ell x) d\mu_{t \ell^{\gamma}}^{F}(x) d\mu_0(\ell).
 \end{equation*}
For $\ell \leq M,$ we set 
\begin{equation*}
 \Psi(\ell)= \dst\int_0^{1} \varphi(\ell x) d\mu_{t \ell^{\gamma}}^{F}(x).
\end{equation*}
We notice that for any $r\geq0$, the moment of order $r$ of the absolute value of the fundamental solution $\mu_t^F$ is uniformly bounded for $t\in[0,T]$ using the rough estimate based on Theorem \ref{thm:WP_BV}
\begin{equation*}
 \dst\int_{\R^+} x^rd|\mu_t^F|(x) \leq \textcolor{black}{L}^r \|\mu_t^F\|_{TV} \leq
 \textcolor{black}{L}^r e^{\alpha \textcolor{black}{L}^{\gamma} (N+1)T}\|\delta(x-1)\|_{TV} =:\tilde{C}(\textcolor{black}{L},N,T,r,\MT{\alpha},\MT{\gamma}).
\end{equation*}
Then for all $\ell \leq \textcolor{black}{L}$,
\begin{equation*}
 \left|\Psi(\ell)\right| \leq \dst\int_0^{+\infty} |\varphi(\ell x)| d|\mu_{t \ell^{\gamma}}^{F}|(x) \leq \|\varphi\|_{\infty} \dst\int_0^{+\infty} d|\mu_{T \textcolor{black}{L}^{\gamma}}^{F}|(x) \leq \tilde{C}(\textcolor{black}{L},N,T,0,\MT{\alpha},\MT{\gamma}) ,
\end{equation*}
and
\begin{equation*}
\begin{aligned}
 \left|\Psi'(\ell)\right| &\leq \dst\int_0^{+\infty}  |\varphi^\inserted{'}(\ell x)|x  d|\mu_{t \ell^{\gamma}}^{F}|(x) + \dst\int_0^{+\infty}  |\varphi(\ell x)| t \gamma \ell^{\gamma-1} \left|\dfrac{\p}{\p t} d|\mu_{t \ell^{\gamma}}^{F}|(x) \right|\\ 
 \end{aligned}
\end{equation*}
where
\begin{equation*}
 \dst\int_0^{+\infty}  |\varphi^{\inserted{'}}(\ell x)|x  d|\mu_{t \ell^{\gamma}}^{F}|(x) \leq \tilde{C}(\textcolor{black}{L},N,T,1,\MT{\alpha},\MT{\gamma})\|\varphi'\|_{\infty},
\end{equation*}
and where 
\begin{equation*}
\begin{aligned}
 \int\limits_0^{\infty}  |\varphi(\ell x)| t \gamma \ell^{\gamma-1}& \left|\dfrac{\p}{\p t} d|\mu_{t \ell^{\gamma}}^{F}|(x) \right| \\
 & \leq  \|\varphi\|_{\infty} T\comMD{\max(\textcolor{black}{L}^{\gamma-1},m^{\gamma-1})} \gamma \alpha \left(\int_0^{\infty}  x^{\gamma} d|\mu_{t\ell^{\gamma}}^{F}|(x)  +  \int\limits_0^{\infty} 
 \int\limits_{x}^{\infty}\kappa\left(\dfrac{x}{y} \right) y^{\gamma-1} d|\mu_{t \ell^{\gamma}}^{F}|
(y)dx \right)
 \\
 &=\|\varphi\|_{\infty} T\comMD{\max(\textcolor{black}{L}^{\gamma-1},m^{\gamma-1})}\gamma \alpha\left(\dst\int_0^{\infty}   x^{\gamma} d|\mu_{t\ell^{\gamma}}^{F}|(x)  +\dst\int_0^{\infty} 
 \dst\int_{0}^{1}\kappa\left(z \right) dz y^{\gamma} d|\mu_{t \ell^{\gamma}}^{F}|
(y)\right)\\
&\leq T\comMD{\max(\textcolor{black}{L}^{\gamma-1},m^{\gamma-1})}\gamma \alpha (N+1) \tilde{C}(\textcolor{black}{L},N,T,\gamma,\MT{\alpha},\MT{\gamma}).
 \end{aligned}
\end{equation*}
We set $C(\textcolor{black}{L},N,T,\alpha,\MT{\gamma,m^{\gamma-1}}):=\tilde{C}(\textcolor{black}{L},N,T,0,\MT{\alpha},\MT{\gamma})+\tilde{C}(\textcolor{black}{L},N,T,1,\MT{\alpha},\MT{\gamma})+ T\comMD{\max(\textcolor{black}{L}^{\gamma-1},m^{\gamma-1})} \gamma \alpha (N+1) \tilde{C}(\textcolor{black}{L},N,T,\gamma,\MT{\alpha},\MT{\gamma})$ and define
 \begin{equation*}
\tilde{\Psi} (\ell) =  \dfrac{\Psi(\ell)}{C(\textcolor{black}{L},N,T,\alpha,\MT{\gamma,m^{\gamma-1}})},  
 \end{equation*}
then 
\begin{equation*}
  \|\tilde{\Psi}\|_{\infty}  \leq 1, \qquad   \|\tilde{\Psi}'\|_{\infty}  \leq 1.
\end{equation*}
We have shown that for any $\varphi \in \mathcal{C}(\R^+)$ satisfying 
$\|\varphi\|_{\infty}  \leq 1, \|\varphi'\|_{\infty}  \leq 1$, there exists 
$\tilde{\Psi}  \in \mathcal{C}(\R^+)$ such that  $\|\tilde{\Psi}\|_{\infty}  \leq 1, \|\tilde{\Psi}'\|_{\infty}  \leq 1$ and 
\begin{equation*}
\dst\int_0^{+\infty} \varphi(x) d\mu_t(x) \leq C(\textcolor{black}{L},N,T,\alpha,\MT{\gamma,m^{\gamma-1}})\dst\int_0^{+\infty} \tilde{\Psi}(x) d\mu_0(x).
\end{equation*}
Thus the conclusion of Theorem \ref{thm:stabBL} holds.
\end{proof}

\Mag{\begin{remark}
 For $\gamma<1$, and for any initial condition $\mu_0$ such that
 $\mu_0(0) \neq 0$, \MT{$m=0$ and thus Theorem \eqref{thm:stabBL} does not provide any estimate on $\|\mu_t\|_{BL}$}. Stability with respect to the initial condition is lost.
\end{remark}
}
\textcolor{black}{
\section{Inverse problem for the fragmentation kernel}
\label{sec:short}
}

\textcolor{black}{
In this section, estimates of the fragmentation kernel and bounds of the  error of such estimates are obtained using  the series expression of the solution $\mu$ of \eqref{eq:frag} provided by Theorem~\ref{thm:power_serie} for short values of the time variable. 
}

\textcolor{black}{
\subsection{An estimation for $\kappa$ using short-time measurements}}
Let us first investigate the best possible case, when the initial data $\mu _0$ is a Dirac delta function at $x=1.$
\begin{theorem}[An estimate for $\kappa$ using short-time measurements of the particles size distribution when initial condition is a delta function at $x=1$.]
\label{thm:short}
Assume $\kappa$ satisfies \eqref{hyp1}
and define 
\begin{equation}
\label{k_est_F}
 \kappa^{est}(t)= \dfrac{\mu_t^{F} - e^{-\a t} \delta(x-1) }{\a t},
\end{equation}
where  $\mu_t^{F}$ is the unique fundamental solution to the fragmentation 
equation~\eqref{eq:frag} with the initial data
$\mu_0 = \delta(x-1)$.
Then we have
\begin{equation}
 \Big\|\kappa^{est}-\kappa\Big\|_{TV} \leq \MT{C(N,T,\alpha)}\; t,\,\,\forall t\in (0, T],  \label{estkappaC}
\end{equation}
for
\begin{equation}
 C= \a \max_{ t \in \left[0,T\right]}\sum\limits_{n=0}^{\infty} (\a t)^n \|a_{n+2}\|_{TV}
 \MT{= \a \sum\limits_{n=0}^{\infty} (\a T)^n \|a_{n+2}\|_{TV}}.
 \label{constantcT}
 \end{equation}
 \end{theorem}

Before proving Theorem \ref{thm:short}, we point out that another possible formula for the estimated kernel is 
\begin{equation*}
  \kappa^{est}_{bis}(t)= \dfrac{\mu_t^{F} - (1-\alpha  t)\delta(x-1) }{\a t}= 1 +\dfrac{\mu_t^{F} - \delta(x-1) }{\a t}.
\end{equation*}
Since $e^{-\alpha t} =1 - \alpha t +o(t)$, 
we also have
\begin{equation*}
 \Big\|\kappa_{bis}^{est}-\kappa\Big\|_{TV} \leq C t. 
\end{equation*}
\begin{proof}
We have, using the notations introduced in Lemma \ref{lemma:shape_u} and Proposition\ref{prop:rescaling},
\begin{equation*}
 \begin{aligned}
  \dfrac{\mu_t^{F} -e^{-\a t} \delta(x-1)}{\a t} -\kappa 
  = \dfrac{\sum\limits_{n=1}^{\infty} (\a t)^{n} a_n }{\a t}- \kappa
  = \sum\limits_{n=1}^{\infty} (\a t)^{n-1} a_{n} -\kappa
 = \sum\limits_{n=0}^{\infty} (\a t)^n a_{n+1} -\kappa 
\end{aligned}
\end{equation*}
 and 
since $a_1= \kappa$, we have

 \begin{equation*}
 \begin{aligned}
 \sum\limits_{n=0}^{\infty} (\a t)^n a_{n+1} -\kappa 
 = \sum\limits_{n=1}^{\infty} (\a t)^n a_{n+1}
 =
\a t \sum\limits_{n=0}^{\infty} (\a t)^n a_{n+2}.
\end{aligned}
\end{equation*}

Thus 
\begin{equation*}
 \|  \dfrac{\mu_t^{F} -e^{-\a t} \delta(x-1)}{\a t} -\kappa \|_{TV} 
 \leq  \a t \sum\limits_{n=0}^{\infty}( \a t)^n \|a_{n+2}\|_{TV}.
\end{equation*}
The series converges (normal convergence) and thus it is bounded on any compact set, for instance for $t\in [0,T]$.
This ends the proof of Theorem \ref{thm:short}.
\end{proof}
\textcolor{black}{When the initial data is a Dirac delta at $x=\ell>0$ Theorem~\ref{thm:short} and Proposition~\ref{prop:rescaling} give an estimate of the following  rescaled fragmentation kernel,
\begin{equation*}
\kappa_{\ell} = T_{\ell} \#\kappa,
\end{equation*}
where the map $T_{\ell}$ is defined in \eqref{Tell}.
Recall that if $\kappa$ is a function, then 
\begin{equation*}
\kappa_{\ell} (z) = \dfrac{1}{\ell}\kappa\left(\dfrac{z}{\ell}\right), \qquad 0\leq z\leq \ell.
\end{equation*}}
\begin{corollary}[An estimate for $\kappa$ using short-time measurements of the particles size distribution when initial condition is a delta function at $x=\ell$]
\label{cor:rescaling}
We define
\begin{equation*}
 \kappa_{\ell}^{est}(t)= \dfrac{\mu_t^{\ell} - e^{-\a t\ell^{\gamma}} \delta(x-\ell) }{\a t \ell^{\gamma}}, 
\end{equation*}
where $\mu_t^{\ell}$ is the unique solution to \eqref{eq:frag} with the initial condition
$\mu_0 = \delta(x-\ell)$. Then, for all $T>0$,
\begin{equation*}
 \Big\| \kappa_{\ell}^{est}(t)-\kappa_{\ell}\Big\|_{TV} \leq C t \ell^{\gamma},\,\,\forall t\in (0, T]. 
\end{equation*}
 where $C$ is the constant given in (\ref{constantcT})
\end{corollary}

\begin{proof}
 We notice that for any continuous map $T$,
 we have $\|T\# \mu\|_{TV} \leq \|\mu\|_{TV} $ (with equality if the measure $\mu$ is positive, 
 or if $T$ is \MT{an injection}).
 Let us set $\eta= \mu_{t\ell^{\gamma}}^{F} - e^{-\a t\ell^{\gamma}} \delta(x-1) $. 
 We have 
 $T_{\ell} \#\eta=\mu_t^{\ell}-e^{-\a t\ell^{\gamma}} \delta(x-\ell) $, hence using Theorem \ref{thm:short}
 \begin{equation*}
 \Big\|\dfrac{\mu_t^{\ell} - e^{-\a t\ell^{\gamma}} \delta(x-\ell) }{\a t \ell^{\gamma}}-\kappa_{\ell}\Big\|_{TV} 
 = \Big\| T_{\ell} \# \Big( \dfrac{\mu_{t\ell^{\gamma}}^{F} - e^{-\a t\ell^{\gamma}} \delta(x-1) }{\a t \ell^{\gamma}}-\kappa\Big)\Big\|_{TV} \leq C \a \ell^{\gamma} t, 
 \end{equation*}
 where $C$ is the constant in (\ref{constantcT}).
 This ends the proof. 
\end{proof}

\textcolor{black}{In most of the cases, a Dirac delta as an initial condition is experimentally out of reach. However, as proved in the next corollary, for all initial data $\mu _0$ satisfying  \eqref{hyp2}, it is possible to estimate not the kernel $\kappa$ itself but the convolution  $\kappa \ast w _0$ where $dw_0(\ell)=\ell^\gamma d\mu _0(\ell)$. Moreover, if the initial data $\mu _0$ becomes closer and closer, in some  suitable sense, to $\delta (x-1)$, so does $\kappa \ast w _0$ and the estimate of $\kappa \ast w_0$ gives an estimate of $\kappa $ itself.}

\textcolor{black}{If $\mu _0$ satisfies   \eqref{hyp2} and  $\mu $ is  the unique solution given by  Theorem \ref{thm:power_serie}  of the equation  \eqref{eq:frag} with initial data $\mu _0$, define
\begin{equation}\label{def:kappaestmu0}
\kappa^{est}(\mu _0; t, x)=\dfrac{\mu (t,x) - e^{-\a tx^{\gamma}}\mu _0(x)}{\a t}.
\end{equation}
We have the following corollary.}
\begin{corollary}[Generic initial condition]\label{cor:generic}
Assume $\kappa$ \textcolor{black}{satisfies $\eqref{hyp1}$ and} $\mu_0$ 
\MT{satisfies \eqref{hyp2}}.
\textcolor{black}{ Then, for all $T>0$, 
\begin{equation}
\label{cor:generic1}
\Big\|{\kappa^{est}(\mu _0; t)}-w_0\ast\kappa(x)\Big\|_{TV} \leq  C\textcolor{black}{L}^{2\gamma }\|\MT{\mu_0}\|_{ TV}\, t,\,\,\forall t\in (0, T], 
\end{equation}
where 
$w_0$ denotes the measure with density $\ell \to \ell^{\gamma}\mu _0(\ell)$, $C$ is given in (\ref{constantcT}) and $\kappa^{est}(\mu_0;t)$ is defined by~\eqref{def:kappaestmu0}.}

 \textcolor{black}{ If $\{\mu  _{ 0, n }\} _{ n\in \N }\subset \mathcal{M^+}(\R^+)$ is a sequence such that 
\begin{equation}
\left\{
\begin{aligned}
\label{mcd1}
&\lim _{  n\to \infty}||\mu _{0, n}-\delta (x-1)||_{BL}=0\\
&\sup_{n\in \N}||\mu _{0, n}||_{TV}<\infty
\end{aligned}
\right.
\end{equation}
or if 
\begin{align}
\label{mcd2}
\lim _{  n\to \infty}||\mu _{0, n}-\delta (x-1)||_{TV}=0,
\end{align}
then for all $\varepsilon >0$ there exists $n_0$ such that for all $T>0$,
\begin{equation}
\label{cor:generic2}
 \Big\|{\kappa^{est}(\mu  _{ 0, n }; t)}- \kappa(x)\Big\|_{BL} \leq C\textcolor{black}{L}^{2\gamma}
 \sup_{n\in \N}||\mu _{0, n}||_{TV}\,t+\varepsilon,\quad\forall t\in (0, T],\quad \forall n>n_0 . 
\end{equation}
}
\end{corollary}

\begin{proof}
For $\ell>0$, we multiply the measure 
\begin{equation*}
 X_{\ell}= \dfrac{\mu_t^{\ell} - e^{-\a t\ell^{\gamma}} \delta(x-\ell) }{\a t \ell^{\gamma}} - \kappa_{\ell}
\end{equation*}
by the smooth function $\ell \to \ell^{\gamma}$, and apply 
Corollary \ref{cor:rescaling} to obtain
\begin{equation*}
 \|Y_{\ell}\|_{TV}\leq C t \ell^{2\gamma},
\end{equation*}
with $C$ the constant given in (\ref{constantcT}) and
$Y_{\ell}= \dfrac{\mu_t^{\ell} - e^{-\a t\ell^{\gamma}} \delta(x-\ell) }{\a t} -  \ell^{\gamma}\kappa_{\ell}$.
We multiply the function $\ell \to Y_{\ell}$ from
$\R^+$ onto $\M(\R^+)$ by $\ell \to \mu_0(\ell)$ and 
integrate over $\R^+$. Since $\left(\M(\R^+), \|.\|_{TV}\right)$ is a Banach space, 
we can use the Bochner integral so that we have
\begin{equation*}\begin{array}{ll}
 \Big\|\Mig{\kappa^{est}(\mu _0; t)}-w_0\ast\kappa (x)\Big\|_{TV}
 &=  \Big\| \dst\int_{\R^+} Y_{\ell} d\mu _0(\ell) \Big\|_{TV} \\ \\
& \leq \dst\int_{\R^+} \|Y_{\ell}\|_{TV} d\mu _0(\ell) 
 \le  t \dst\int_{\R^+}\ell^{2\gamma} d \mu _0(\ell),
\end{array}\end{equation*}
and (\ref{cor:generic1}) follows.

\textcolor{black}{
If we suppose now that $\{\mu _{0, n}\}_{n\in \N}$ satisfies (\ref{mcd1}), then so does $\{w _{0, n}\}_{n\in \N}$ and therefore
$||w_{0, N}\ast \kappa -\kappa ||_{BL} \underset{n\to \infty}{ \longrightarrow}0$.
We deduce by (\ref{BL_TV}) and (\ref{cor:generic1}),
\begin{align*}
|| \kappa^{est}(\mu  _{ 0, n }; t)-\kappa || _{ BL }& \le || \kappa^{est}(\mu  _{ 0, n }; t)-w_{0, n}\ast \kappa || _{ TV }+||\kappa -\kappa \ast w _{ 0, n }|| _{ BL }\\
&\le C\textcolor{black}{L}^{2\gamma }||\mu _{ 0, n }|| _{ TV }t +||\kappa -\kappa \ast w _{ 0, n }|| _{ BL }
\end{align*}
from where (\ref{cor:generic2}) follows.}

\textcolor{black}{Now if $\{\mu _{0, n}\}_{n\in \N}$ satisfies (\eqref{mcd2}, $\{w _{0, n}\}_{n\in \N}$ does too and therefore
\begin{equation*}
\lim _{ n\to \infty  }||\kappa -\kappa \ast w _{ 0, n }|| _{ TV }=0.
\end{equation*}
We deduce,
\begin{align*}
|| \kappa^{est}(\mu  _{ 0, n }; t)-\kappa || _{ TV } \le C \textcolor{black}{L}^{2\gamma }||\mu _{ 0, n }|| _{ TV }t +||\kappa -\kappa \ast w _{ 0, n }|| _{ TV }\\
\end{align*}
and (\ref{cor:generic2}) follows.}
\end{proof}

\begin{remark}
\textcolor{black}{If $\{\mu _{0, n}\}_{n\in \N}$ is such that
\begin{equation*}
\begin{aligned}
& \mu _{0, n}  \underset{n\to \infty}{ \rightharpoonup}\delta (x-1),\,\,\hbox{in the weak sense of measures},\\
&\sup_{n\in \N}||\mu _{0, n}||_{TV}<\infty,\\
&\exists\,\,Q\subset [0, \infty),\,\hbox{compact}\; ;\,\,\supp \mu _{0, n}\subset Q,\,\,\forall n\in \N,
\end{aligned}
\end{equation*}
then, by  Proposition 4 in \cite{Hanin},
$||w_{0, n}-\delta (x-1)||_{BL}\underset{n\to \infty}{ \longrightarrow 0}$. It follows that Property (\ref{mcd1}) is satisfied and (\ref{cor:generic2}) holds. Notice however that property (\ref{mcd1}) is not satisfied {\color{vert} for any weakly-converging sequence $\mu_{0,n}$.}}
\end{remark}
\inserted{}

\subsection{Stability of the $\kappa$ estimate {color{vert} with respect to model and measurement noises}.} 
Let us now turn to error estimates in more realistic observation cases, where the noise may be twofold: 1/ a model noise, where the initial condition is close to a Dirac delta in the BL distance; and 2/ a measurement noise, where the size distributions $\mu_0$ and $\mu_t$ are observed with an error. \textcolor{black}{A stability result for the time-dependent solution with respect to the initial condition $\mu _0$ has already been proved in Theorem \ref{thm:stabBL}.}

\begin{theorem}
[Stability of the $\kappa$ estimate with respect to {\color{vert}noises on} the initial condition {\color{vert} and the measurements}]
\label{thm:stab}
Assume $\kappa$ satisfies \eqref{hyp1}.
Take an initial condition $\mu_0^{q}$ satisfying \eqref{hyp2}
and that is close to a delta function at $x=1$ in the sense that
\begin{equation*}
 \|\mu_0^{q} - \delta(x-1)\|_{BL} \leq q.
\end{equation*}
We denote by $\mu_t^q$ the unique solution to the fragmentation equation 
\eqref{eq:frag} with initial condition $\mu_0^{q}$.
Consider the noisy measurements $\mu_0^{q,\eps_0}$ and 
 $\mu_t^{q,\eps}$ of the respective measures $\mu_0^q$ and $\mu_t^q$
 such that 
 \begin{equation*}
 \|\mu_0^{q,\eps_0}- \mu_0^q\|_{BL} \leq \eps_0,\qquad 
   \|\mu_t^{q,\eps}- \mu_t^q\|_{BL} \leq \eps.
 \end{equation*}
 \Mag{Assume moreover either $\gamma\geq 1$ or $\supp (\mu_0)\subset [m,\textcolor{black}{L}]$ with $m>0$}.
Then, for all $0\leq t\leq T$, there are some constants $\MT{C_1(N,T,\alpha)}$ and $\MT{C_2(L,N,T,\alpha,\MT{\gamma,m^{\gamma-1}})}$ such that 
\begin{equation}
\label{thm:stabE4}
 \Big\|\dfrac{\mu_t^{q,\eps} - e^{-\a t} \mu_0^{q,\eps_0} }{\a t}-\kappa\Big\|_{BL} 
 \leq \MT{C}_1  t + \dfrac{\eps_0+\eps+\MT{C}_2 q}{\a t}. 
\end{equation}
\end{theorem}
\begin{proof}
 We use the triangle inequality to write 
 \begin{equation}
 \label{thm:stabE7}
 \begin{aligned}
 \Big\|\dfrac{\mu_t^{q,\eps} - e^{-\a t} \mu_0^{q,\eps_0} }{\a t}-\kappa\Big\|_{BL} 
 \leq  &\dfrac{\|\mu_t^{q,\eps} -\mu_t^{q}\|_{BL}}{\a t}
 +\dfrac{\|\mu_t^{q} -\mu_t^{\MT{F}}\|_{BL}}{\a t}
 +e^{-\a t}\dfrac{\|\delta(x-1) -\mu_0^q\|_{BL}}{\a t}
 \\
 &e^{-\a t}\dfrac{ \|\mu_0^q- \mu_0^{q,\eps_0}\|_{BL}}{\a t}
 + \Big\|\dfrac{\mu_t - e^{-\a t} \delta(x-1) }{\a t}-\kappa\Big\|_{BL}.
 \end{aligned}
\end{equation}
The first, third and  fourth terms in the right-hand side of (\ref{thm:stabE7}) are directly controlled
using the assumptions of Theorem \ref{thm:stab}.In the last term at the right-hand side of (\ref{thm:stabE7}),
Theorem \ref{thm:short} combined with \eqref{BL_TV}
guarantee that 
\begin{equation*}
  \Big\|\dfrac{\mu_t - e^{-\a t} \delta(x-1) }{\a t}-\kappa\Big\|_{BL}\leq \MT{C(N,T,\alpha)}\; t.
\end{equation*}
For the second term, we use Theorem \ref{thm:stabBL} to obtain
\begin{equation*}
\|\mu_t^{q} -\mu_t\|_{BL}
\leq \MT{ C(L,N,T,\alpha,\MT{\gamma,m^{\gamma-1}}) \|\mu_0^q-\delta(x-1)\|_{BL}}
\end{equation*}
Thus with the assumptions of Theorem \ref{thm:stab},
 we obtain 
 \begin{equation*}
\|\mu_t^{q} -\mu_t\|_{BL}
\leq  \MT{ C(L,N,T,\alpha,\MT{\gamma,m^{\gamma-1}}) } q.
\end{equation*}
This completes the proof of Theorem \ref{thm:stab} with $\MT{C_1= \MT{C(N,T,\alpha)}}$ and $\MT{C_2=1+C(L,N,T,\alpha,\MT{\gamma,m^{\gamma-1}}).}$
\end{proof}
\MT{
\begin{remark}
\label{rem:t_optimal}
\textcolor{black}{{\color{vert} We notice that~\eqref{thm:stabE4} presents a balance between two terms, which is classically encountered in the field of inverse problems~\cite{Engl} and which is also reminiscent of the classical bias-variance tradeoff in nonparametric statistics~\cite{GineNickl2016}. The time interval $t$ plays the same role as a regularisation parameter: if too small, the noise is not smoothed and the right-hand side of~\eqref{thm:stabE4} tends to infinity; if too large, the estimate loses its accuracy, the right-hand side being not small. There} is a time $t^*$ such that the estimate provided by Theorem \eqref{thm:stab} is optimal, namely
\begin{equation}
\label{t_opt}
t^*=\sqrt{\dfrac{\eps_0+\eps+C_2q}{\alpha C_1}}.
\end{equation}
{\color{vert} For this value, the error estimate is in the order of $\sqrt{\eps_0+\eps+ q}$, vanishing when the noise levels vanish, though at a lower speed than the noises themselves - the rate of convergence in the order of $\sqrt{\eps}$ being reminiscent of mildly ill-posed problems.}}
\end{remark}}

{\color{vert}\begin{remark}
    Using short-time measurements to estimate parameters of a given time-dependent equation is an idea that has appeared for other types of equations. Recently, a very similar approach has been used for estimating the tumbling kernel of a mesoscopic equation for chemotaxis~\cite{hellmuth2022kinetic}; in their approach, convergence of their estimate is obtained, but no quantitative error estimate as~\eqref{thm:stabE4}. Further away from our equation, it has been used to estimate the exponent of a time-fractional diffusion equation~\cite{li2019inverse}, or yet the diffusion parameter in the heat equation~\cite{cao2006natural}. However, up to our knowledge, no systematic approach which would analyse the "short-time method" in a general framework, and which would justify our analogy of the time window of the observation  with a regularisation parameter, has yet been developed.
\end{remark}
}

\section{Reconstruction formula \textcolor{black}{in Mellin variables}}
\label{sec:reco}
 We have seen in the previous section how to approximate $\kappa$  when the initial condition is not too far from a Dirac measure, and {\color{vert} how to approximate $w_0 * \kappa$ by } $\kappa^{est}(\mu _0;.)$ {\color{vert} for generic initial condition}. \textcolor{black}{This section is devoted to the deduction of a reconstruction formula for the Mellin transform of the fragmentation kernel $\kappa$ in the case of generic initial condition, and to estimate the error  of such an approximation, in terms of short-time measurements of the population data and the initial data. The best method to this end is not to use the Mellin transform of the approximation $\kappa^{est}(\mu_0;.)$ of $\kappa$ obtained in Corollary \ref{cor:generic}.  Instead, the series representation of $\mu_t$  is used to deduce a series representation of its Mellin transform $U$, and   then  an approximation of the Mellin transform directly.}
 
\textcolor{black}{Suppose that $\kappa $ satisfies   \eqref{hyp1}, $\mu _0$ satisfies    \eqref{hyp2} and let $\mu $ be the solution to \eqref{eq:frag} with initial condition $\mu_0$ given by Theorem \ref{thm:power_serie} . We denote by $U(t,.)$ the  $x$-Mellin transform  of $\mu_t$ to \eqref{eq:frag}, and we denote by $K$ the Mellin transform of $\kappa$, {\it i.e.}
 \begin{equation*}
 U(t,s) = \dst\int_0^{+\infty} x^{s-1} d\mu_t(x), \quad K(s)= \dst\int_0^{+\infty} z^{s-1} d\kappa(z).
 \end{equation*}
We also define
\begin{equation*}{ W(t,h,s) = \dst\int_0^{+\infty} x^{s-1}e^{-\alpha h x^{\gamma}} d\mu_t(x)},\,\,\forall t\geq0,\,\,\forall h> 0.
\end{equation*}
It follows from \eqref{hyp1} and Theorem \ref{thm:power_serie} that $K$ is   analytic in $s\in S_1^+= \{s\in \C;\,\,\Re e(s)> 1\}$ and so are  $U(t)$ and $W(t, h)$ for all $t>0$ and $h>0$.
 }
 \subsection{A formula for $U$}
\begin{lemma}[Representation of $U$ as a power series]
\label{lemma:MellinPS}
Take $\kappa$ satisfying \eqref{hyp1} and $\mu_0\in \M(\R^+)$
\textcolor{black}{that satisfies \eqref{hyp2}}.
Then, the Mellin transform $U$ of the solution $\mu_t$ to \eqref{eq:frag}
satisfies
\begin{equation}
\label{U_power_serie_inter}
\begin{aligned}
 U(t+\Delta \tau , s) =& 
\Mag{W(t,\Delta \tau , s) }
 +  \sum_{n=1}^{\infty} \dfrac{(\a \Delta \tau )^n}{n!}U(t,s+\gamma n )\times \\
 &\times \sum_{j=0}^{n-1}(-1)^{n-1-j}K(s+j\gamma) 
 \prod_{m=0}^{j-1} (K(s+m\gamma)-1),
\end{aligned}
\end{equation}
for $t>0$ and $\Delta \tau >0$, with the convention 
\begin{equation*}
\prod_{n\in \emptyset} b_n=1. 
\end{equation*}
\end{lemma}
\begin{proof}
\Mag{
 Since the fragmentation equation is autonomous, Theorem 
 \ref{thm:power_serie} implies that for all $t>0$, $\Delta \tau >0$, 
 we have
 \begin{equation*}
  \mu_{t+\Delta \tau }= e^{-\a x^{\gamma}\Delta \tau } \mu_t
  +  \sum_{n=0}^{\infty} (\a \Delta \tau )^n
  \dst\int_0^{\infty} \ell^{n\gamma}a_n\left(\dfrac{x}{\ell}\right)\mu_t(\ell) \dfrac{d\ell}\ell{},
 \end{equation*}
}

We apply the Mellin transform to both sides of the above equality and use Proposition \ref{prop:MellinMC}: it follows
\begin{equation*}
 U(t+\Delta \tau ,s) = \Mag{W(t,\Delta \tau , s) }
  +  \sum_{n=0}^{\infty}(\a\Delta \tau )^n U(t,s+n\gamma)A_n(s),
\end{equation*}
where we denote by $A_n$ the Mellin transform of the measure $a_n$. 
Passing \eqref{induction} into the Mellin coordinates, 
the sequence $A_n$ satisfies 
\begin{equation*}
 A_0=0, \quad A_{n+1}(s) = \dfrac{1}{n+1} \left(( K(s)-1) A_n(s+\gamma) +\dfrac{(-1)^n}{n!}K(s) \right) .
\end{equation*}
By induction, 
we deduce 
\begin{equation*}
 A_n(s) = \dfrac{1}{n!}\left((-1)^{n-1}K(s) +  \sum_{j=1}^{n-1}(-1)^{n-1-j}K(s+j\gamma) 
 \prod_{m=0}^{j-1} (K(s+m\gamma)-1) \right),
\end{equation*}
and 
Lemma \ref{lemma:MellinPS} is proved.
    \end{proof}

\subsection{A reconstruction for $K$ using short times}
    
    \textcolor{black}{Since $\kappa$ {is} supported on $[0,1]$, {it follows that} $K(s+n\gamma) \to 0$ as $n\to \infty$,  {and then}  an approximation formula for $K$ {may be obtained} by {truncation} at $n=1$ of the second term at the right-hand side of (\ref{U_power_serie_inter}). To this end, let us give the following definitions.}
 \begin{definition}[Approximation formula for the Mellin transform of the kernel]
\label{def:Kest}
For $s\in \C$, we denote
 \begin{align}
& K^{est}(s,t,\Delta \tau ) 
=\dfrac{U(t+\Delta \tau, s) - W(t,\Delta \tau , s) }{\a \Delta \tau  U(t,s+\gamma)}\nonumber \\
&R(s,t,\Delta \tau )=  K(s) - K^{est}(s,t,\Delta \tau). \label{reco2}
\end{align}
\end{definition}
\textcolor{black}{
The error term  $R(s,t,\Delta \tau )$ may be estimated uniformly for $s$ on some vertical strip of the complex plane such that $|\Im m (s)|>V$ for $V>0$ large enough. This
 requires some further regularity on the kernel $\kappa $, the initial data $\mu_0$  and the solution $\mu $ that ensure that $U$ and $K$ decay fast enough at infinity.
 }
\textcolor{black}{
\begin{enumerate}[label={\bf Hyp}-\arabic*]
\setcounter{enumi}{2}
\item \label{hyp3} 
There exists an interval $I\subset (0, \infty)$ 
such that $\kappa $ and the function
$x\mapsto s^{-1}x^{s}\kappa (x)$ are absolutely continuous functions on $x\in [0, 1]$, for all $s\in S_I$, where
 \begin{equation*}
 S_I=\left\{s\in \C;\,\Re e(s)\in I \right\}.
 \end{equation*}
\item \label{hyp4} Let $\mu _0\equiv u_0\in \mathcal{C}^3([0,\textcolor{black}{L}])$ and  either $u_0(\textcolor{black}{L})>0$, or $u_0(\textcolor{black}{L})=0$ and $u_0'(\textcolor{black}{L})<0$.
\end{enumerate}
The decay at infinity of  the Mellin transform $U$ follows from the condition  \eqref{hyp4} on $u_0$  thanks to the following  lemma whose proof is  postponed until  the end of Section \ref{sec:reco}.
}

\begin{lemma}[Regularity and support of the solution to the fragmentation equation]
 \label{lemma:prop}
 Assume the fragmentation kernel $\kappa$ satisfies \eqref{hyp1}. Take
$u_0 \in \mathcal{C}^3([0,\textcolor{black}{L}])$ such that $\supp(u_0)=[0,\textcolor{black}{L}]$.
Then, if we denote $\mu = u$ the solution to the fragmentation equation \eqref{eq:frag} with $\mu _0=u_0$, it holds

\begin{enumerate}
 \item The function $x \to u(t,x)$ is in $\mathcal{C}^3([0,\textcolor{black}{L}])$ for all $t>0$.
 \label{1}
 \item $\supp(u(t,.)) = [0,\textcolor{black}{L}]$.
 \item If $u_0(\textcolor{black}{L})>0$, then for all $t>0$, $u(t,\textcolor{black}{L})= e^{-\a \textcolor{black}{L}^{\gamma}t}u_0(\textcolor{black}{L})>0$.
 
 If $u_0(\textcolor{black}{L})=0$ and $u'_0(\textcolor{black}{L})<0$,
 then $u(t,\textcolor{black}{L})=0$ and $\p_x u(t,\textcolor{black}{L})= e^{-\a \textcolor{black}{L}^{\gamma}t}u'_0(\textcolor{black}{L})<0$ for all $t>0$.
\end{enumerate}
\end{lemma}

For any interval $I\subset (0, \infty)$ and  $V>0$ let us define the domain
\begin{align*}
D_{I, V}=\{s\in \C;\,\Re e s\in I,\,|\Im m s|>V \}    
\end{align*}
We have then the following
\textcolor{black}{
\begin{theorem}[Reconstruction formula for $K$]
\label{thm:reco}
Suppose that  the fragmentation kernel  $\kappa $  satisfies \eqref{hyp1} and \eqref{hyp3} and  $u_0$ satisfies \eqref{hyp2}  and \eqref{hyp4}. Then, the following holds.\\
(i) For all $T>0, \tau _0>0$   and $V>0$ sufficiently large, there exists  a constant $C>0$ depending on $\alpha , \gamma, V, I, T, \tau _0$, and $\textcolor{black}{L}$, such that for all $t\in (0, T)$ and all $\Delta \tau \in (0, \tau _0)$
\begin{equation}
 \label{thm:recoE1}
|R(s,t,\Delta \tau )|
\leq 
\dfrac{C\a \Delta \tau }{|s|}, 
\quad  \forall s\in D _{I, V}.
\end{equation}
(ii) For all $T>0$, all $\tau _0>0$ and $s\in \R$ such that $s>1$ there exists a constant $C=C(t, s, \tau _0)>0$ such that 
\begin{align}
 \label{thm:recoE15}
|R(s, t, \Delta \tau )|\le C\alpha \Delta \tau ,\,\forall t\in (0, T),\,\forall \Delta \tau\in (0, \tau _0)      
\end{align}
\end{theorem}
}
\begin{proof}[Proof of Thereom \ref{thm:reco}]
We first prove (i). Combining \eqref{reco2} with Lemma \ref{lemma:MellinPS}, we 
have the expression for the rest $R$
\begin{equation}
\label{R2}
R(s,t,\Delta \tau )=  
\dfrac{1}{\a \Delta \tau }\sum_{n=2}^{\infty} \dfrac{(\a \Delta \tau )^n}{n!} \dfrac{U(t,s+\gamma n )}{  U(t,s+\gamma)}
\sum_{j=0}^{n-1}(-1)^{n-1-j}K(s+j\gamma) 
\prod_{m=0}^{j-1} (K(s+m\gamma)-1).
\end{equation}
{\bf Step 1. Estimate for $K$}.
We prove here that for some $\tilde{C}>0$ 
depending on $I$ it holds
\begin{equation}
\label{Step1}
 |  K(s) | \leq \dfrac{\tilde{C}\MT{(I)}}{1+|s|},\quad \forall s\in  S_I.
\end{equation}
 \textcolor{black}{By (\ref{hyp1}), $K(s)$ is well defined and  analytic for $\Re e s>0$. Take $s\in I$. Since by (\ref{hyp3}), $\kappa \in C([0,1])$ and $\Re e(s)>0$ it follows that $x^{s}\kappa (x)\underset{x\to 0}\longrightarrow 0$. And since $\kappa $ and $x\to s^{-1}x^{s}\kappa (s)$ are absolutely continuous on $[0, 1]$
 }
\begin{equation*}
    K(s) = \dst\int_0^1 \kappa(x)x^{s-1} dx
    =\dfrac{\kappa(1)}{s} -\dfrac{1}{s}
    \dst\int_0^1 \kappa'(x)x^{s} dx,\,\,\forall s \in S_I.
    \end{equation*}
\textcolor{black}{Because $\kappa $ is absolutely continuous on $[0, 1]$ there exists two {\color{vert} non-de}creasing functions  on $[0, 1]$, $\kappa _1$ and $\kappa _2$, such that $\kappa =\kappa _1-\kappa _2$ on $[0, 1]$, $\kappa '_i$ are measurable and non negative on $[0, 1]$ for $i=1, 2$ and 
\begin{align*}
&\int _0^1{\color{vert}\kappa}'_i(x)\le {\color{vert}\kappa}_i(1)-{\color{vert}\kappa}_i(0),\,i=1, 2\\
&\int _0^1 |{\color{vert}\kappa}'(x)|dx\le \int _0^1 ({\color{vert}\kappa}_1'(x)+{\color{vert}\kappa}'_2(x))dx\\
&\hskip 2.3cm \le {\color{vert}\kappa}_1(1)+{\color{vert}\kappa}_2(1)-{\color{vert}\kappa}_1(0)-{\color{vert}\kappa}_2(0).
\end{align*}
Therefore,
\begin{align*}
    |K(s)|& \leq \dfrac{\kappa(1)}{|s|}
    +\dfrac{1}{|s|} \dst\int_0^1 |\kappa'(x)|dx
   ,\,\forall s\in S_I\\
   &\leq \dfrac{\kappa(1)}{|s|}+\dfrac{1}{|s|}
   ({\color{vert}\kappa}_1(1)+{\color{vert}\kappa}_2(1)-{\color{vert}\kappa}_1(0)-{\color{vert}\kappa}_2(0))
\end{align*}
from where \eqref{Step1} follows.}
\vspace{0.5cm}
 {\bf Step 2. Estimate for $U$}.
We prove here that for every $T>0$ and for $V$ large enough, there exists  a constant $C=C\MT{(\textcolor{black}{L},T, V,\alpha, \gamma)}>0$ such that for all $t\in (0, T)$ and $n\geq 2$,
\begin{equation}
\label{Step2}
   \left|\dfrac{ U(t,w +n\gamma+ iv)}{ U(t,w +\gamma+ iv)}\right| \leq 
   C\MT{(\textcolor{black}{L},T,\alpha, \gamma)} n(n-1)  \textcolor{black}{L}^{(n-1)\gamma},\,\,
   \forall s\in D _{I, V }.
\end{equation} 

We follow, for $|v|$ large, 
the calculation of \cite[Chapter IV, Section 4]{Dieudonne_infinitesimal} where the 
 stationary phase method is used 
 to study the behaviour of oscillatory integrals.
 For $w >0$, we have for $v\neq 0$
 \begin{equation*}
  \begin{aligned}
   U(t,w+iv)= \dst\int_0^{\textcolor{black}{L}}u(t,x)x^{w-1} x^{iv}  dx
   &=\dst\int_0^{\textcolor{black}{L}}u(t,x)x^{w-1} e^{iv\ln(x)}  dx\\
   &=\dfrac{1}{iv}\dst\int_0^{\textcolor{black}{L}} u(t,x) x^{w} \dfrac{d}{dx}\left( e^{iv\ln(x)} \right) dx.
  \end{aligned}
 \end{equation*}
 since $\dfrac{d}{dx} (e^{iv\ln(x)} )= \dfrac{iv}{x}e^{iv\ln(x)}$.
We perform an integration by part and we obtain

  \begin{equation*}
  \begin{aligned}
   U(t,w+iv)=\dfrac{1}{iv}u(t,\textcolor{black}{L})\textcolor{black}{L}^{w}e^{iv\ln(\textcolor{black}{L})}
   - \dfrac{1}{iv}\dst\int_0^{\textcolor{black}{L}}e^{iv\ln(x)}  \dfrac{\p}{\p x}\left(u(t,x) x^{w}  \right) dx,
  \end{aligned}
 \end{equation*}
 
 which we rewrite, using the same trick than above

   \begin{equation*}
  \begin{aligned}
   U(t,w+iv)=\dfrac{1}{iv}u(t,\textcolor{black}{L})M^{w}e^{iv\ln(\textcolor{black}{L})}
  - \left( \dfrac{1}{iv}\right)^2\dst\int_0^{\textcolor{black}{L}}   x \dfrac{\p}{\p x}\left( u(t,x) x^{w}\right) \dfrac{d}{dx}\left( e^{iv\ln(x)} \right) dx.
  \end{aligned}
 \end{equation*}
 
 We perform another integration by part to obtain

\begin{equation*}
  \begin{aligned}
   U(t,w+iv)=
   \dfrac{1}{iv}u(t,\textcolor{black}{L})^{w}e^{iv\ln(L)}
  &-  \left( \dfrac{1}{iv}\right)^2 L^{w}\left(  \textcolor{black}{L}\dfrac{\p}{\p x}  u(t,\textcolor{black}{L}) + w u(t,M) \right)
  e^{iv \ln( \textcolor{black}{L})}\\
  &+ \left( \dfrac{1}{iv}\right)^2\dst\int_0^{\textcolor{black}{L}}  \dfrac{\p}{\p x} \left( x \dfrac{\p}{\p x}\left( u(t,x) x^{w}\right)\right)
  e^{iv\ln(x)}  dx.
  \end{aligned}
 \end{equation*}
 The third term of the right-hand side above 
 can be expanded using
 
 \begin{equation*}
     \begin{aligned}
    \dfrac{\p}{\p x} \left( x \dfrac{\p}{\p x}\left( u(t,x) x^{w}\right)\right)
=w^2 x^{w-1}u(t,x) +  x^{w} (1+2w) \dfrac{\p}{\p x} u(t,x)+  x^{w+1} \dfrac{\p^2}{\p x^2} u(t,x).
     \end{aligned}
 \end{equation*}

 Then
we have

\begin{equation}
\label{fraction}
\left\{
    \begin{aligned}
    & U(t,w+\gamma+iv) =\dfrac{C(\MT{L},t,\MT{\alpha,\gamma},w,v)}{iv}+
     \dfrac{C'(\MT{L},t,\MT{\alpha,\gamma},w,v)}{(iv)^2},\\ 
    & U(t,w+n\gamma+iv) = n\textcolor{black}{L}^{(n-1)\gamma}\left(\dfrac{C(\MT{L},t,\MT{\alpha,\gamma},w,v)}{niv}+
     \dfrac{C''(\MT{L},t,\MT{\alpha,\gamma},w,v,n)}{(iv)^2} \right)
    \end{aligned}
    \right.
\end{equation}

 for some complex constants $C(\MT{L},t,\MT{\alpha,\gamma},w,v)$, $C'(\MT{L},t,\MT{\alpha,\gamma},w,v)$ and $C''(\MT{L},t,\MT{\alpha,\gamma},w,v,n)$ 
defined as
\begin{equation*}
    \begin{aligned}
    & C(\MT{L},t,\MT{\alpha,\gamma},w,v)=   u(t,\textcolor{black}{L})  \textcolor{black}{L}^{w+\gamma} e^{iv\ln(\textcolor{black}{L})}, \\
     &C'(\MT{L},t,\MT{\alpha,\gamma},w,v)=-\textcolor{black}{L}^{w+\gamma}\left(  \textcolor{black}{L}\dfrac{\p u}{\p x}  (t,\textcolor{black}{L}) +( w+\gamma) u(t,\textcolor{black}{L}) \right)
  e^{iv \ln( \textcolor{black}{L})},\\
   &  +\int_0^{\textcolor{black}{L}}  \left(
    ( w+\gamma)^2x^{w+\gamma-1}u(t,x) +  x^{w+\gamma} (1+2(w+\gamma))\dfrac{\p}{\p x} u(t,x)+  x^{w+\gamma+1} \dfrac{\p^2}{\p x^2} u(t,x)\right)  e^{iv\ln(x)} dx \\
  &  C''(\MT{L},t,\MT{\alpha,\gamma},w,v,n) =\textcolor{black}{L}^{w+\gamma}\left(  \textcolor{black}{L}\dfrac{\p}{\p x}  u(t,\textcolor{black}{L}) +\dfrac{w+n\gamma}{n} u(t,\textcolor{black}{L}) \right)
  e^{iv \ln( \textcolor{black}{L})}\\
   &  + \int\limits_0^{\textcolor{black}{L}}  \left(
    (w+n\gamma)^2 \left(1+ x^{w+n\gamma-1}\right) u(t,x) +  x^{w+n\gamma} (1+2(w+n\gamma))\dfrac{\p u}{\p x} (t,x)+  x^{w+n\gamma+1} \dfrac{\p^2u}{\p x^2} (t,x)\right)  dx.\\
    \end{aligned}
\end{equation*}

If $u_0(\textcolor{black}{L})>0$, then Lemma \ref{lemma:prop} guarantees that 
 $u(t,\textcolor{black}{L})>0$ as well.
 Then we have the following estimates on $C, C'$ and $C''$
 \begin{align*}
 &0 <C_0 \leq |C\MT{(L,t,\alpha, \gamma,w,v)}|\leq C_1, \quad   |C'\MT{(L,t,\alpha, \gamma,w,v)}|\leq C_2, \quad w\in\MT{I},\;  t\in[0,T], \; v\in \mathbb{R},\\
 &\,\hbox{and}\quad |C''\MT{(L,t,\alpha, \gamma,w,v,n)}|\leq C_3,\,\,\,\, w\in\MT{I},\;  t\in[0,T], \; v\in \mathbb{R}, \; n\geq 1.
 \end{align*}

Then, using \eqref{fraction}, there exists $V>0$ such that for $|v|\geq V$ and $w\in I$, 
\begin{align}
\label{geqV}
\left|\dfrac{ U(t,w +n\gamma+ iv)}{ U(t,w +\gamma+ iv)}\right|
&\leq n\textcolor{black}{L}^{(n-1)\gamma} \dfrac{|C i v+ C' |}{|Civ + C''|}\nonumber\\
&\leq 
n\textcolor{black}{L}^{(n-1)\gamma}  \left( 1+ \dfrac{|C''-C'|}{|C iv + C'|} \right)
\leq n\textcolor{black}{L}^{(n-1)\gamma}  C(V).
\end{align}
for some constant $C(V)>0$ that depends on $V$, and formula \eqref{Step2} follows.

Now if $u_0(\textcolor{black}{L})=0$, then Lemma \ref{lemma:prop} guarantees that 
 $u(t,\textcolor{black}{L})=0$ as well. Thus
 
   \begin{equation*}
  \begin{aligned}
   U(t,w)=
- \left( \dfrac{1}{iv}\right)^2\dst\int_0^{\textcolor{black}{L}}   x \dfrac{\p}{\p x}\left( u(t,x) x^{w}\right) \dfrac{d}{dx}\left( e^{iv\ln(x)} \right) dx.
  \end{aligned}
 \end{equation*}

 In that case, $u_0'(\textcolor{black}{L})<0$ so that 
 Lemma \ref{lemma:prop} guarantees that 
 $\p_xu(t,\textcolor{black}{L})<0$ as well, and 
 we go one step further in the expansion and write
 
   \begin{equation*}
  \begin{aligned}
   U(t,w)=
&- \left( \dfrac{1}{iv}\right)^2   \textcolor{black}{L}^{w+1}  e^{iv\ln(\textcolor{black}{L})}\dfrac{\p}{\p x}\left( u(t,x) x^{w}\right)\Big|_{x=\textcolor{black}{L}} 
\\
&+\left( \dfrac{1}{iv}\right)^3 
\textcolor{black}{L}  e^{iv\ln(\textcolor{black}{L})}\dfrac{\p}{\p x}\left( x \dfrac{\p}{\p x} \left(u(t,x) x^{w}\right)\right)\Big|_{x=\textcolor{black}{L}} 
\\
+&\left( \dfrac{1}{iv}\right)^3 \dst\int_0^{\textcolor{black}{L}}
\dfrac{\p}{\p x} \left(x \dfrac{\p}{\p x}\left( x \dfrac{\p}{\p x} \left(u(t,x) x^{w}\right)\right)\right) e^{iv\ln(x)} dx
  \end{aligned}
 \end{equation*}
 Using the same types of arguments than above,  formula \eqref{Step2} holds again.

{\bf Step 3. Estimate for $R$}.

Using formula \eqref{R2} and the triangle inequality,
we have 
\begin{equation*}
    |R(s,y,\Delta \tau )| \leq 
 \a \Delta \tau   
 \sum_{n=2}^{\infty} \dfrac{(\a \Delta \tau )^{n-2}}{n!}\left| \dfrac{U(t,s+\gamma n )}{ U(t,s+\gamma)}\right| \sum_{j=0}^{n-1}\left|K(s+j\gamma)\right| 
 \prod_{m=0}^{j-1} \left|K(s+m\gamma)-1\right|.
\end{equation*}

Using now \eqref{Step1} and \eqref{Step2}
we obtain for $\Re(s)\in I$ and $\Im(s)\in \R$

\begin{equation}
\label{Step3E1}
    |R(s,y,\Delta \tau )| \leq 
\frac{ \a \Delta \tau   \textcolor{black}{L}^{\gamma}}{1+|s|}
 \sum_{n=2}^{\infty} \dfrac{(\a \Delta \tau )^{n-2}}{(n-2)!}\left(\textcolor{black}{L}^{\gamma}\right)^{n-2} \sum_{j=0}^{n-1}\tilde{C}^{j+1}.
\end{equation}

which implies

\begin{equation*}
    |R(s,y,\Delta \tau )| \leq 
 \a \tau _0  \dfrac{\textcolor{black}{L}^{\gamma} \tilde{C}}{1+|s|}
 \sum_{n=2}^{\infty} \dfrac{(\a \Delta \tau )^{n-2}}{(n-2)!}\left(\textcolor{black}{L}^{\gamma}\right)^{n-2} \dfrac{1-\tilde{C}^n}{1-\tilde{C}},
\end{equation*}
and Theorem \ref{thm:reco}
is proved for $C= \textcolor{black}{L}^{\gamma} \tilde{C}\sum_{n=2}^{\infty} \dfrac{(\a \tau _0)^{n-2}}{(n-2)!}\left(\textcolor{black}{L}^{\gamma}\right)^{n-2} \dfrac{1-\tilde{C}^n}{1-\tilde{C}}<\infty$. 

\textcolor{black}{In order to prove now (ii),  suppose that $s\in \R$ is fixed such that $s>1$, estimate (\ref{Step1})  still holds. Moreover, since $\mu(t) \ge 0$ for all $t>0$ it follows that $U(t, s )>0$. Since  $U(\cdot, s)\in C(0, \infty)$ it follows that for all $T>0$ there exists a constant $C=C(T, s)>0$ such that $U(t, s+\gamma )>C^{-1}$ for all $t\in (0, t)$. It then follows, arguing as for (\ref{Step3E1}),
\begin{align*}
\left|\frac{U(t, s+n\gamma )}{U(t, s+\gamma )} \right|&\le 
C \textcolor{black}{L}^{\gamma n-1}||\mu (t)|| _{TV  }\\
|R(s,t,\Delta \tau )|&\le
C||\mu (t)|| _{TV  }\a \Delta \tau  \sum_{n=2}^\infty \frac{(\alpha \tau _0)^{n-2}}{n!} \sum_{j=0}^{n-1} \tilde{C}^{j+1}
\\
&\le \left(C||\mu (t)|| _{TV  } \sum_{n=2}^\infty \frac{(\alpha \tau _0)^{n-2}}{n!}
\dfrac{1-\tilde{C}^n}{1-\tilde{C}}\right) \a \Delta \tau 
\end{align*}
where $\tilde C$ comes from Step 1.
}
\end{proof}

The estimate in Theorem \ref{thm:reco} may be improved under stronger assumptions on $\kappa$. For example,

\begin{corollary}[A better estimate for kernels not allowing erosion]
\label{Col3}
\textcolor{black}{Suppose that  the hypothesis of Theorem \ref{thm:reco} hold. Suppose moreover that $\kappa '$ and  $x\mapsto \kappa '(x)x^{s+1}$ are absolutely continuous on $x\in [0, 1]$, and $\kappa(1)=0$ for  $s\in S_I$ .}
Then, for all $T>0$ and $\tau _0>0$ there exists two constants $V>0$ and $C\MT{(L,T,\alpha, \gamma, I, V,  \tau _0)}>0$ such that for all $t\in (0, T)$, $\Delta \tau \in (0, \tau _0)$
\begin{equation}
\label{Col3E1}
|R(s,t,\Delta \tau )|
\leq 
\dfrac{C\MT{(L,T,\alpha, \gamma, I, V, \tau _0)} \a \Delta \tau }{ |s|^2},\, \forall s\in  D_{I, V}.
\end{equation}
\end{corollary}

\begin{proof}
 The proof is the same than the proof of Theorem \ref{thm:reco}, 
 except that the estimate for the Mellin transform in Step 1 becomes,
   \begin{align*}
  K(s)& = \dst\int_0^1 \kappa(x)x^{s-1} dx = -\dfrac{1}{s}\dst\int_0^1 \kappa'(x)x^{s} dx\\
 & = -\dfrac{1}{s(s+1)}\left( \kappa'(1) -\kappa'(0) - \dst\int_0^1 \kappa''(x) x^{s+1} dx\right),\,\forall s\in S_I
 \end{align*}
and thus for some $\bar{C}>0$
\begin{equation*}
 |K(s)| \leq \dfrac{\bar{C}}{|s|^2+1},\,\,\,\,\,\forall s\in S_I.
\end{equation*}
\end{proof}
\textcolor{black}{
A natural question arising from Theorem \ref{thm:reco} and Corollary \ref{Col3} is if, and in what sense the inverse Mellin transform of $K^{est}(t, \Delta \tau )$, $\mathcal M^{-1}\left(K^{est}(t, \Delta \tau )\right)$,  is an approximation of the kernel $\kappa$ itself. By (\ref{hyp3}), for all $s\in I$,
\begin{align}
\label{Hkappa1}
\sup_{x\in [0, 1]}x^s|\kappa (x)|<\infty,
\end{align}
and then, for any $s_0\in I$, (\ref{Hkappa1}) holds for all $s>s_0$. But (\ref{Hkappa1}) is not known to be true for $s<\inf\{\sigma; \sigma \in I\}$. 
\begin{theorem}
\label{prop3}
 Suppose that the hypothesis of Corollary \ref{Col3} are satisfied and denote $I=(a, b)$ for some $b>a\ge 0$. Then, for every $t>0$, every $\tau_0 >0$ and $\delta >0$ sufficiently small,  there exists a positive constant $C$ that depends on $t$, $\tau _0$ and $\delta $  such that, for all $s\in (a, b)$,
 \begin{align}
 \label{prop3E0}
\sup_{x\in [0, 1]}x^s\left|\kappa  (x)-\mathcal M^{-1}\left( K^{est}(t, \Delta \tau) \right)(x)\right|\le C\Delta \tau,\,\,\,\,\forall\, \Delta \tau\in (0, \tau _0) . 
\end{align}
\end{theorem}
\begin{proof}
By hypothesis, for all $T>0$ and $\tau _0>0$ there exist two constants $V>0$ and  $C>0$ such that (\ref{Col3E1}) holds
for $t\in [0, T]$, $\Delta \tau \in (0, \tau _0)$,  $s\in S_I$ and $|\Im m(s)|>V$. \\
Moreover, 
for each $t\in (0, T)$ the function $U(t, s+\gamma )$ is analytic on the domain  $\{s\in \C; \Re e s>1-\gamma\}$ and $\gamma \ge 1$. Then, for any  $\delta >0$ sufficiently small to have   $(a+\delta/3 , a+2\delta )\subset $ the function $U(t, s+\gamma )$ is analytic on 
$$Q_{\delta , V}=\{s\in S_I; \Re e s\in (a+\delta/3 , a+2\delta , |\Im m s|<2V\}$$
and it may then  have only a finite number of zeros in $Q_{\delta , V}$. Therefore there exists a closed sub-interval $J\subset (a+\delta /2, a+\delta )$ such that $U(t, s+\gamma )\not =0$ for $s\in S_J\cap Q_{\delta , V}$. It follows by continuity  that for some constant $C(t, J, V, \delta )>0$ that may depend on $t, J, V, \delta $, 
\begin{equation}
\label{prop3E5}
|U(t, s+\gamma )|\ge C(t, J, V, \delta )>0,\,\,\,\forall s\in S_J,\,|\Im s|<V. 
\end{equation} 
On the other hand, for $w\in J$, 
\begin{equation*}
\begin{aligned}
      |U(t,w+n\gamma+iv)|=
      \left| \dst\int_0^{\textcolor{black}{L}} u(t,x) x^{w+n\gamma-1}e^{iv\ln(x)}dx \right|
      &\leq 
      \textcolor{black}{L}^{(n-1)\gamma} \left| \dst\int_0^{\textcolor{black}{L}} u(t,x) x^{w+\gamma- 1}dx \right|
      \\
      &\leq  \textcolor{black}{L}^{(n-1)\gamma}  \textcolor{black}{L}^{a+\gamma+1} \|u(t,.)\|_{\infty}.
      \end{aligned}
\end{equation*}
Then,
\begin{equation}
\label{leqV}
       \left|\dfrac{ U(t,w +n\gamma+ iv)}{ U(t,w +\gamma+ iv)}\right|\leq  \dfrac{2}{n C(t, J, V, \delta )}\textcolor{black}{L}^{a+\gamma+1} \|u(t,.)\|_{\infty} \textcolor{black}{L}^{(n-1)\gamma}, \qquad w\in J,\; v\in [-V,V].
\end{equation}
Therefore, by \eqref{geqV} and \eqref{leqV}, for every $t\in (0, T)$, $\tau _0 >0$ and $\delta >0$ small enough, there exists a closed  interval $J\subset I$ and a constant $C>0$ depending on $t, J, \tau _0, \delta $ such that
\begin{align}
\label{prop3E1}
\left|R(s, t, \Delta \tau)\right|\le \frac{C}{1+|s|^2}\,\,\forall s\in S_J,\,\,\forall \Delta \tau \in (0, \tau _0).
\end{align}
The inverse Mellin transform of $R(s, t, \Delta \tau )$ is then a well-defined function for all $x>0$ and $\Delta \tau \in (0, \tau _0)$, given by 
\begin{align*}
&\mathcal M^{-1}( R(t, \Delta \tau ))(x)=\frac {1} {2i\pi }\int _{ r-i\infty }^{r+i\infty}  R(s, t, \Delta \tau )x^{-s}ds,\,\,r\in J.
\end{align*}
and is such that
\begin{align*}
\mathcal M^{-1}( R(t, \Delta \tau ))\in E' _{ J }
\end{align*}
where $ E' _{ J }$ is the space of distributions whose Mellin transform is analytic on $S_J$ (cf. \cite{ML}). Since it also holds for $x>0$,
\begin{align*}
\kappa (x)\equiv\mathcal M^{-1}(K)(x)= \frac {1} {2i\pi }\int _{ r-i\infty }^{r+i\infty}  K(x) x^{-s}ds,\,\,r\in J.
\end{align*}
it follows from (\ref{reco2}) that the inverse Mellin transform of $K^{est}(t, \Delta \tau )$ is also well defined for $\Delta \tau \in (0, \tau _0)$ and $x>0$, and given by a similar integral expression.
It is then possible to apply the inverse Mellin transform to both sides of (\ref{reco2}) to obtain for $t>0$, $\Delta \tau \in (0, \tau _0)$, $x>0$ and some constant $C>0$ depending on $t$, $\tau _0$ and $\delta $,
\begin{align*}
\left|\kappa (x) -\mathcal M^{-1}(K^{est}(t, \Delta \tau)(x)\right|&=\left|\mathcal M^{-1}( R(t, \Delta \tau ))(x)\right|\\
&\le C\Delta \tau\, x^{-r}\int _{\R} (1+|v|^2)^{-1}dv,\,\,\forall x\in (0, 1),\,\,\forall r\in J.
\end{align*}
Since $r\in J\subset  (a+\delta/2, a+\delta )$, 
\begin{align*}
\left|\kappa  (x)-\mathcal M^{-1}\left( K^{est}(t, \Delta \tau) \right)(x)\right|\le C_\delta  x^{-a-\delta } \Delta \tau\,\,\,\,\forall x\in (0, 1),\,\,\forall \Delta \tau\in (0, \tau _0). \end{align*}
For every $s\in (a, b)$ there exists $\delta >0$ such that $s>a+\delta $ and then, for all $x\in [0, 1]$, $\Delta \tau\in (0, \tau _0)$,
\begin{align*}
x^s\left|\kappa  (x)-\mathcal M^{-1}\left( K^{est}(t, \Delta \tau) \right)(x)\right|\le C_\delta  x^{-a-\delta+s } \Delta \tau\le 
C_\delta  \Delta \tau,
\end{align*}
and estimate (\ref{prop3E0}) follows.
\end{proof}
A different reconstruction formula  of $\kappa $ was already obtained in Theorem 2, (iii) of~\cite{DET18}.
{\color{vert}
We notice that there are similarities in the formulae: both of them are the inverse Mellin transform of a ratio between two Mellin transforms of linear functionals of the solution,  the numerator taken in $s$ and the denominator taken in $s+\gamma.$ This reveals a serious drawback when noise is considered: both formulae then fall in the scope of so-called {\it severely ill-posed} inverse problems, exactly as for deconvolution problems, see e.g.~\cite{Baumeister_book}, ch. 4 for an introduction. However, despite the fact that this new formula is an approximation whereas the previous one was exact, its advantages are many.
\begin{itemize}
\item Experimentally, it is possible to use the measurement of the solutions at several pairs of close time points, thus making the most of experimental data, see~\cite{BTMPSTDX19}. On the contrary, with the formula in~\cite{DET18}, only the large-time asymptotic profile can be used.
\item In~\cite{DET18} there are specific difficulties linked to the measurement of the asymptotic profile: first, as time passes, the distribution is closer and closer to zero-size particles, making the measurement all the more noisy ; second, one needs to assess the validity of considering that the asymptotic behaviour is reached - the distance to the true asymptotics being a second source of noise ; third, it has been proved by numerical simulation that different fragmentation kernels may give rise to very close asymptotic size-distribution of particles~\cite{DETX20}.
\item Experimentally, it should be possible to depart from several very different initial conditions, and then use the superimposition principle to combine them in such a way that we get the most information. This is a direction for future research.
\end{itemize}
}}
\begin{remark}
If only the hypothesis of Theorem \ref{thm:reco} are assumed, the same argument as above still shows (\ref{prop3E5}) for some interval $J\subset I$ and some constant $C$ depending on $t, J$ and $V$. But instead of (\ref{prop3E1}) only the following holds for some constant $C=C(t, J, \tau _0)$,
\begin{align}
\label{prop3E7}
\left|R(s, t, \Delta \tau)\right|\le \frac{C}{1+|s|}\,\,\forall s\in S_J,\,\,\forall \Delta \tau \in (0, \tau _0).
\end{align}
The inverse Mellin transform of $R(t, \Delta \tau )$ is then still well defined, but its expression is now
\begin{align*}
&\mathcal M^{-1}( R(t, \Delta \tau ))=-\left(x \frac{\partial}{\partial x}\right)(h(t, \Delta \tau ))\\
&\forall x>0:\,\,h(t, \Delta \tau, x)=  \frac {1} {2i\pi }\int _{ r-i\infty }^{r+i\infty}  R(s, t, \Delta \tau )s^{-1}x^{-s}ds,\,\,r\in J,
\end{align*}
where $h(t, \Delta \tau , x)$ is a function, defined for all $x>0$ and all $\Delta \tau\in (0, \tau _0)$,  such that $h(t, \Delta \tau _0)\in E'_J$. As above, the inverse Mellin transform of $K^{est}(t, \Delta \tau )$ is then well defined too, but no point wise estimate like (\ref{prop3E0}) holds.
\end{remark}

{\color{vert} 
\begin{remark}
At first sight, regularity assumptions such as~\eqref{hyp3}, \eqref{hyp4} and the non-erosion of Corollary~\ref{Col3} may be surprising. It is however classical in the field of inverse problems to assume regularity on the object we want to estimate, and to gain a better convergence rate when the regularity increases, see for instance~\cite{Engl}. Here however it is not only on $\kappa$ that regularity is required - moreover~\eqref{hyp3} is satisfied for a large class of measures and for all the classical fragmentation kernels - but also on the initial condition, which is less expected. We thus have been able to reconstruct the fragmentation kernel in two extreme cases: either very singular initial condition, given by a Dirac delta function, or very regular ones, at least ${\cal C}^3$.
\end{remark}}

\textcolor{black}{Point (ii) of  Theorem \ref{thm:reco} may also be used to estimate the  statistical parameters of the kernel $\kappa $ like mean, variance, skewness, kurtosis, since all of them may be expressed in terms of $K(s)$ for integer values of $s$. Consider for example the variance given by
\begin{equation*}
Var\left[\dfrac{\kappa}{2}\right] =\dfrac{1}{2} \dst\int_{(0,1)} \left| x-\dfrac{1}{2}\right|^2 \kappa(x)dx
=\dfrac{1}{2}K(3)- \dfrac{1}{2}K(2)+ \dfrac{1}{8}K(1)=\dfrac{1}{2}K(3)- \dfrac{1}{4}
\end{equation*} 
It is then possible to estimate $Var[\kappa ]$ using \ref{U_power_serie_inter} and defining:
\begin{equation*}
\left(Var\left[\dfrac{\kappa}{2}\right] \right)^{est}(t, \Delta \tau )=\dfrac{1}{2}K^{est}(3, t, \Delta \tau )- \dfrac{1}{4}.
\end{equation*}
The following Corollary immediately follows from Point (ii) of  Theorem \ref{thm:reco} for $s=3$.
\begin{corollary}[Estimate of the variance of the kernel]
\label{cor:variance}
Suppose that the assumptions of Theorem \ref{thm:reco} are satisfied. Then for all $T>0$ and $\tau _0>0$ there is $C\MT{(L,T, \tau _0,\alpha, \gamma)}$ such that
\begin{equation*}
 \left| Var[\kappa]-(Var[\kappa])^{est}(t, \Delta \tau )\right| \leq 
C\MT{(L,T,\alpha, \gamma)}\alpha \Delta \tau ,\,\,\forall t\in (0, T),\,\forall \Delta \tau \in (0, \tau _0).
\end{equation*}
\end{corollary}
}

\subsection{Proof of Lemma \ref{lemma:prop}.}
\begin{proof}[Proof of Lemma \ref{lemma:prop}]
 The arguments
rely on the formula obtained in Theorem \ref{thm:power_serie}.
\begin{enumerate}
 \item 
 We use the formula \eqref{representation_solution}
 obtained in Theorem \ref{thm:power_serie}.
 Note that it can be rewritten using the change of variables
\begin{equation}
\label{CV}
 z=\dfrac{x}{\ell}, \quad dz= -\dfrac{z^2}{x}d\ell,
\end{equation}

as 
\begin{equation}
\label{name}
\begin{aligned}
  u(t,x) &= e^{-\a x^{\gamma}t} u_0(x)
  +  \sum_{n=0}^{\infty} (\a t)^n
  \dst\int_0^{1} \dfrac{x^{n\gamma}}{z^{n\gamma}} a_n(z) u_0\left(\dfrac{x}{z}\right) \dfrac{dz}{z}.
\end{aligned}
\end{equation}

The first term of the sum is clearly $C^1$, since $u_0$ is.
To deal with the second term, set 
\begin{equation*}
 I_n(x)= \dst\int_0^{1} \dfrac{x^{n\gamma}}{z^{n\gamma}} a_n(z) u_0\left(\dfrac{x}{z}\right) \dfrac{dz}{z}.
\end{equation*}

The fisrt step is to prove by induction
that for all $x_0>0$, for all $n\geq 0$, 
 $z\to a_n(z) \in \mathcal{C}^1[x_0,1]$.
 The function $a_0$ is clearly $ \mathcal{C}^1$, since it is identically zero.
 Let us assume that for some $n\geq 0$, 
 $z\to a_n(z) \in \mathcal{C}^1[x_0,1]$.
The function $a_{n+1}$ satisfies \eqref{induction} and is composed with three terms. 
The first term and third term are clearly $\mathcal{C}^1$ since 
$a_n(z) \in \mathcal{C}^1[x_0,1]$.
We focus on the second term
 \begin{equation*}
  J_n(x)= \dst\int_x^{\infty} y^{\gamma-1} \kappa \left(\dfrac{x}{y} \right)a_n(y)dy.
  \end{equation*}
Once again,  it can be rewritten using the change of variables
\eqref{CV}
 \begin{equation*}
  J_n(x)=\dst\int_{x_0}^{1} \dfrac{x^{\gamma}}{z^{\gamma}} \kappa(z) a_n\left(\dfrac{x}{z}\right) \dfrac{dz}{z}.
 \end{equation*}
 
 The dominated convergence theorem guarantees that 
 $J_n\in \mathcal{C}^1[x_0,1]$ and 
 that 
 \begin{equation*}
  J_n'(x)= \dst\int_{x_0}^{1} \dfrac{x^{\gamma}}{z^{\gamma}} \kappa(z) a_n'\left(\dfrac{x}{z}\right) \dfrac{dz}{z^2}
  + \dst\int_{x_0}^{1} \gamma\dfrac{x^{\gamma-1}}{z^{\gamma}} \kappa(z) a_n\left(\dfrac{x}{z}\right) \dfrac{dz}{z}.
 \end{equation*}
Indeed
\begin{equation*}
 \left|\dfrac{x^{\gamma}}{z^{\gamma}} \kappa(z) a_n'\left(\dfrac{x}{z}\right) \dfrac{1}{z^2}\right|
\leq \dfrac{1}{x_0^{2+\gamma}}\|a_n'\|_{\mathcal{C}^0[x_0,1]}, \quad 
 \left|\gamma\dfrac{x^{\gamma-1}}{z^{\gamma}} \kappa(z) a_n\left(\dfrac{x}{z}\right) \dfrac{1}{z}\right|
\leq \dfrac{\gamma \max\{x_0^{\gamma-1}, 1\}}{x_0^{1+\gamma}}\|a_n\|_{\mathcal{C}^0[x_0,1]}.
\end{equation*}
We have proven that $a_n\in\mathcal{C}^1(0,1)$ since it is 
$\mathcal{C}^1(K)$ for all $K$ compact of $(0,1)$.


The fisrt step is to prove
 that $I_n \in \mathcal{C}^1([0,\textcolor{black}{L}])$.
 To do so, we use the dominated convergence to prove that 
for all $x_0>0$, we have $I_n  \in \mathcal{C}^1([x_0,\textcolor{black}{L}])$ and 
that
\begin{equation}
\label{I_n'}
\begin{aligned}
I_n'(x)
=
n \gamma x^{n\gamma-1}   \dst\int_0^{1}  \dfrac{1}{z^{n\gamma}}a_n(z) u_0\left(\dfrac{x}{z}\right) \dfrac{dz}{z}
+\dst\int_0^{1} 
 \dfrac{x^{n\gamma}}{z^{n\gamma}}
  a_n(z) u_0'\left(\dfrac{x}{z}\right)
  \dfrac{dz}{z^2}, \quad x\in [x_0,\textcolor{black}{L}].
\end{aligned}
\end{equation}
Indeed, the conclusion of the dominated convergence holds: 
$u_0 \in \mathcal{C}^1([0,\textcolor{black}{L}])$, hence the the integrand is in $\mathcal{C}^1([0,\textcolor{black}{L}])$ as well. 
The domination is as follows: since $\supp(u_0)\subset [0,\textcolor{black}{L}]$, the bounds of the integral $I_n$ are $z\in \left[\frac{x}{\textcolor{black}{L}},1\right] \subset \left[\frac{x_0}{\textcolor{black}{L}},1\right]$, and thus
\begin{equation*}
\begin{aligned}
 \left| \dfrac{x^{n\gamma-1} }{z^{n\gamma}}a_n(z) u_0\left(\dfrac{x}{z}\right) 
 \dfrac{1}{z}\right|
&\leq \dfrac{\Mag{\max\{\left(\frac{x_0}{\textcolor{black}{L}}\right)^{n\gamma-1}, \textcolor{black}{L}^{n\gamma-1}\}}}{x_0^{2} }
\|u_0\|_{\infty}a_n(z), 
\\  
\left|\dfrac{x^{n\gamma}}{z^{n\gamma}}
  a_n(z) u_0'\left(\dfrac{x}{z}\right) \dfrac{dz}{z^2}\right|
&\leq \dfrac{\Mag{\max\{\left(\frac{x_0}{\textcolor{black}{L}}\right)^{n\gamma}, \textcolor{black}{L}^{n\gamma}\}}}{x_0^{2} }
\|u_0'\|_{\infty} a_n(z),
\end{aligned}
\end{equation*}
and it was proved in \eqref{u_n} that 
$\|a_n\|_{TV} \leq \dfrac{(N+2)^n}{n!}$.

Now we claim that the function $S$ defined as
\begin{equation}
 S(x)= \sum_{n=1}^{\infty} (\a t)^n I_n(x)
\end{equation}
is of class $\mathcal{C}^1([x_0,\textcolor{black}{L}])$ for all $x_0>0$.
Indeed, we just saw that $I_n\in \mathcal{C}^1([0,\textcolor{black}{L}])$, and 
that $I_n'$ is given by \eqref{I_n'}.
For $x\in [x_0,\textcolor{black}{L}]$, we can control 
each of the two terms of the sum \eqref{I_n'} by 
two sequences that converge. Indeed
using again \eqref{u_n}, we have 
\begin{equation*}
\begin{aligned}
 n  \dst\int_0^{1}  \left| \dfrac{x^{n\gamma-1} }{z^{n\gamma}}a_n(z) u_0\left(\dfrac{x}{z}\right) 
 \dfrac{1}{z}\right|dz
&\leq \dfrac{\Mag{\max\{\left(\frac{x_0}{\textcolor{black}{L}}\right)^{n\gamma-1}, \textcolor{black}{L}^{n\gamma-1}\}}}{x_0^{2} }
\|u_0\|_{\infty} n \dfrac{(N+2)^n}{n!}, 
\\  
  \dst\int_0^{1}  \left|\dfrac{x^{n\gamma}}{z^{n\gamma}}
  a_n(z) u_0'\left(\dfrac{x}{z}\right) \dfrac{dz}{z^2}\right| dz
&\leq \dfrac{\Mag{\max\{\left(\frac{x_0}{\textcolor{black}{L}}\right)^{n\gamma}, \textcolor{black}{L}^{n\gamma}\}}}{x_0^{2} }
\|u_0'\|_{\infty} \dfrac{(N+2)^n}{n!},
\end{aligned}
\end{equation*}

and 
\begin{equation*}
 \sum_{n=1}^{\infty} (\a t)^n \left(\dfrac{ \textcolor{black}{L}^{\gamma}}{x_0^{\gamma}} \right)^n \dfrac{(N+2)^n}{n!} <\infty, \qquad  \sum_{n=1}^{\infty} (\a t)^n  \left(\max\{x_0^{\gamma}, \textcolor{black}{L}^{\gamma}\}\right)^n
 n \dfrac{(N+2)^n}{n!} <\infty,
\end{equation*}
This ends the proof of \ref{1}, 
and we have in addition for $x\in [x_0,\textcolor{black}{L}]$
an expression of the spatial derivative of $u$
\begin{equation}
 \label{u_x}
 \begin{aligned}
 \dfrac{\p}{\p x}u(t,x)= &
 e^{-\a x^{\gamma} t}u_0'(x)
 - \a \gamma x^{\gamma-1}t  e^{-\a x^{\gamma} t}u_0(x)
 \\
&+\sum_{n=1}^{\infty} (\a t)^n 
\left(n \gamma x^{n\gamma-1}   \dst\int_0^{1}  \dfrac{1}{z^{n\gamma}}a_n(z) u_0\left(\dfrac{x}{z}\right) \dfrac{dz}{z}
+\dst\int_0^{1} 
 \dfrac{x^{n\gamma}}{z^{n\gamma}}
  a_n(z) u_0'\left(\dfrac{x}{z}\right)
  \dfrac{dz}{z^2}
  \right).
  \end{aligned}
\end{equation}
\Mag{Similar arguments hold to guarantee that $x\to u(t,x) \in C^3([0,\textcolor{black}{L}])$.}
\item 
First, we claim that $\supp(u(t,.))\subset[0,\textcolor{black}{L}]$.
Indeed, this is a consequence of formula \eqref{name} and of the fact that $\supp(a_n) \subset [0,1]$ for $n\geq 0$.
Let us now prove that 
$\supp(u(t,.))=[0,\textcolor{black}{L}]$.
Take $y\in [0,\textcolor{black}{L}]$ and set $Y(t)= u(t,y)$.
The fragmentation equation \eqref{eq:frag}
implies 
\begin{equation*}
 Y'(t) \geq -\a y^{\gamma} Y(t), 
\end{equation*}
{\it i.e.}
\begin{equation*}
 u(t,y)= Y(t) \geq e^{-\a y^{\gamma}t} Y(0)=e^{-\a y^{\gamma}t} u_0(y).
\end{equation*}
If $u_0(y)\neq 0$, then  for all $t\geq 0$, $u(t,y)\neq 0$.
If $u_0(y)= 0$, since $y\in \supp(u_0)$, for all $\eps>0$, there exists 
$y_{\eps}$ such that $|y-y_{\eps}|<\eps$
and $u_0(y_{\eps})\neq 0$ and then  $u(t,y_{\eps})\neq 0$, 
which implies that $y\in \supp(u(t,.))$.
Thus $\supp(u(t,.))=[0,\textcolor{black}{L}]$.
\item
 It is clear from formula \eqref{representation_solution} that 
$u(t,\textcolor{black}{L})= e^{-\a \textcolor{black}{L}^{\gamma}t} u_0(\textcolor{black}{L})$.
Then, if  $u_0(\textcolor{black}{L})>0$, we have
$u(t,\textcolor{black}{L})>0$.
If $u_0(\textcolor{black}{L})=0$ and $u'_0(\textcolor{black}{L})<0$, we have 
$u(t,\textcolor{black}{L})= 0$, and 
formula \eqref{u_x} implies 
\begin{equation*}
 \dfrac{\p}{\p x}u(t,\textcolor{black}{L})= e^{-\a \textcolor{black}{L}^{\gamma}t}u'_0(\textcolor{black}{L}) <0.
\end{equation*}

\end{enumerate}
\end{proof}

 \section{Numerical simulations}

{\color{vert}In this section, we illustrate the different theoretical results and investigate the convergence errors of the reconstruction formulae.}
\MT{
\begin{itemize}
\item {\bf Illustration of Theorem \ref{thm:short}:} 
 We show on Figure \ref{fig:kernel} the profile of the estimated kernel  $\kappa^{est}(t)$ defined in formula \eqref{k_est_F}, 
 for $\gamma=\alpha=1$, for four different kernels $\kappa$ and for different times $t$.
 For each plot, the kernel $\kappa$ is displayed in an inset on the upper right. The initial condition $\mu_0$ is a highly peaked gaussian centered at $x=1$ and the numerical solution $\mu_t$ used to build $\kappa^{est}(t)$
is obtained using a numerical scheme with a time step $\Delta t=0.01$.
We observe that the estimate $\kappa^{est}(t)$ is valid for early time points: indeed, at the naked eye, $\kappa^{est}(t)$ and $\kappa$ look alike.
As time goes by, the size distribution is driven towards the stationary state and the information on the kernel is lost. \\
On Figure~\ref{fig:Th2:errorTV} Left, we 
\MT{illustrate the estimate \eqref{estkappaC} and show that the time evolution of the error $$e_{TV}(t)=\Big\|\kappa^{est}(t)-\kappa\Big\|_{TV}$$ increases linearly with time $t$ for the same four kernels $\kappa$ considered in Figure~\ref{fig:kernel} and for an initial condition $\mu_0$ very close to $\delta(x=1)$.
We observe that the slope of $t\to e_{TV}(t)$ is small for kernels of erosion type $(\kappa(0)\neq 0)$, and large for kernels producing daughter particles of similar sizes.} {\color{vert} This may be linked to a larger constant in~\eqref{estkappaC} for more peaked kernels; this provides us with an interesting direction for future work.}
\item \MT{\bf{Illustration of Corollary \ref{cor:generic}}}
In  Figure~\ref{fig:Th2:errorTV} Right, we draw the curves of the error $e_{TV}(t)$ for  three initial conditions $\mu_0$  given by (truncated) gaussians of standard deviation $\sigma=0.01,$ $\sigma=0.1$ and $\sigma=0.2$ and for the kernel $\kappa$ in black on the left figure.
As seen on the formula \eqref{cor:generic2}, the increase of $e_{TV}(t)$ is linear with respect to time $t$, but
 an extra constant error $\eps$ is added, related to the distance between $\delta_1$ and $\mu_0$. 
We notice that a small error term $\eps$ was already observed in Figure \ref{fig:Th2:errorTV}, Left, due to the distance between $\delta(x=1)$ and its numerical approximation on a discrete grid.
 For large standard deviations (e.g. $\sigma=0.2)$, the error $\eps$ becomes so large that the estimate in Total Variation norm is no more meaningful: we see the interest to turn to the Bounded Lipshitz norm.
 \\
 In Figure \ref{fig:cor2:time1}, we display the shape of the estimated kernel $\kappa^{est}(t)$ for a small value of $t$, for a kernel $\kappa$ of erosion type and with three initial conditions being gaussians with various spreading. It can be observed at the naked eye that the thiner the gaussian is, the better the approximation $\kappa^{est}(t)$ is as well. 
We observe how  the estimated kernel $\kappa^{est}$ is differently impacted around $x=0$ and around $x=1$:
 this  gives interesting hints on how the kernel symmetry could be used to improve the theoretical estimates.\\
 \item  {\bf Illustration of Theorem \ref{thm:stab}}:
 In Fig~\ref{fig:Marie}, we display the error
 $$e_{BL}(t)=\|\kappa^{est}(t)-\kappa\|_{BL}$$ as a function of time for a two-peaked gaussian kernel $\kappa$, for a gaussian initial condition $\mu_0$ of variance $\sigma^2$ and for a noise $\eps_0$ on the measurement of the initial data $\mu_0$ and a noise $\eps_1$ on the measurement of the solution $\mu_t$ used in the calculation of $\kappa^{est}(t)$. {\color{vert} The standard deviation $\sigma$ thus plays the role of $q$ in Theorem~\ref{thm:stab}. To simulate the noise on the solution observed, we add a multiplicative uniform noise on $[-0.5\varepsilon_i,+0.5\varepsilon_i]$ to the simulation.} 
 Numerically, we approximate the BL norm by the Wasserstein distance $W_1$, since 1/it is easier to compute using the monotone rearrangement theorem, 2/the BL norm is close to the Wasserstein distance between two measures of approximately same mass and whose supports are not too far.
  The error $e_{BL}(t)$ first decreases and then increases, as expected by Remark \ref{rem:t_optimal}. 
 In the inset, we superimposed the kernel $\kappa$ (in red) with the best estimated $\kappa$, namely $\kappa^{est}(t^*)$ (in blue) taken at the optimal time where the error reaches its minimum.
\end{itemize}
 }

  \MT{\begin{itemize}
 \item 
{\bf Illustration of Remark \ref{rem:t_optimal}}
 In Figure \ref{fig:Marie2}, we investigated how the minimal error and the optimal time, {\it i.e.} the time displaying the minimal error (drawn on Figs.~\ref{fig:Marie} as the red asterisk), evolve with respect to the noise level. To do so, we take an equal level of noise for the three noise sources $\eps_0,$ $\eps_1$ and $q$ (with $q=\sigma$ the standard deviation of the gaussian taken for the initial size distribution). We ran fifty simulations - to take into account the fact that the noise we simulate is random - and we draw the optimal time (blue asterisks in Fig.~\ref{fig:Marie2}) giving the optimal error (green asterisks in Fig.~\ref{fig:Marie2}). We  then compare  the mean curves over these fifty simulations, and compare it with the curve $x \mapsto \sqrt{x}$.  We observe a good qualitative agreement with the expected rate of convergence.
\end{itemize}
}

 \MT{\begin{itemize}
 \item {\bf Illustration of Corollary \ref{cor:variance}}
 We illustrate how we recover the variances of the 6 different typical fragmentation kernels described in  Table 1.
  We recall that the variance and standard deviation are given by
  \begin{equation*}
  \Mag{Var=} Var\left[\dfrac{\kappa}{2}\right] =\dfrac{1}{2}K(3)- \dfrac{1}{4}, \qquad \Mag{SD=\sqrt{Var}},
  \end{equation*}
  and we define the estimated variance of the kernel as,
    \begin{equation*}
    \begin{cases}
 Var^{est}(t,\Delta t)=\dfrac{1}{2}K^{est}(3,t,\Delta t)- \dfrac{1}{4}, \qquad \text{ if }\dfrac{1}{2}K^{est}(3,t,\Delta t)- \dfrac{1}{4}>0\\
 Var^{est}(t,\Delta t)=0, \qquad \qquad \qquad\hspace{2cm} \text{ else.}
 \end{cases}
  \end{equation*}
   where the formula for $K^{est}$ is given in Definition \ref{def:Kest}. Let us recall that the estimation of the variance $Var^{est}$ is not a priori the variance of the estimated kernel $\kappa^{est}$.
We also define $SD^{est}(t,\Delta t)=\sqrt{Var^{est}(t,\Delta t)}$.
 In Figure \ref{fig:Dt}, we assume $\gamma=\alpha=1$, we consider the six kernels $\kappa$ described in Table \ref{table} and the initial condition is a peaked gaussian centered at $x=2$. We plot the
 relative error on the standard deviation defined as 
 \begin{equation*}
 \text{Relative Error on the Standard Deviation}= \dfrac{|SD^{est}(0,\Delta t) - SD|}{{SD}}
 \end{equation*}
as a function of $\Delta t$.
We observe that for large values of $\Delta t$ the relative error is saturated and equal to $1$ for the kernels in blue, red and yellow, corresponding to kernels with small variances.
For these kernels, the estimated variance becomes negative from a certain value for $\Delta t$, so that $Var^{est}$ is then $SD^{est}$ are taken to be zero.
The worst estimation of the relative standard deviation we have is for the kernel in blue, {\it i.e.} for the kernel with a very small standard deviation (SD=$0.1001$): the estimation of the standard deviation $SD^{est}$ is zero, and then the relative error is equal to $1$.
For $\alpha \Delta t=0.1$, we are able to have a good idea 
of the ordering of standard deviations of the six kernels.
 \end{itemize}
 }

\begin{figure}
 \begin{tabular}{cc}
\includegraphics[width=0.47\textwidth]{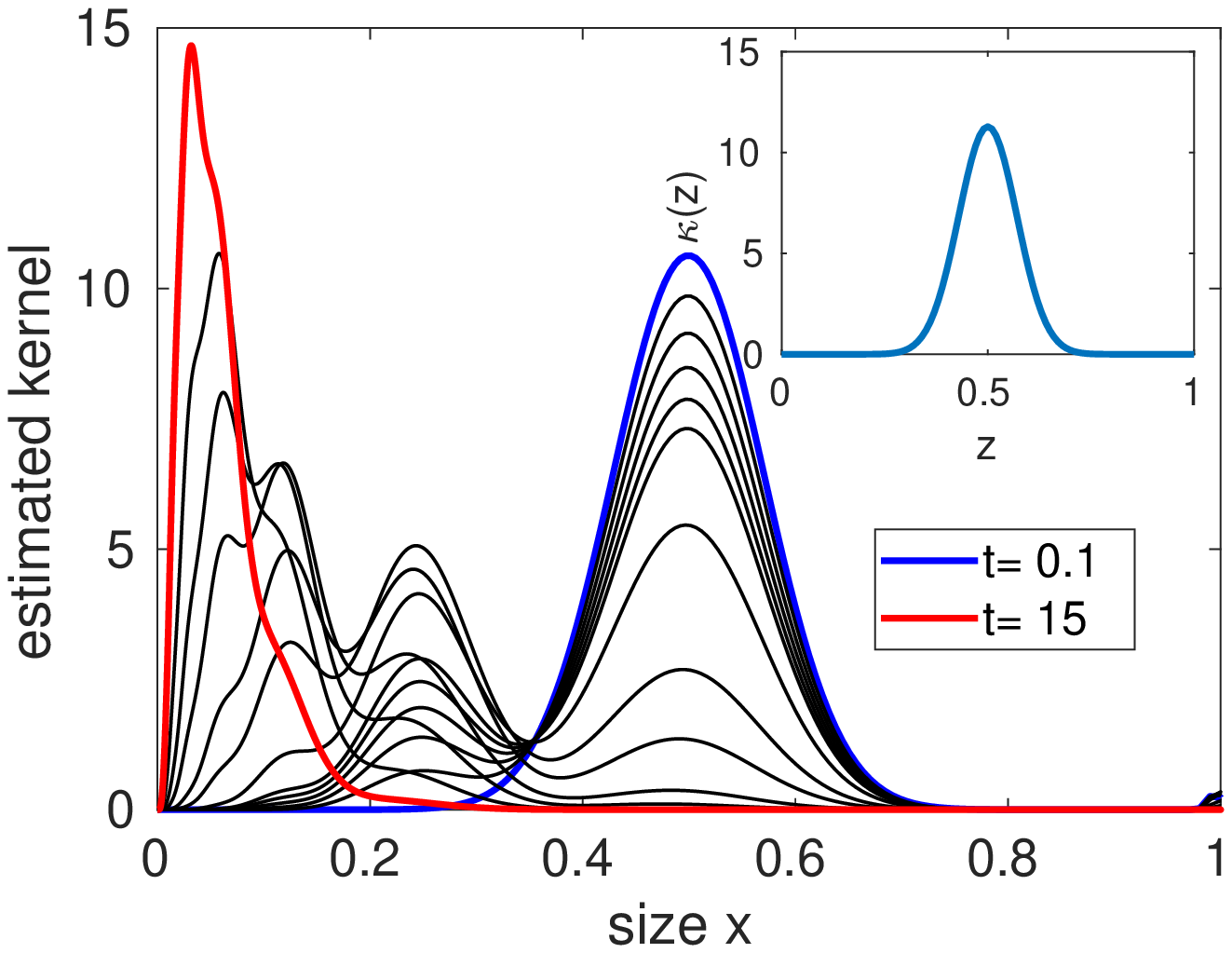} &  \includegraphics[width=0.47\textwidth]{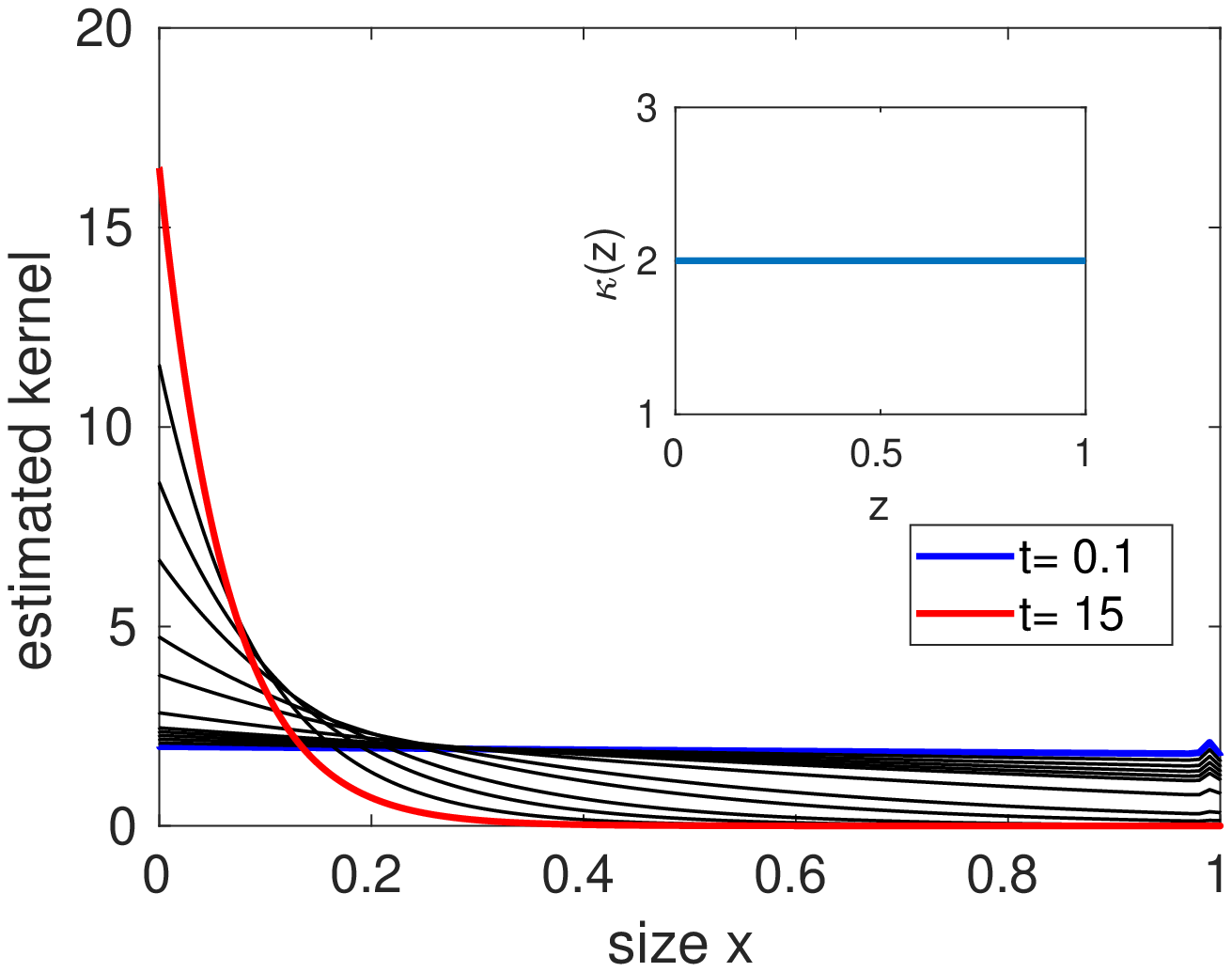} \\
\includegraphics[width=0.47\textwidth]{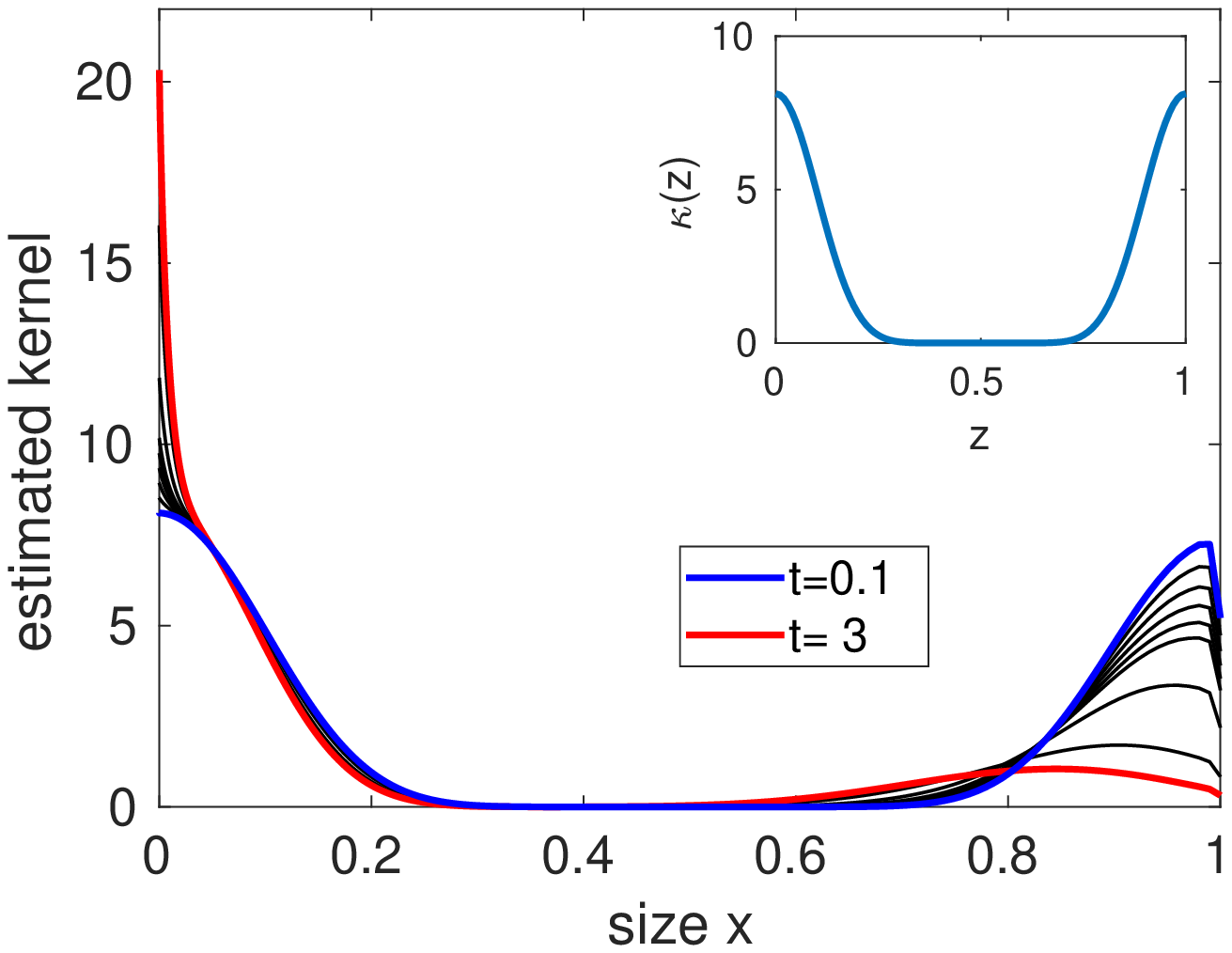} &  \includegraphics[width=0.47\textwidth]{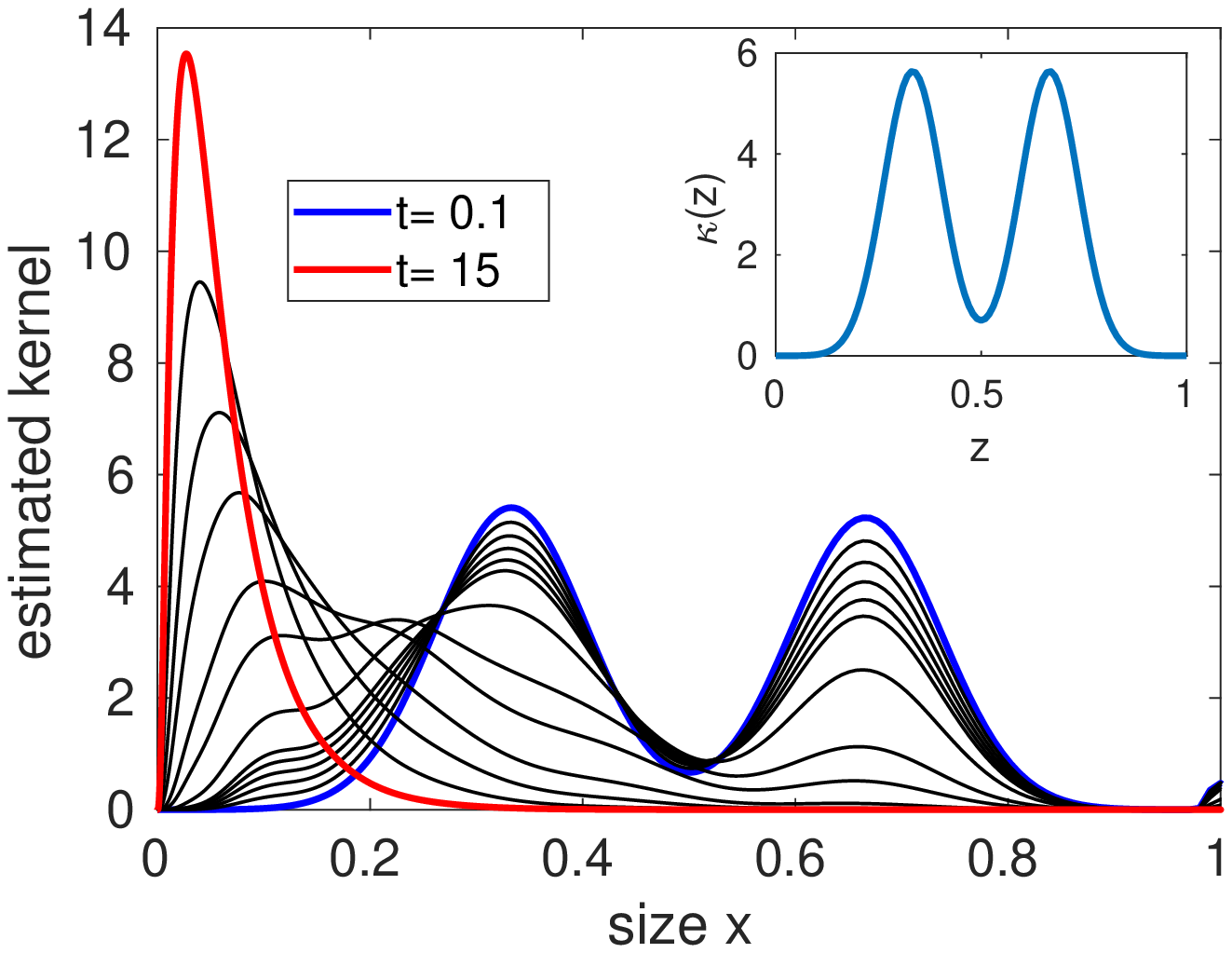}
 \end{tabular}
 \caption{\MT{{\bf Profile of the estimated kernel}  $\kappa^{est}(t)$ for $\gamma=\alpha=1$.
 Upper-right inset, light blue: the kernel $\kappa$.
Blue: $\kappa^{est}(t)$ for $t=0.1$. Red: $\kappa^{est}(t)$ for a large value of $t$.
Black: $\kappa^{est}(t)$ for intermediate times $t$.
 }}
   \label{fig:kernel}
 \end{figure}

   \begin{figure}
\includegraphics[width=0.5\textwidth]{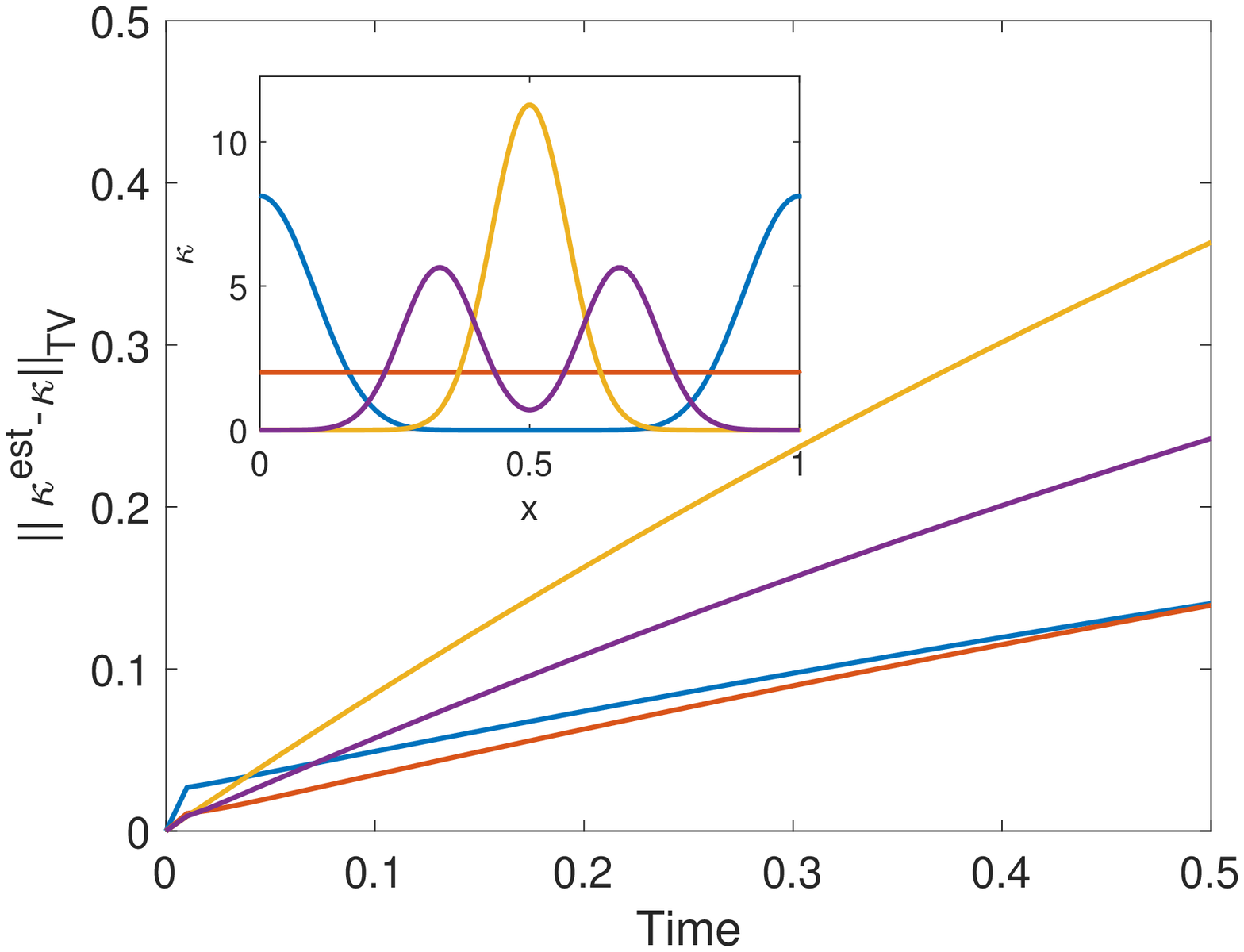}
   \includegraphics[width=0.5\textwidth]{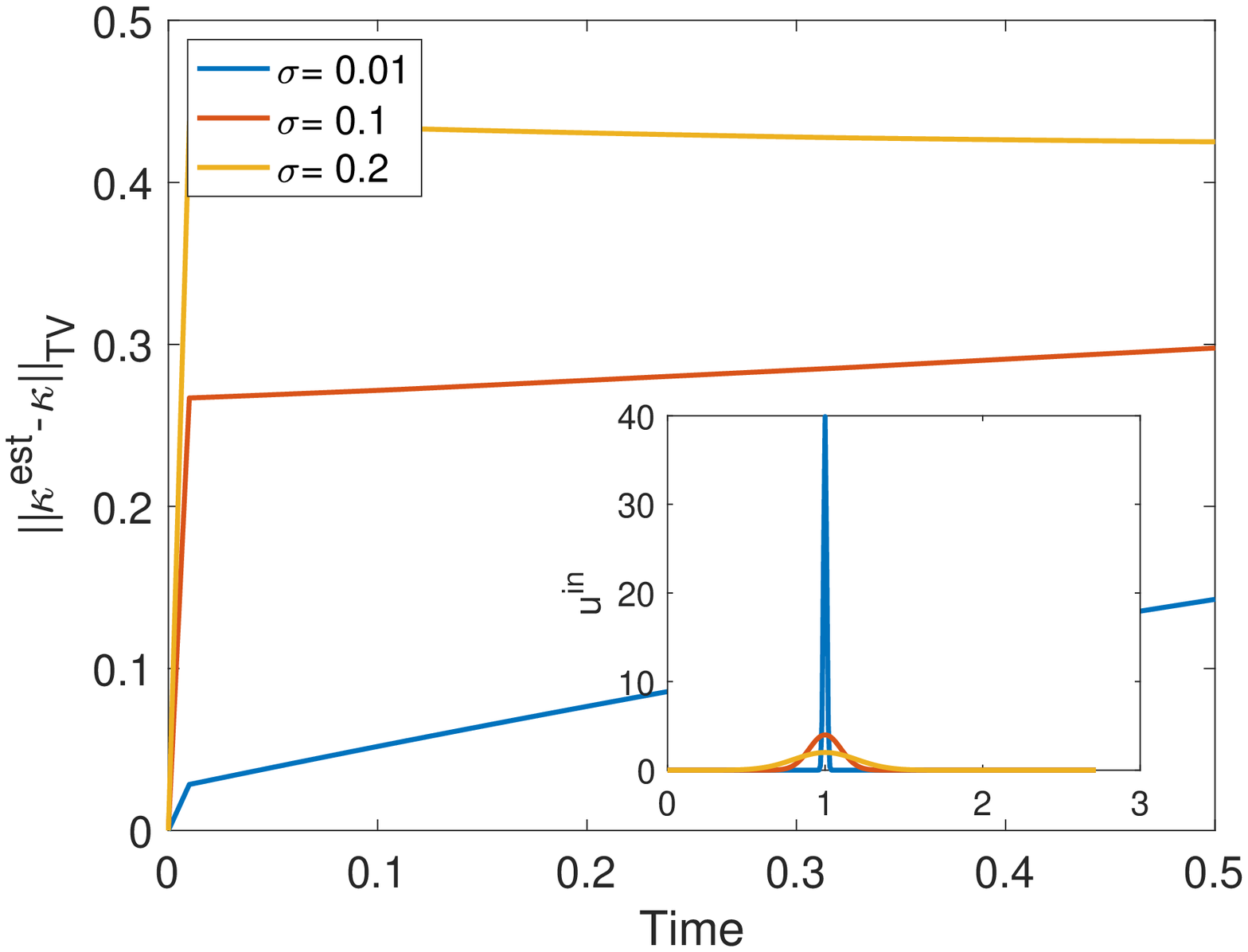}
 \caption{ \comMD{{\bf Time evolution of the distance in the Total Variation norm between the fragmentation kernel and its first order estimate} given by $\kappa^{est},$ 
 departing from $u_0=\delta_1$ (Left) or departing from a Gaussian curve centered at $x=1$ with a standard deviation $\sigma=0.01$, $\sigma=0.1$ and $\sigma=0.2$ respectively (Right). 
 Left:   the corresponding kernel is displayed on the inset with the same colour as the error curve. Right: the fragmentation kernel is the one in Fig~\ref{fig:kernel} bottom left (in blue on the inset of the left figure). The corresponding initial condition is displayed on the inset with the same colour as the error curve.}
   \label{fig:Th2:errorTV}}
 \end{figure}
 
 
    \begin{figure}
\includegraphics[width=0.5\textwidth]{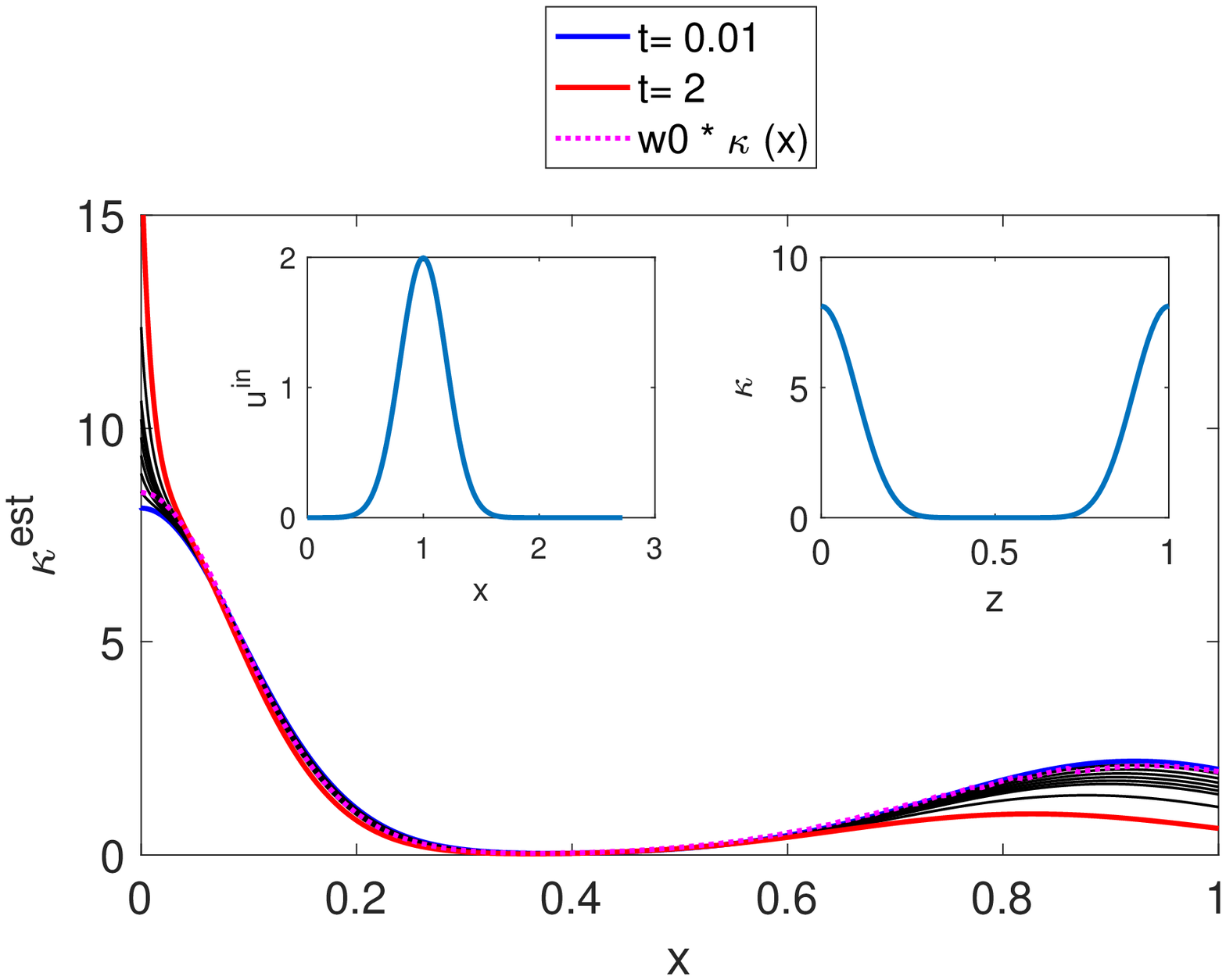}
   \includegraphics[width=0.5\textwidth]{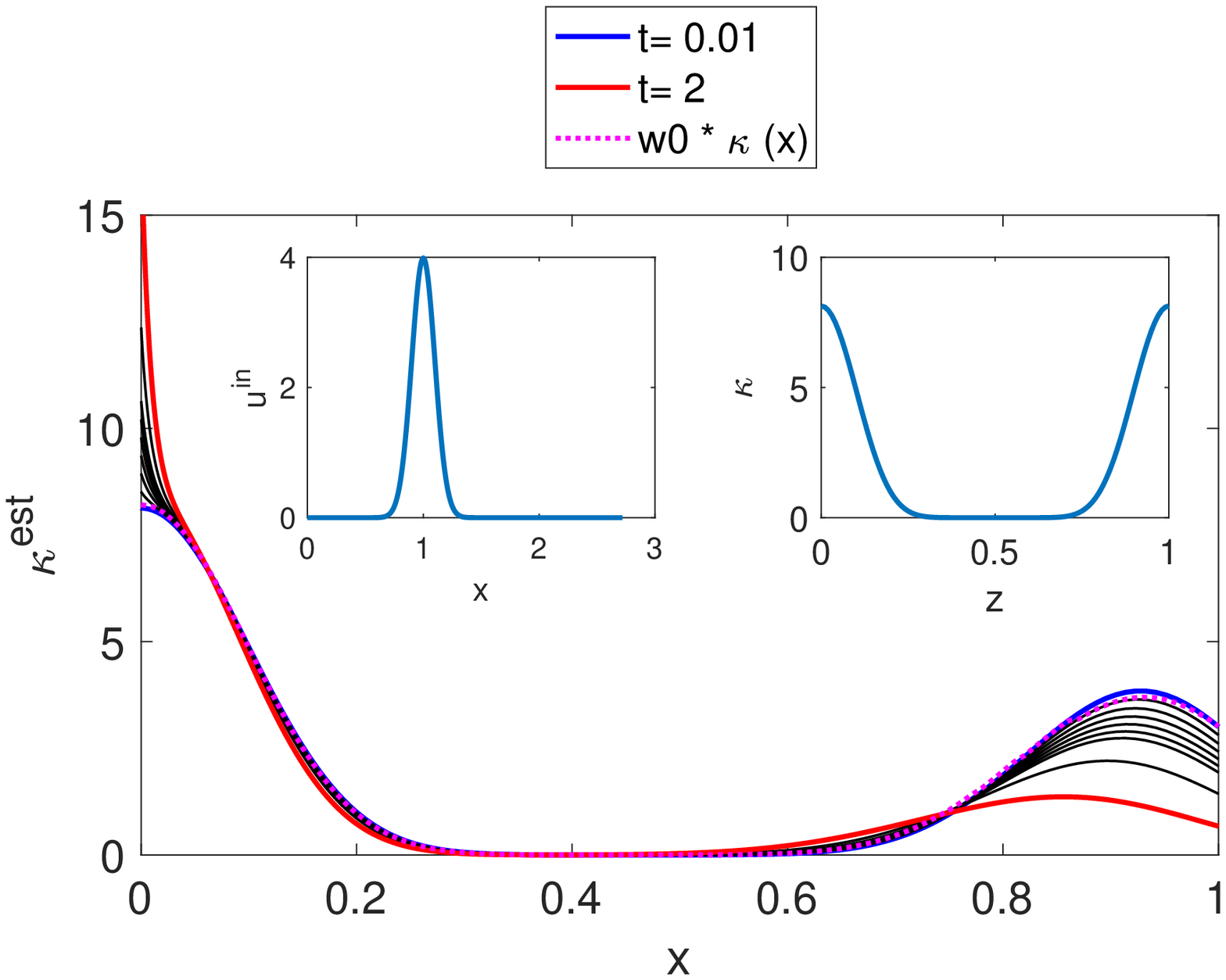}\\
\includegraphics[width=0.5\textwidth]{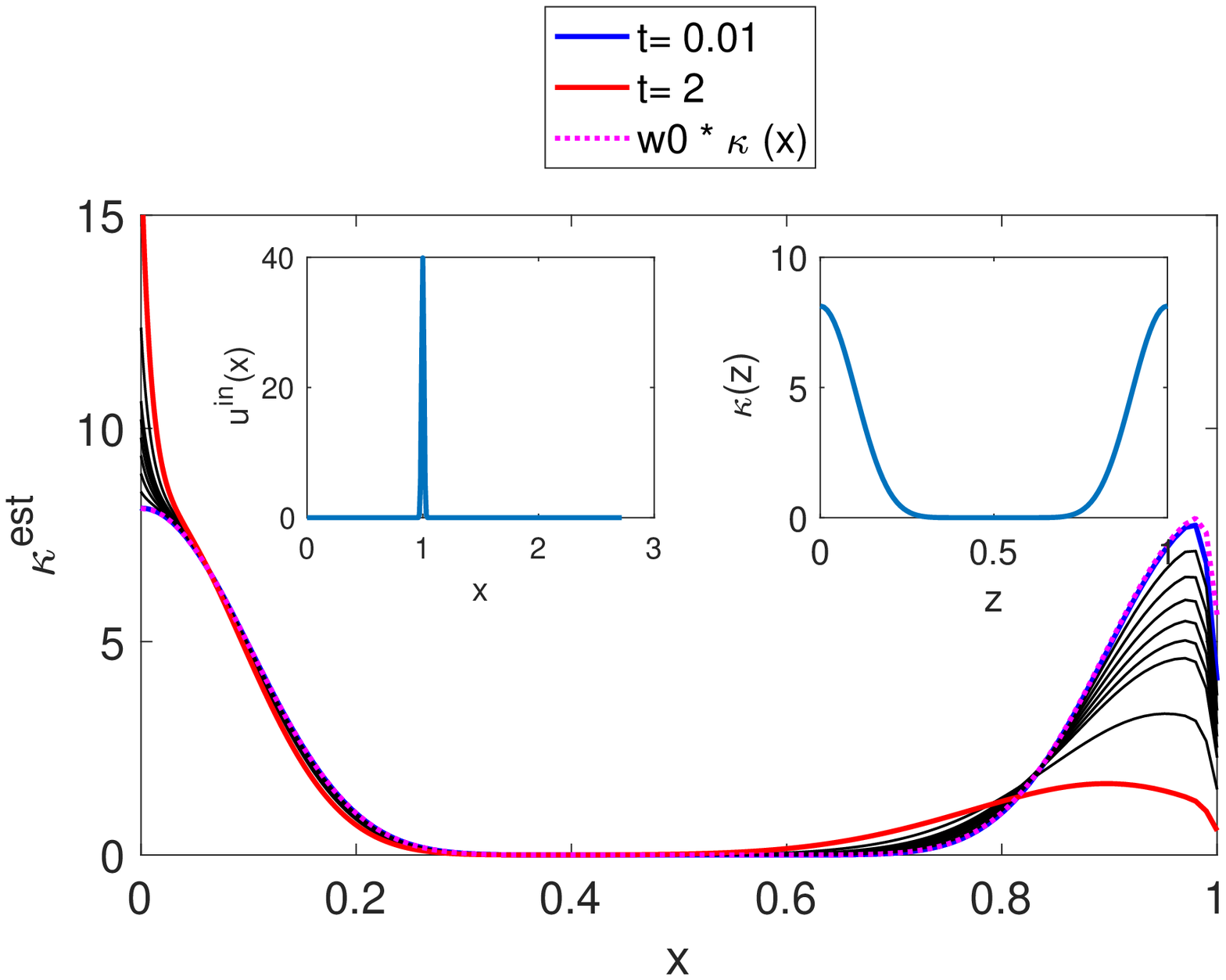}
   \includegraphics[width=0.5\textwidth]{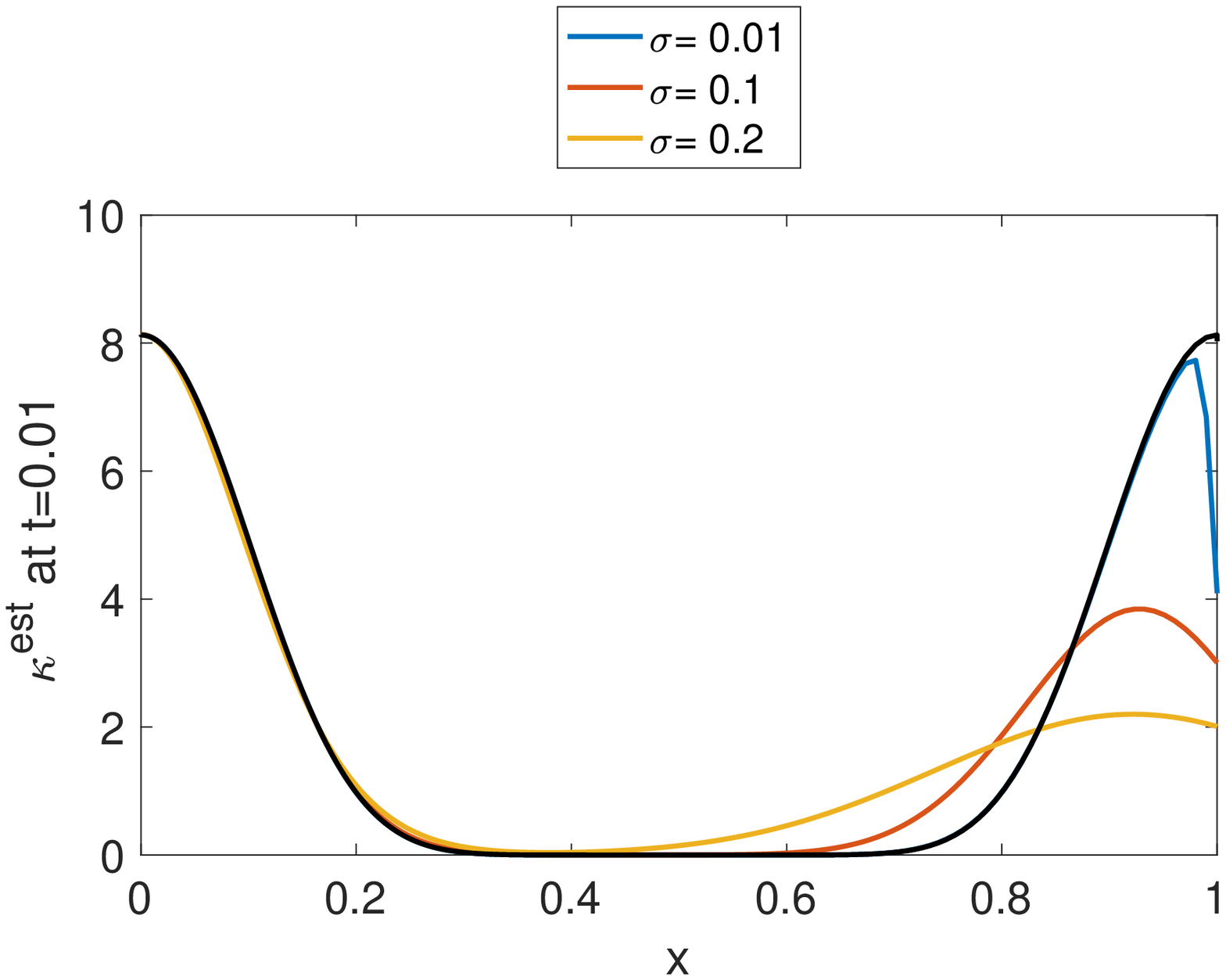}
    \caption{
{\bf Estimation of the fragmentation kernel} \MT{$\kappa^{est}(t)$ at various times $t$ (black curves), first time point in blue, latest time point in red. 
Upper left: for an initial data $\mu_0$ with variance 
$\sigma=0.2$, Upper Right: with variance $\sigma=0.1$, Bottom Left: with variance $\sigma=0.01$. In dotted pink is what is truly 
estimated, namely the convolution $w_0 \ast \kappa$ (see notations of Corollary \ref{cor:generic}). Bottom Right: Superimposition of the three estimates $\kappa^{est} (t)$ at an early time point $t=0.01$ $t=0.01$ corresponding to the three initial conditions of respective variance $\sigma =0.01$, $\sigma=0.1$ and $\sigma=0.2$.
}
}
\label{fig:cor2:time1}
 \end{figure}

   \begin{figure}
\includegraphics[width=0.5\textwidth]{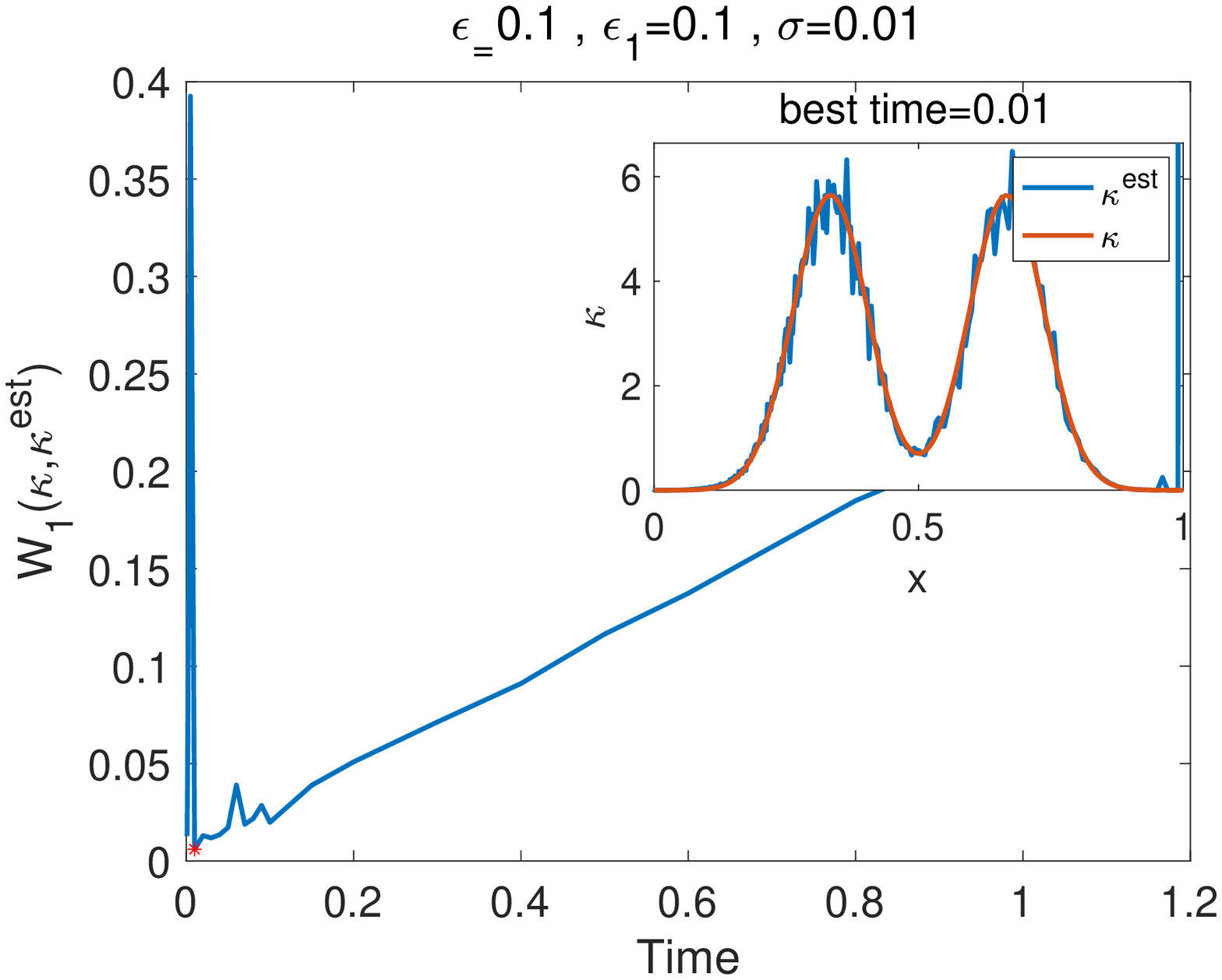}
   \includegraphics[width=0.5\textwidth]{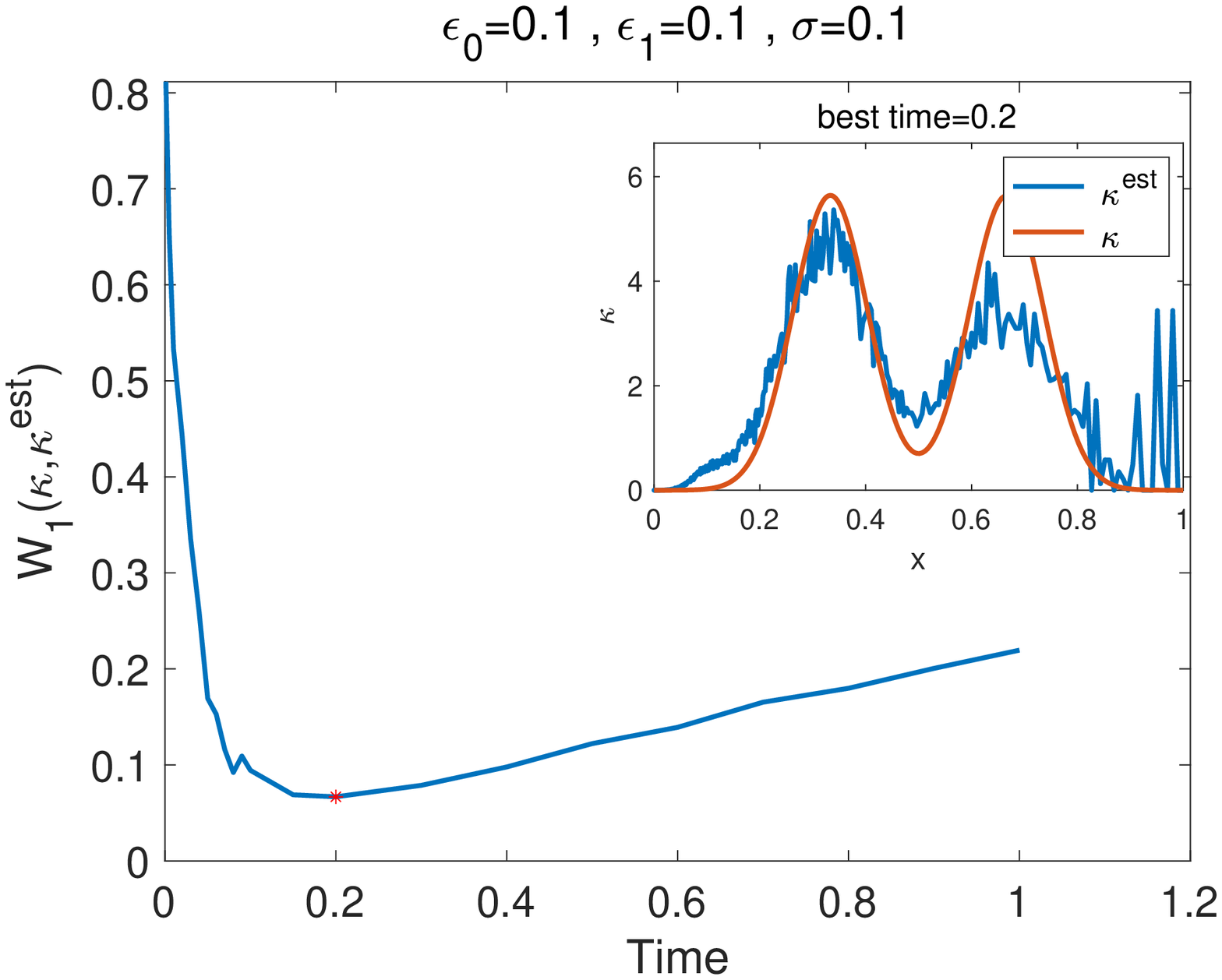}\\
 \caption{\comMD{{\bf Time evolution of the error estimate} \MT{in BL norm (approximated here by the $W_1-$ distance) between} the fragmentation kernel \MT{$\kappa$} and its estimate $\kappa^{est}(t)$, for $\epsilon_0=\epsilon_1=0.1$ and for $\MT{q}=0.01$ (Left), $\MT{q}=0.1$ (Right). The insets display the best estimate \MT{$\kappa^{est}(t_0)$}, obtained at the timepoint $t_0$ where the $W_1$ distance is minimal.}}
   \label{fig:Marie}
 \end{figure}
 \begin{figure}
\begin{center}\includegraphics[width=0.7\textwidth]{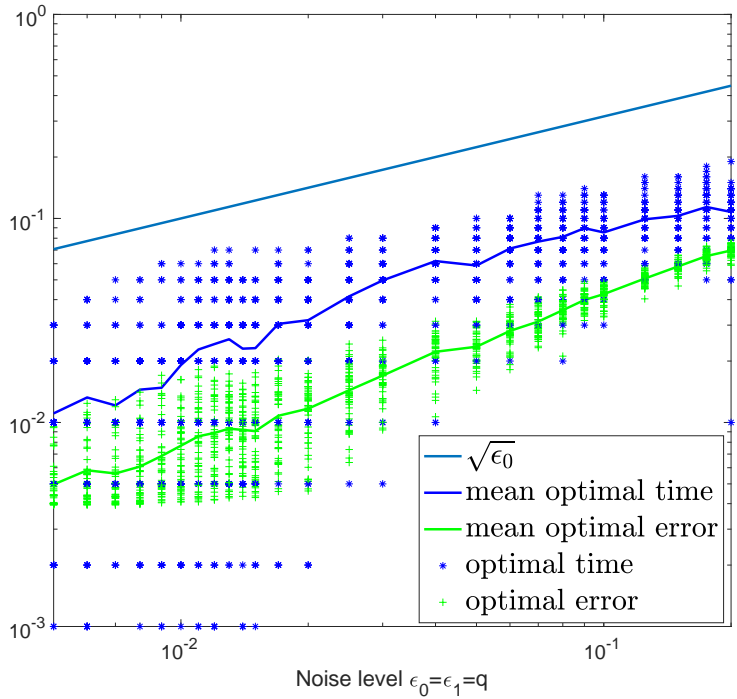} 
\end{center}
\caption{{\bf Optimal error and optimal time with respect to the noise level.} We choose an equal level for the three noise sources $\eps_0,$ $\eps_1$ and $q$ (i.e. $q=\sigma$ the standard deviation of the initial gaussian), and for 50 simulations we draw the optimal time (blue asterisks) giving the optimal error (green asterisks). In plain lines we draw the mean over the 50 simulations, to be compared with the curve $\varepsilon \mapsto \sqrt{\varepsilon}$ (lighter blue plain line).  We observe a good agreement with the expected rate of convergence.}
   \label{fig:Marie2}
   \end{figure}
  
  \begin{table}
 \begin{tabular}{|c||c|c|c|}
 \hline
 Kernel& \includegraphics[width=0.25\textwidth]{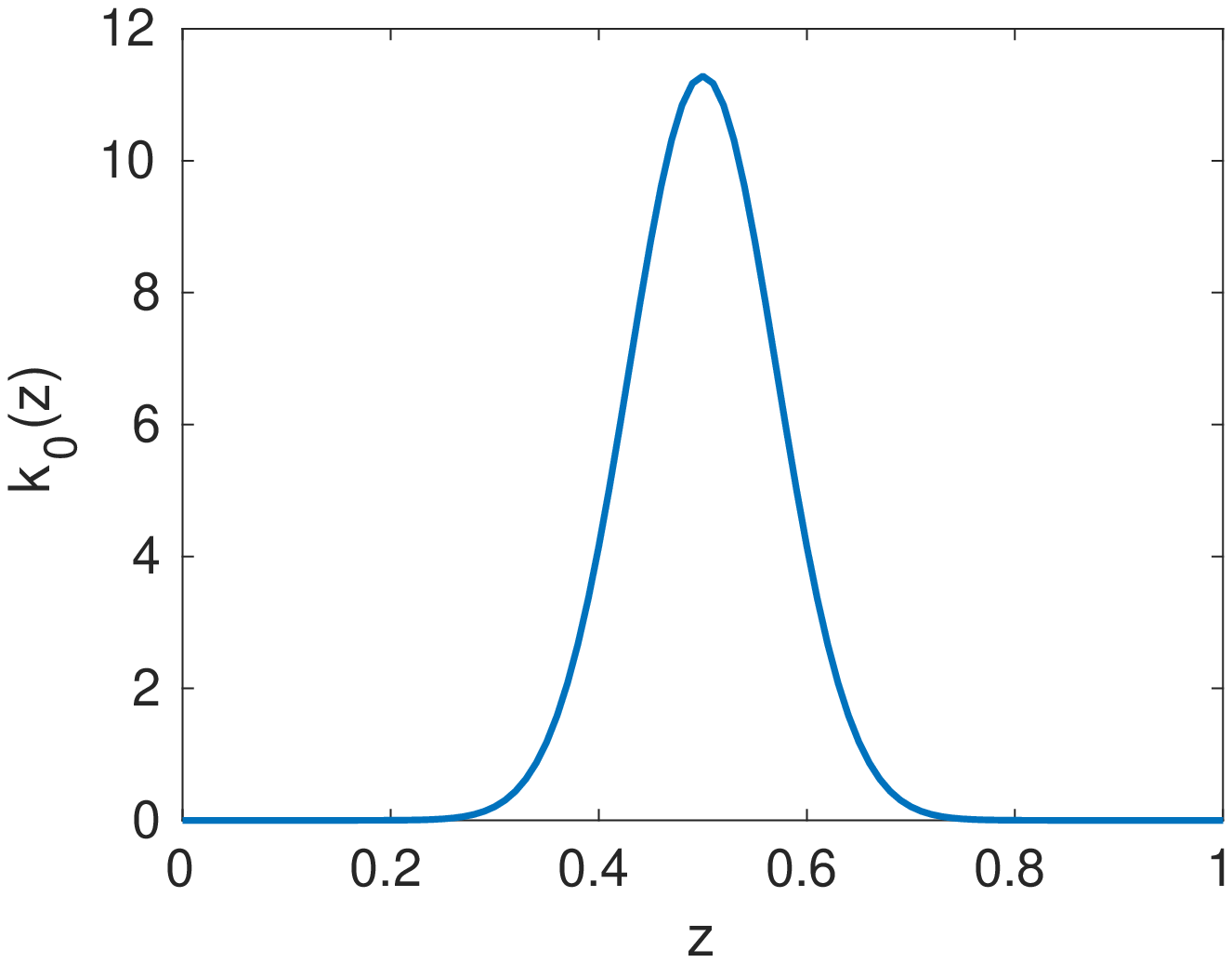} &  \includegraphics[width=0.25\textwidth]{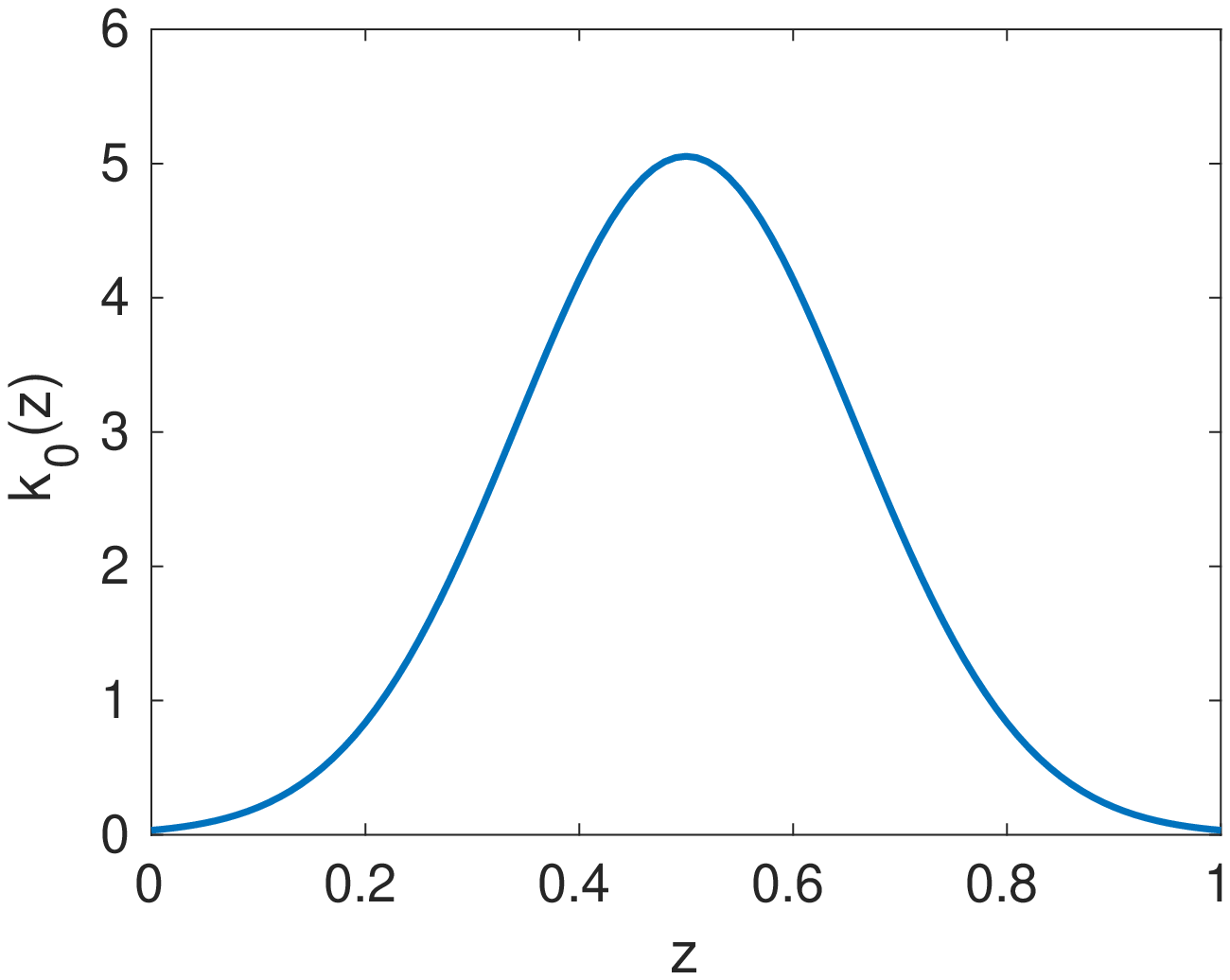} & \includegraphics[width=0.25\textwidth]{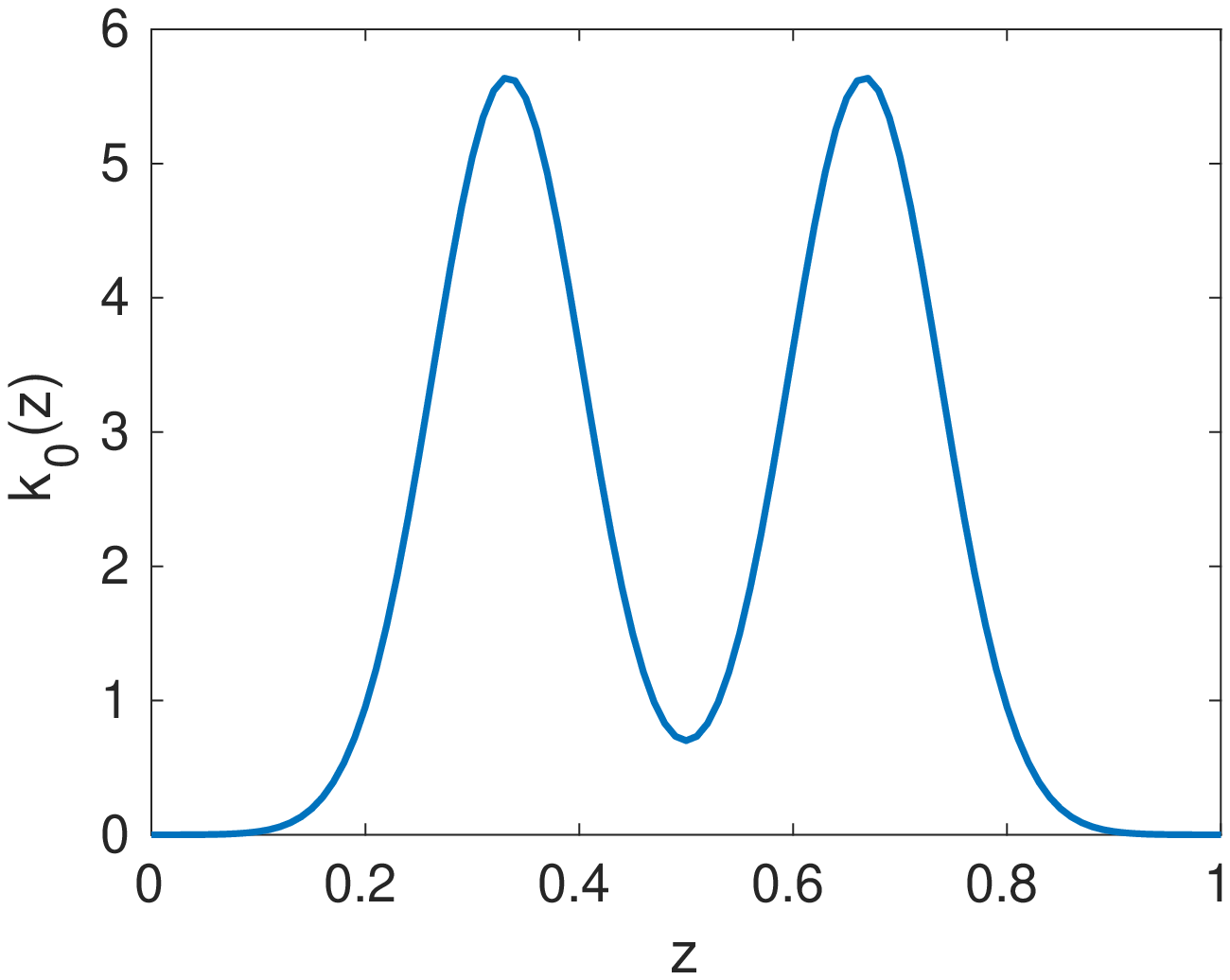}\\
\hline
Var & 0.0100 &0.0295 &0.0378 \\
\hline
SD & 0.1001  &0.1718&0.1944 \\
\hline
\hline
\hline
 Kernel &  \includegraphics[width=0.25\textwidth]{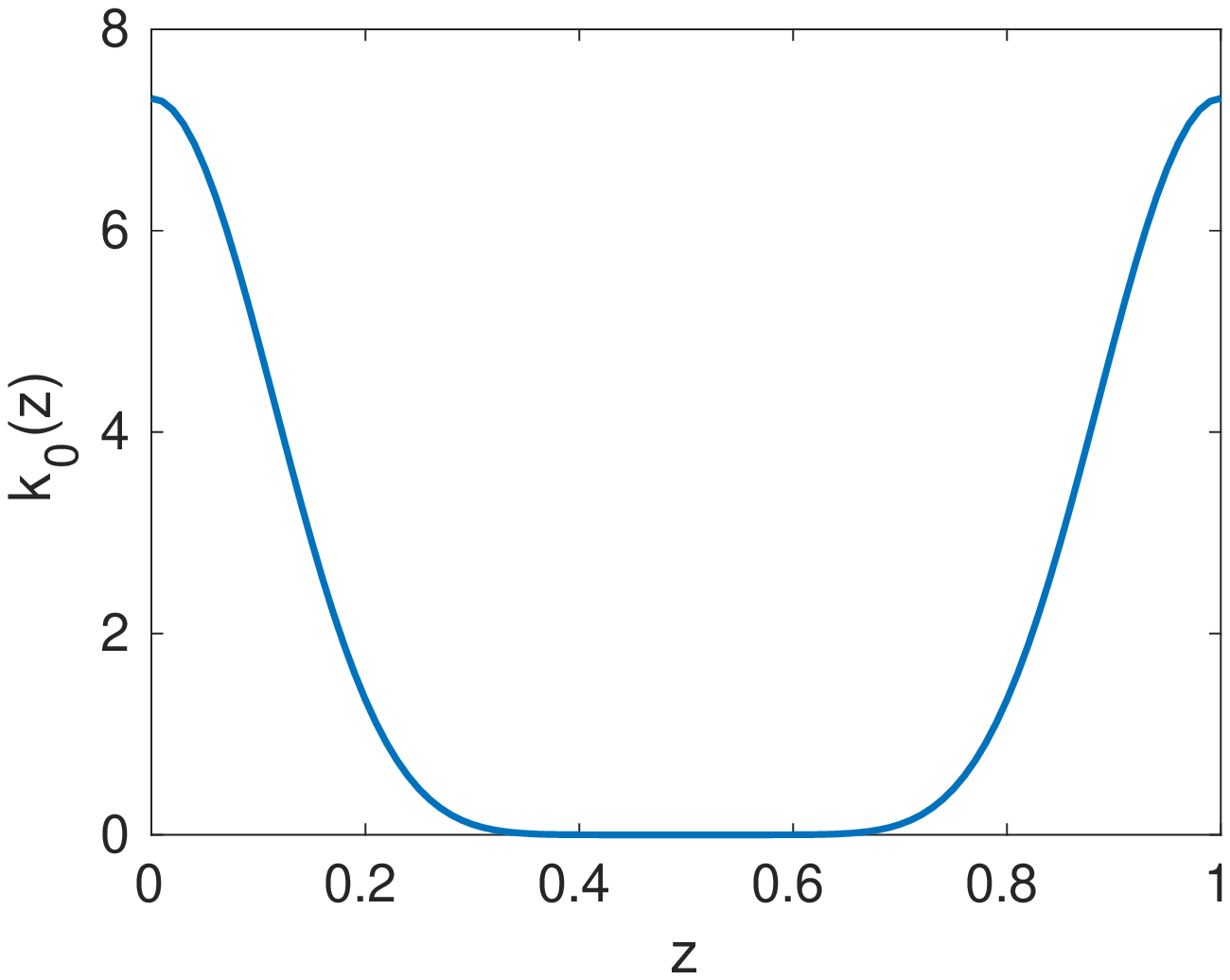} & \includegraphics[width=0.25\textwidth]{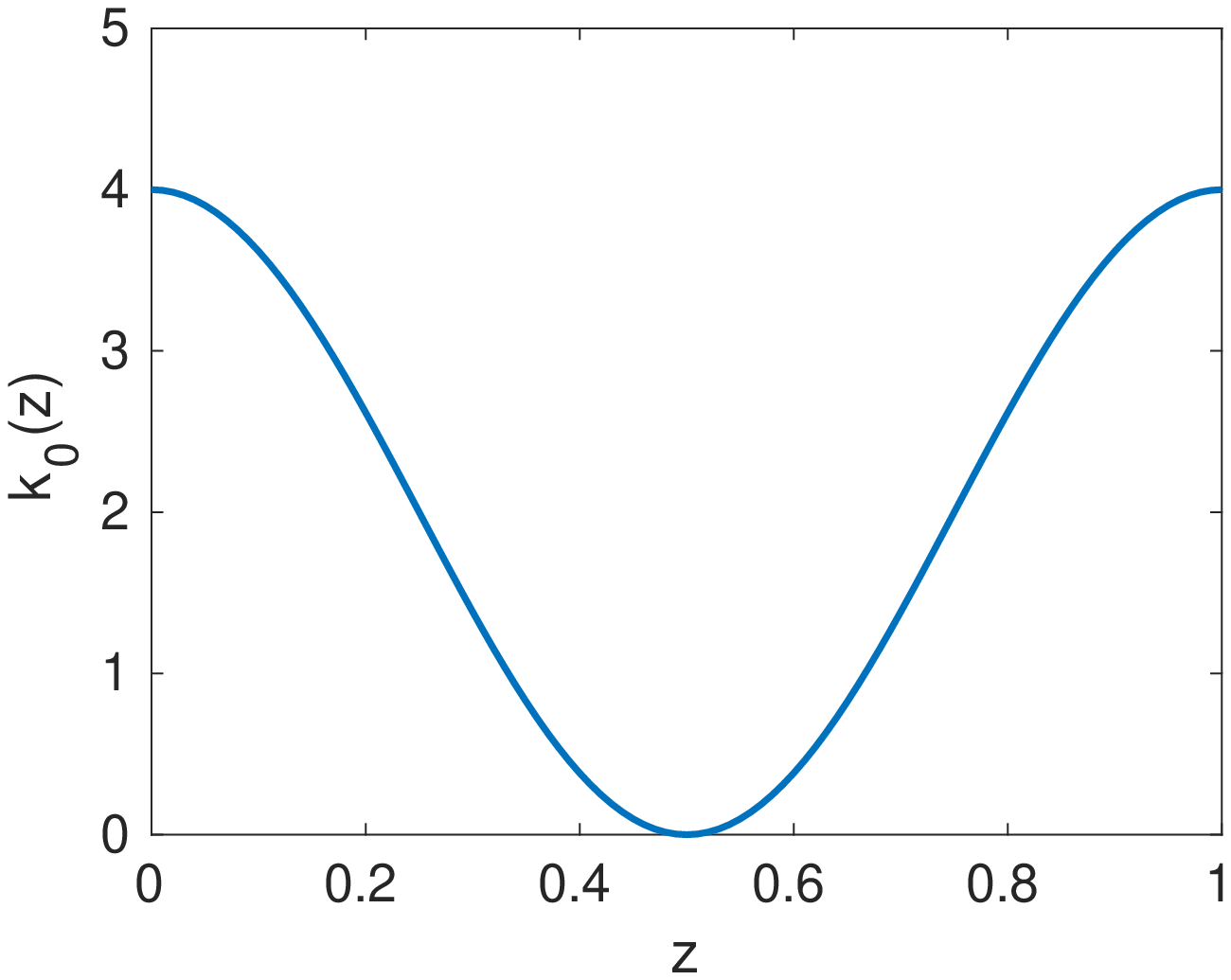}&\includegraphics[width=0.25\textwidth]{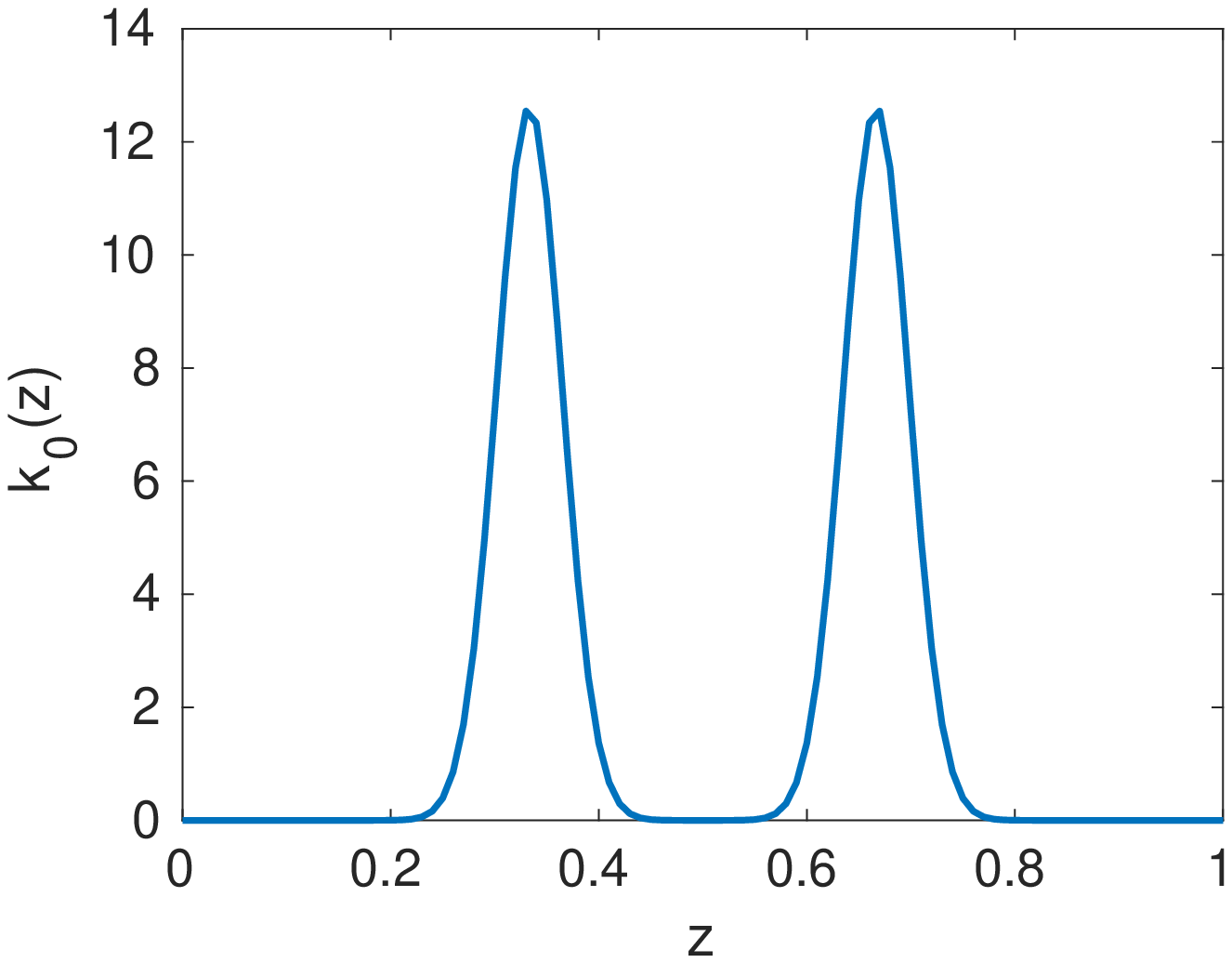}\\
\hline
 Var &0.1625  &0.1290& 0.0338    \\
 \hline
 SD& 0.4031 &0.3591 &0.1839  \\
 \hline
 \end{tabular}
 \caption{Variance and Standard Deviation for 6 typical fragmentation kernels.}
 \label{table}
 \end{table}

\begin{table}[h]
    \centering
    \hfill
   \includegraphics[width=0.6\textwidth]{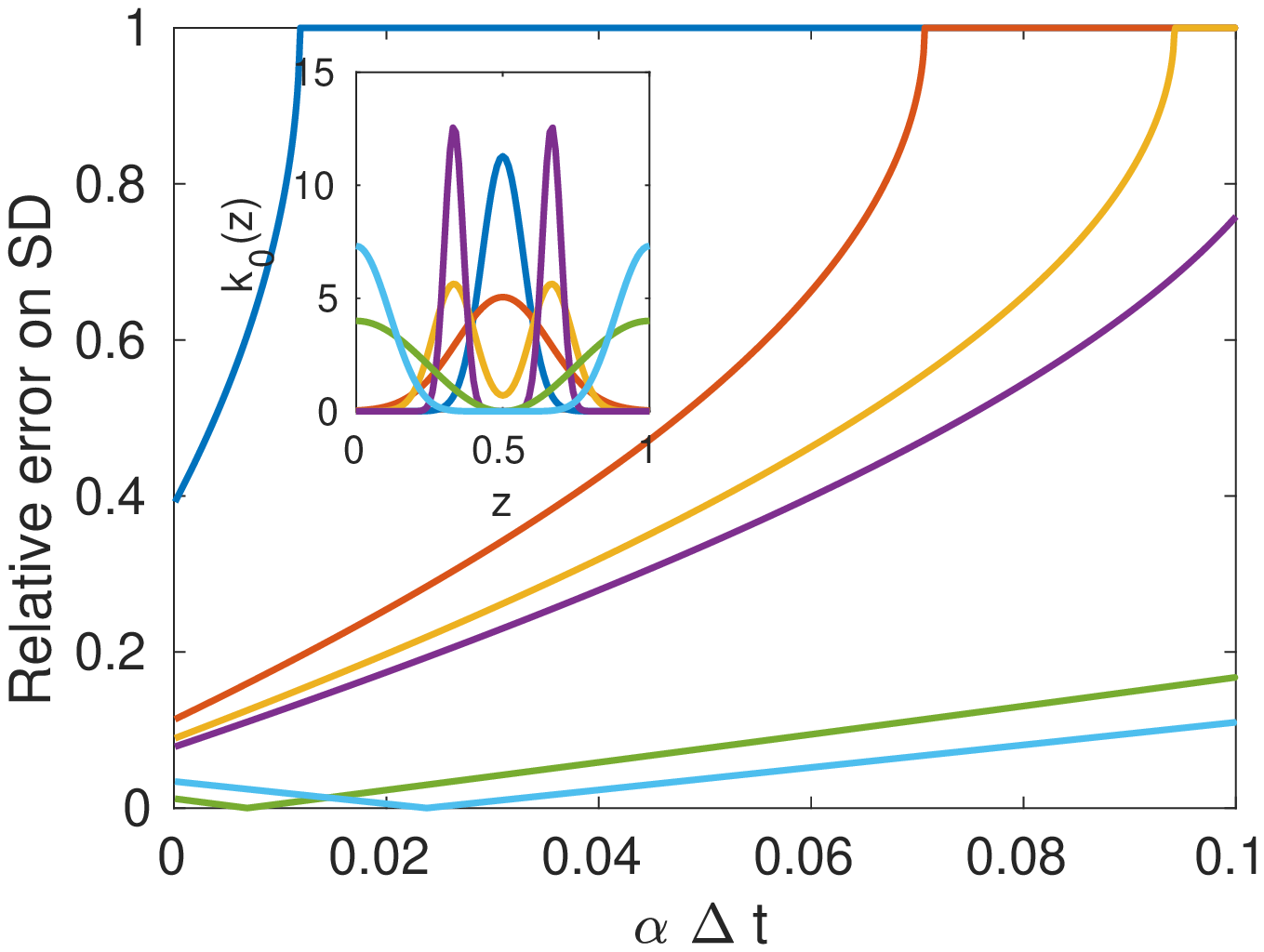} 
    \hfill
 \begin{tabular}{|c|c|||c|c|}
 \hline
 \multicolumn{2}{|c|||}{$\alpha \Delta t =0.02$}&\multicolumn{2}{|c|}{$\alpha \Delta t=0.1$}\\
  \hline
SD& $SD^{est}$  & SD & $SD^{est}$ \\
\hline
0.1001 & 0&0.1001&0 \\
0.1718 & 0.1296&0.1718&0 \\
0.1944 & 0.1574&0.1944&0 \\
0.2065 & 0.1718&0.2065& 0.0537\\
0.3591 & 0.3515&0.3591&0.2997 \\
0.4031 & 0.4060&0.4031 &0.3595\\
 \hline
 \end{tabular}
    \hfill\null
    \caption{\MT{Parameters: $\alpha=\gamma=1$. Initial condition: a thin gaussian centered at $x=2$. For the 6 kernels described in Table \ref{table} and displayed in the inset, we plot the relative error
 on the standard deviation estimate as a function of $\alpha \Delta t$. We detail two cases: $\alpha \Delta t=0.02$ and $\alpha \Delta t=0.1$ on the table (Right)}}
\label{fig:Dt}
\end{table}

\bibliographystyle{plain}
 \bibliography{bibli19}
\end{document}